\documentclass[11pt]{article}
\usepackage[T1]{fontenc}
\usepackage{amsmath,amsfonts,amssymb,theorem,color}
\usepackage{tensor} 

\usepackage[margin=1.5in]{geometry}

\usepackage{imakeidx} 

\usepackage{authblk} 

\usepackage{etoolbox}

\usepackage{hyperref}

\usepackage{mathbbol,MnSymbol}
\theoremstyle{change} 

\newenvironment{proof}{\noindent \bf Proof:   \rm}{\hspace*{\fill}
$\square$ \medskip}

\newtheorem{theorem}{Theorem}
\newtheorem{lemma}[theorem]{Lemma}
\newtheorem{proposition}[theorem]{Proposition}
\newtheorem{corollary}[theorem]{Corollary}
{\theorembodyfont{\upshape}
\newtheorem{definition}{Definition}
\newtheorem{remark}{Remark}
\newtheorem{example}{Example}
}
\DeclareMathAlphabet{\mathpzc}{OT1}{pzc}{m}{it}

\usepackage[all,cmtip]{xy}

\usepackage{appendix}

   \mathchardef\mhyphen="2D

\newcommand{\eq}[1][r]
   {\ar@<-3pt>@{-}[#1]
    \ar@<-1pt>@{}[#1]|<{}="gauche"
    \ar@<+0pt>@{}[#1]|-{}="milieu"
    \ar@<+1pt>@{}[#1]|>{}="droite"
    \ar@/^2pt/@{-}"gauche";"milieu"
    \ar@/_2pt/@{-}"milieu";"droite"}

\usepackage{tikz}
 
\newcommand*\mycirc[1]{%
   \begin{tikzpicture}[baseline=(C.base)]
     \node[draw,circle,inner sep=1pt](C) {#1};
   \end{tikzpicture}}

\begin{document}

\title{Hilbertian Frobenius algebras}


\author{Laurent Poinsot\footnote{Second address: CREA, 
\'Ecole de l'Air, 
Base a\'erienne 701,
13661 Salon-de-Provence,
France.}}         
\affil{LIPN - UMR CNRS 7030\\
University Sorbonne Paris Nord\\
93140 Villetaneuse, France\\\quad\\
laurent.poinsot@lipn.univ-paris13.fr}

\date{}

\maketitle

\begin{abstract}
Commutative Hilbertian Frobenius algebras are those commutative semigroup objects in the monoidal category of Hilbert spaces, for which the Hilbert adjoint of the multiplication satisfies the Frobenius compatibility relation, that is, this adjoint is a bimodule map. In this note we prove that they split as an orthogonal direct sum of two closed ideals, their Jacobson radical which in fact is nothing but their annihilator,  and the closure of the linear span of their group-like elements. As a consequence such an algebra is semisimple if, and only if, its multiplication has a dense range. In particular every commutative special Hilbertian algebra, that is, with a coisometric multiplication,  is semisimple. Extending a known result in the finite-dimensional situation, we prove that the structures of such  Frobenius algebras on a given Hilbert space are in one-one correspondence with its bounded above orthogonal sets. We show, moreover, that the category of commutative Hilbertian Frobenius algebras is dually equivalent to a category of pointed sets. Thus,  each semigroup morphism between commutative Hilbertian Frobenius semigroups arises from a unique base-point preserving map (of some specific kind), from the set of minimal ideals of its codomain to the set of minimal ideals of its domain, both  with zero added. \\

\end{abstract}

\section*{Introduction}

Frobenius algebras have deep connections with various mathematical objects: they appear of course in module theory since each finite-dimensional selfinjective algebra over a field is Morita equivalent to a Frobenius algebra~\cite[Corollary~3.11, p.~351]{Skowronski},    but also  in representation theory because of their similarity with group algebras~\cite{Reiner}, in mathematical physics as their category is equivalent to that of 2-dimensional topological quantum field theories~\cite{Abrams}, in  (categorical) quantum computing since they provide an algebraic characterization of orthonormal bases and observables~\cite{Coecke2007, Coecke} or also in the theory of Hopf algebras in particular because their relation with weak Hopf algebras~\cite{Bohm}. 

As is well-known the Hilbert space $\ell^2(X)$ of square-summable functions on $X$ is by no means free over $X$. However in this note we prove that $X$ -- or more precisely $X$ plus a new point added -- freely generates a (non-unital, when $X$ is not finite) commutative Frobenius algebra, whose underlying space is (unitarily isomorphic to) $\ell^2(X)$. More generally we show that every commutative Frobenius algebra $(H,\mu)$, where $(H,\mu)$ is a semigroup in the category of Hilbert space and for which $\mu^{\dagger}$ satisfies the Frobenius compatibility relation, splits into an orthogonal direct sum of two ideals $\ell^2(X)\oplus_2 A(H,\mu)$, where $A(H,\mu)$ is the annihilator of $(H,\mu)$ and $X$ is a set equipotent to that of minimal ideals of $(H,\mu)$. The free Frobenius algebras, that is, those  of the form $\ell^2(X)$, are precisely the semisimple ones. Before describing in more detail the content of this note, let us put it into perspective.

A Frobenius algebra may be described in several equivalent ways (\cite[Theorem~61.3, p.~414]{Reiner}), for instance it is a unital algebra over some base field which as a left module over itself is isomorphic to its algebraic (right) dual. This definition implies directly that the algebra must be finite-dimensional or at least finitely generated projective if base rings were allowed~\cite{Eilenberg}. 

Not all equivalent characterizations are equally suitable for the applications or generalizations we have in mind, for instance if one expects to extend this notion to not necessarily finite-dimensional algebras. Thus alternatively a Frobenius algebra is a finite-dimensional unital algebra together with a non-degenerate and associative bilinear form.  Dropping the finiteness condition leads to  infinite-dimensional Frobenius algebras studied in~\cite{Jans}.  At this point one may add that under this form the Frobenius algebras are substantially similar to  a class of (non unital) algebras combining Banach algebras and Hilbert spaces, called $H^*$-algebras~\cite{Ambrose}, where operators of right multiplication are the Hilbert adjoints of the operators of left multiplication.  

More recently another way to characterize Frobenius algebras appears  in~\cite[Theorem~1, p.~572]{Abrams} in a commutative situation. A Frobenius algebra  is a finite-dimensional unital algebra which at the same time is a counital coalgebra such that both structures interact nicely. More precisely the compatibility  condition between the  algebra and the  coalgebra of a Frobenius algebra -- the so-called {\em Frobenius relation} -- asserts that the comultiplication of the latter is a morphism of bimodules over the former. 

Because the above compatibility relation is stated entirely using only    tensor products and linear maps, the former description has the advantage over others  to allow for talking about Frobenius algebras in the realm of monoidal categories. This is precisely the point of view adopted in the recent book~\cite{Kock}. This approach is  used with success in~\cite{Coecke,Heunen} where the authors take advantage of the presence of the Hilbertian adjoint to define a comultiplication from a multiplication. More precisely they consider commutative  {\em $\dagger$-Frobenius monoids}, that is, Frobenius algebras $(H,\mu,\eta,\delta,\epsilon)$ over a finite-dimensional Hilbert space $H$, where $\delta=\mu^{\dagger}$ and $\epsilon=\eta^{\dagger}$.  

In what follows ``{\em Hilbertian Frobenius}''  stands for  ``{\em $\dagger$-Frobenius}'',  to recall the fact that the comultiplication is the Hilbert adjoint of the multiplication, because no other kinds of Frobenius semigroups in the category of Hilbert spaces are considered here. To summarize the coalgebra structure of a Hilbertian Frobenius algebra is the Hilbert adjoint of its algebra structure.  

Most notably the main result in~\cite{Coecke} is the statement that  orthogonal bases on a given finite-dimensional Hilbert space and  its structures of commutative (unital and counital) Hilbertian Frobenius algebras are in a one-one correspondence.   


In an effort to extend this result to arbitrary Hilbert spaces, a notion of non-unital Frobenius algebra is proposed in~\cite{Heunen}, referred to as (commutative) {\em Hilbertian Frobenius semigroups} (or {\em algebras}) in what follows, obtained by dropping the unital and counital assumptions. More precisely, these are Hilbertian Frobenius semigroups $(H,\mu)$  (in the monoidal category $(\mathbf{Hilb},\hat{\otimes}_2,\mathbb{C})$ of Hilbert spaces and bounded linear maps), that is, $H\hat{\otimes}_2 H\xrightarrow{\mu}H$ is commutative and associative, and its adjoint $H\xrightarrow{\mu^{\dagger}}H\hat{\otimes}_2H$ satisfies the Frobenius condition. 

While the authors only partially achieve their goal,  that is, the characterization of arbitrary orthonormal bases by means of  Frobenius structures, one of their merit is a clarification of the relation  between commutative {\em special}  Hilbertian Frobenius semigroups, that is, those Hilbertian Frobenius semigroups $(H,\mu)$  with an isometric comultiplication ($\mu\circ\mu^{\dagger}=id$),  and Ambrose's $H^*$-algebras, relation which was further analyzed in~\cite{Poinsot2019} for {\em special Hilbertian} algebras which are not necessarily of the Frobenius kind. 

It is precisely the intention of this note to provide a better understanding of the relations between orthogonal bases or better orthogonal sets, and algebraic structures of commutative Hilbertian Frobenius semigroups  over not necessarily finite-dimensional Hilbert spaces.  

We, hence, provide a structure theorem for commutative Hilbertian Frobenius semigroups (Theorem~\ref{thm:stHFs}): it is an easy fact that they split as an orthogononal direct sum of their Jacobson radical and the topological closure of the linear span of their {\em group-like} elements, that is, those non zero elements $x$ sent to $x\otimes x$ by the comultiplication $\mu^{\dagger}$. But what is less immediate is that in fact, the orthogonal complement of the Jacobson radical is a subalgebra, that is, that the set of group-like elements is closed under multiplication. In fact it is not difficult to notice that the product of two {\em distinct} group-like elements is equal to zero, but what is not as immediate is that the square of a group-like element belongs to the one-dimensional space  spanned by the element (to be fair this observation is free when the comultiplication is assumed isometric as in~\cite{Heunen}, which is not the case in what follows), which by the way turns this space into a minimal ideal.

A similar result is claimed in~\cite{Heunen} for the particular case of commutative Hilbertian Frobenius semigroups with an isometric comultiplication, but its proof  is not completely correct (see the introduction of~\cite{Poinsot2019} for more details). Furthermore our structure theorem is completed by the observation that the radical is precisely the annihilator of the algebra, which is a direct consequence unnoticed elsewhere, of the Frobenius condition  (see Proposition~\ref{prop:normality}) which forces the multiplication operators to be normal, that is, that they commute with their own adjoint.  

Thus, clearly, a commutative Hilbertian Frobenius semigroup $(H,\mu)$ not only splits into an orthogonal direct sum of a subalgebra and an ideal, but actually of two (closed) ideals,  one being radical while the other is semisimple. As becomes clear the question of semisimplicity of such an algebra is completely governed by this structure theorem:  a necessary and sufficient condition for a commutative Hilbertian Frobenius algebra to be semisimple is that its regular representation is faithful, or equivalently that its multiplication $H\hat{\otimes}_2 H\xrightarrow{\mu} H$  has a dense range or using the Hilbert adjoint,  that its comultiplication is one-to-one (Theorem~\ref{thm:semisimplicityforFrob}). (In particular every commutative {\em special} Hilbertian Frobenius algebra is semisimple.) It is a remarkable fact that in the finite-dimensional situation this may be interpreted as the existence of a unit (Corollary~\ref{cor:semisimplicityforFrob1}). Consequently is recovered (with a proof free of a $C^*$-argument) the result of~\cite{Coecke}  that finite-dimensional commutative Hilbertian Frobenius {\em monoids} are semisimple.

By the way, because group-like elements of a commutative Hilbertian Frobenius semigroup, are non-zero and pairwise orthogonal (even orthonormal when furthermore the algebra is special, that is, when the comultiplication is an isometry) several bijective correspondences (Theorem~\ref{thm:dictionary}) between structures of Frobenius semigroups of certain kinds available on a given Hilbert space, and some of its orthogonal (or orthonormal) sets are obtained, which extend the main result of~\cite{Coecke}. It is worth mentioning  that contrary to the finite-dimensional situation, not all orthogonal sets correspond to a  structure of a commutative Hilbertian Frobenius semigroup but only those which are bounded above (or below by a strictly positive constant), including the empty set; in fact one cannot expect  unbounded orthogonal families to be in the range of the above bijections as boundedness of the norm of the group-like elements, is a direct consequence of the fact that in a Banach algebra, the multiplication as a bilinear map, is jointly continuous. 

Besides the above structure theorem also has some important consequences at the level of the category ${}_c\mathbf{FrobSem}(\mathbb{Hilb})$ of commutative Hilbertian Frobenius semigroups  and semigroup morphisms. Most notably it is shown that every semigroup morphism between commutative Hilbertian Frobenius semigroups arises from a unique set-theoretic base-point preserving map (of some specific kind), from the set of minimal ideals of its codomain to the set of minimal ideals of its domain, both with zero added as base-point. Among others one proves that 
\begin{enumerate}
\item\label{it:pt1} ${}_c\mathbf{FrobSem}(\mathbb{Hilb})$ is equivalent to the product ${}_{\mathsf{semisimple},c}\mathbf{FrobSem}(\mathbb{Hilb})\times \mathbf{Hilb}$ (Proposition~\ref{prop:structuretheoremrevisited1}) where the first factor is  the  full subcategory  of ${}_c\mathbf{FrobSem}(\mathbb{Hilb})$ spanned by the semisimple algebras. The splitting into a semisimple Hilbertian Frobenius semigroup and the radical (which is essentially a Hilbert space as its multiplication is trivial) provides the equivalence. 

\item\label{it:pt2} ${}_{\mathsf{semisimple},c}\mathbf{FrobSem}(\mathbb{Hilb})$ is dually equivalent to a category of pointed weighted sets (Theorem~\ref{thm:equivalence_weighted_sets_and_semisimple_Frob}). The functor whose object component sends a semisimple commutative Hilbertian Frobenius semigroup to its set of minimal ideals (with the trivial ideal added),  implements the equivalence. Its equivalence inverse is similar to the $\ell^2_\bullet$-functor introduced in~\cite{Poinsot2019}.
\end{enumerate}

The paper is organized as follows: 

Section~\ref{sec:preliminaries} is mainly devoted to the introduction of the  terminology, definitions, names of categories and notions used in this note. Here  in particular are recalled among others the notions of group-like elements and of  Hilbertian Frobenius semigroups. 

The short Section~\ref{sec:weighted_hilb} provides an example of a  commutative Hilbertian Frobenius semigroup which is of great interest since in fact the orthogonal complement of the Jacobson radical of any commutative Hilbertian Frobenius semigroup is unitarily isomorphic to such an algebra (Proposition~\ref{prop:onlyoneexample}). 

Section~\ref{sec:structuretheoremgeneral} may be seen as a continuation of~\cite{Poinsot2019} and {\em does not} concern Frobenius algebras. We provide a structure theorem (Theorem~\ref{thm:structthm4specialHilbAlg}) for  commutative Hilbertian algebras with an isometric comultiplication which are not necessarily of the Frobenius kind. 

Section~\ref{sec:comHilbFromSemigroups} discusses the case of commutative Hilbertian Frobenius semigroups and in particular the conditions under which they are semisimple. In Section~\ref{sec:frobalgrevisited} is stated the  structure theorem for commutative Hilbertian Frobenius semigroups (Theorem~\ref{thm:stHFs}) and the resulting conditions for semisimplicity (Theorem~\ref{thm:semisimplicityforFrob}).

In Section~\ref{sec:directcons} are explored some of the immediate consequences of the structure theorem. Thus here is provided Theorem~\ref{thm:dictionary}  about the bijective correspondences between bounded orthogonal sets and some structures of Hilbertian Frobenius algebras on a Hilbert space, Proposition~\ref{prop:hstar} about the logical  equivalence between commutative $H^*$-algebras and commutative Hilbertian Frobenius semigroups  and a reduction of semisimplicity to the existence of   approximate (co)units (Corollary~\ref{cor:semisimplicityforFrob}).

Section~\ref{sec:reformulation} concerns the reformulation of the structure theorem as an equivalence of categories (item~\ref{it:pt1} above).   

Section~\ref{sec:mainequiv} deals with the equivalence of categories from point~\ref{it:pt2} above.

Section~\ref{sec:otherequiv} treats the case of other equivalences obtained by considering other kinds of  morphisms, for instance  it deals with ambidextrous morphisms (Section~\ref{sec:ambi}), that is, maps which are simultaneously morphisms of algebras and coalgebras, and with proper morphisms (Section~\ref{sec:proper}). 

Section~\ref{sec:epilogue} discusses an example of a non-commutative non-semisimple Hilbertian Frobenius semigroup whose merit is to show that the Frobenius condition does not govern anymore semisimplicity in the non-commutative situation. 

\section{Preliminaries}\label{sec:preliminaries}

Let us begin with some words about the terminology used in this note:
\begin{enumerate}
\item Every vector space is over the field of complex numbers, and unless stated otherwise, all algebras are over $\mathbb{C}$, associative and  {\em commutative}, but not necessarily unital, so one usually refers to them solely as {\em algebras}. Consequently in this note one usually -- but not always -- drops the adjective {\em commutative} and, when commutativity is not assumed, we write it explicitly. An algebra map or morphism of algebras is not required to preserve a unit.
\item  Let $H$ be a Hilbert space.  Let $X\subseteq H$ be a set of pairwise orthogonal vectors, that is, $\forall x,y\in H,\ x\not=y\Rightarrow \langle x,y\rangle=0$. $X$ is an {\em orthogonal set} when moreover $0\not\in X$, that is, an orthogonal set is a set of pairwise {\em non-zero} orthogonal vectors. It follows that every orthogonal set in linearly independent, and that $\emptyset$ is an orthonormal set, while $(0)$ is not. An {\em orthogonal basis} is an orthogonal set $X$ such that $X^{\perp}=0$. 

\item Recall from the Introduction that in this note the only kind of Frobenius algebras we deal with are those for which the coalgebra structure is adjoint in the Hilbert sense of the algebra structure. Rather than calling them ``{\em $\dagger$-Frobenius}'' one says ``{\em Hilbertian Frobenius}''. 
\end{enumerate}
 
Let us now summarize some notations and results from~\cite{Poinsot2019} as far as they are needed hereafter. 

\subsection{Hilbert spaces}
When $E$ is Banach space, $\mathcal{B}(E)$ stands for the Banach space of all bounded linear endomorphism of $E$ with the operator norm $\|-\|_{\mathsf{op}}$. The obvious forgetful functor from Hilbert spaces to Banach spaces, with   bounded linear maps as morphisms for both categories, is denoted $U$ but there is no risk to identify -- as we shall do hereafter -- a  Hilbert space $H$ with its underlying Banach space as $U$ is injective on objects.\footnote{Let $H,K$ be Hilbert spaces such that $U(H)=U(K)$, then as complex vector spaces $H=K$. The parallelogram law implies that the norms of $U(H)$ and $U(K)$ comes from a common inner product, and thus $H=K$ as Hilbert spaces.} The inner  product (linear in its first variable) of a Hilbert space $H$ is denoted $\langle\cdot,\cdot \rangle_H$ or simply $\langle \cdot,\cdot\rangle$. Basic properties about the Hilbertian tensor product $\hat{\otimes}_2$ are provided in~\cite{Kadison}. Given a bounded linear map $\mu\colon H\hat{\otimes}_2 K\to L$, where $H, K, L$ are Hilbert spaces, $\mu_{\mathsf{bil}}\colon H\times K\to L$ denotes its (unique) associated bounded bilinear map. For a bounded multilinear or linear map $f$, $\|f\|_{\mathsf{op}}$ stands for its usual {\em operator norm}. The {\em Hilbert direct sum} (or {\em orthogonal direct sum}) of Hilbert spaces $H,K$ is denoted $H\oplus_2 K$. Finally given a closed subspace $V$ of $H$, $p_V\colon H\to H$ denotes the orthogonal projection onto $V$, that is, $p_V=i_V\circ \pi_V$, where $\pi_V\colon H\to V$ is the canonical projection $H\simeq V\times V^{\perp}\to V$ and $i_V\colon V\hookrightarrow H$ is the canonical inclusion.

\begin{lemma}\label{lem:restriction}
Let $H,K$ be Hilbert spaces and $V,W$ be  closed subspaces of respectively $H$ and $K$. Let $H\xrightarrow{f}K$ be a bounded linear map. $f(V)\subseteq W$ if, and only if, $f^{\dagger}(W^{\perp})\subseteq V^{\perp}$.
\end{lemma}
%

\subsection{Categories}
Let $\mathbb{Hilb}:=(\mathbf{Hilb},\hat{\otimes}_2,\mathbb{C})$ be the symmetric monoidal category of complex Hilbert spaces and bounded linear maps together with the usual Hilbertian tensor product. The associaitivity constraint $\alpha_{H,K,L}\colon (H\hat{\otimes}_2K)\hat{\otimes}_2 L\simeq H\hat{\otimes}_2(K\hat{\otimes}_2L)$ and the symmetry constraint $\sigma_{H,K}\colon H\hat{\otimes}_2K\simeq K\hat{\otimes}_2 H$ are unitary transformations. Let $\mathbb{FdHilb}:=(\mathbf{FdHilb},\otimes_2,\mathbb{C})$ be its monoidal subcategory of finite-dimensional Hilbert spaces and necessarily bounded linear maps. $H\otimes_2 K$ thus is just the finite-dimensional vector space $H\otimes_\mathbb{C}K$ together with the inner product $\langle u\otimes v,u'\otimes v'\rangle=\langle u,u'\rangle_H\langle v,v'\rangle_K$, $u,u'\in H$, $v,v'\in K$. 

A semigroup object $(H,\mu)$ in $\mathbb{Hilb}$ has an {\em underlying} Banach algebra $(H,\mu_{\mathsf{bil}})$. Note that $\|\mu_{\mathsf{bil}}(x,y)\|\leq M\|x\|\|y\|$ for some constant $M$, where $\|-\|$ is the norm induced by the inner product on $H$. So it may happen that strictly speaking, $(H,\mu_{\mathsf{bil}})$ is not a Banach algebra (i.e., $\|-\|$ is not submultiplicative). However $\|x\|':=\max\{\, 1,\|\mu_{\mathsf{bil}}\|_{\mathsf{op}}\,\}\|x\|$ defines a submultiplicative norm, that is, $\|\mu_{\mathsf{bil}}(x,y)\|'\leq \|x\|'\|y\|'$,  equivalent to $\|-\|$. In other words $((H,\|-\|'),\mu_{\mathsf{bil}})$ is a usual Banach algebra, which is the underlying Banach algebra of $(H,\mu)$.

An {\em ideal} $I$ of $(H,\mu)$ is defined as an ideal of the underlying Banach algebra $(H,\mu_{\mathsf{bil}})$. When $I$ is closed, this is equivalent to the requirement that $\mu((I^{\perp}\hat{\otimes}_2 I^{\perp})^{\perp})\subseteq I$. Note that $(I^{\perp}\hat{\otimes}_2 I^{\perp})^{\perp}=(I\hat{\otimes}_2 I^{\perp})\oplus_2 (I\hat{\otimes}_2 I)\oplus_2 (I^{\perp}\hat{\otimes}_2 I)$ (see \cite[Lemma~12, p.~13]{Poinsot2019}).


By  $\mathbf{Sem}(\mathbb{C})$, ${}_c\mathbf{Sem}(\mathbb{C})$,  and ${}_{coc}\mathbf{Cosem}(\mathbb{C})$  are meant respectively the categories of  semigroups, commutative semigroups, and cocommutative  cosemigroups in a symmetric monoidal category $\mathbb{C}$. 
One observes that ${}_c\mathbf{Sem}(\mathbb{FdHilb})$ is nothing but the full subcategory of ${}_c\mathbf{Sem}(\mathbb{Hilb})$ spanned by those commutative semigroups whose underlying vector space is finite-dimensional. A {\em Hilbertian (co)algebra} (or {\em (co)semigroup}) is an object of ${}_c\mathbf{Sem}(\mathbb{Hilb})$ (resp.,  ${}_{coc}\mathbf{Cosem}(\mathbb{Hilb})$) while by {\em special} is meant a Hilbertian (co)algebra $(H,\mu)$ (resp. $(H,\delta)$) with a coisometric multiplication (resp. isometric comultiplication), that is, $\mu\circ \mu^{\dagger}=id$ (resp. $\delta^{\dagger}\circ\delta=id$). ${}_{c}^{\dagger}\mathbf{Sem}(\mathbb{Hilb})$ is the full subcategory of ${}_c\mathbf{Sem}(\mathbb{Hilb})$  spanned by the special Hilbertian algebras. 

The dagger functor $\mathbf{Hilb}^{\mathsf{op}}\xrightarrow{(-)^{\dagger}}\mathbf{Hilb}$ lifts to an isomorphism from the category  ${}_{coc}\mathbf{Cosem}(\mathbb{Hilb})^{\mathsf{op}}$ to ${}_{c}\mathbf{Sem}(\mathbb{Hilb})$. This is still true after the substitution of $\mathbb{Hilb}$ by $\mathbb{FdHilb}$.

By a {\em subalgebra} $V$ of a Hilbertian algebra $(H,\mu)$ is meant a {\em closed} subspace $V$ of $H$ such that $\mu(V\hat{\otimes}_2 V)\subseteq V$, where here $V\hat{\otimes}_2 V$ is identified with the range of $V\hat{\otimes}_2 V\xrightarrow{i_V\hat{\otimes}_2 i_V}H\hat{\otimes}_2 H$. $(V,\mu_{|_V})$ is a Hilbertian algebra on its own right with multiplication $V\hat{\otimes}_2 V\xrightarrow{\mu_{|_V}}V$ the restriction and co-restriction of $H\hat{\otimes}_2H\xrightarrow{\mu}H$. 

By a {\em subcoalgebra} $V$ of $(H,\mu)$ is meant a {\em closed} subspace  $V$ of $H$ such that $\mu^{\dagger}(V)\subseteq V\hat{\otimes}_2 V$. (Alternatively $V$ could be called a {\em subcoalgebra of $(H,\mu^{\dagger})$}.) $(V,(\mu^{\dagger})_{|_V})$ is a Hilbertian coalgebra on its own right with comultiplication $V\xrightarrow{(\mu^{\dagger})_{|_V}}V\hat{\otimes}_2 V$ the restriction and co-restriction of $H\hat{\otimes}_2 H\xrightarrow{\mu^{\dagger}}H$.

When $(H,\mu)$ is a Hilbertian algebra, $xy$ stands for $\mu(x\otimes y)$ and $x^2$ for $\mu(x\otimes x)$, $x,y\in H$.

\begin{lemma}\label{lem:semigroupiso}
Let $(H,\mu),(K,\gamma)$ be Hilbertian algebras. Let $f\colon H\to K$ be a bounded linear map.
\begin{enumerate}
\item For all $u,v\in H$, $f(uv)=f(u)f(v)$ if, and only if,  $f\colon (H,\mu)\to (K,\gamma)$ is a semigroup morphism if, and only if, $f^{\dagger}\colon (K\gamma^{\dagger})\to (H,\mu^{\dagger})$ is a cosemigroup morphism.
\item $f\colon (H,\mu)\to (K,\gamma)$ is a semigroup isomorphism if, and only if, $f$ is both a semigroup morphism and a bijection.
\item $f\colon (H,\mu^{\dagger})\to (K,\gamma^{\dagger})$ is a cosemigroup isomorphism if, and only if, $f$ is both a cosemigroup morphism and a bijection.
\end{enumerate} 
\end{lemma}
\begin{proof}
\begin{enumerate}
\item The first equivalence is due to the fact that $H\otimes_{\mathbb{C}}H$ is dense into $H\hat{\otimes}_2 H$. The second equivalence is immediate.
\item The direct implication is immediate. Let us prove the converse. By the Open Mapping Theorem, since $f\colon H\to K$ is bounded and is bijective, $f\colon H\to K$ has a bounded inverse $f^{-1}\colon K\to H$. Now let $u,v\in K$. Then, $f(f^{-1}(uv))=uv=f(f^{-1}(u))f(f^{-1}(v))=f(f^{-1}(u)f^{-1}(v))$ so that $f^{-1}(uv)=f^{-1}(u)f^{-1}(v)$. One concludes using the first point above.
\item By combining the first and second points. 
\end{enumerate}
\end{proof}

\subsection{Semisimplicity}

\subsubsection{The Jacobson radical}
The {\em Jacobson radical} or {\em radical} $J(A)$, or $J$ when there is no risk of confusion, of a not necessarily commutative algebra $A$   is the intersection of all its maximal modular left (or right) ideals~\cite[p.~166]{Rowen} and thus if $A$ is unital, then $J(A)$ is the intersection of all maximal left (or right) ideals~\cite{Farb} since in this case every ideal is modular. Observe that every maximal modular left, right or two-sided ideal is closed (\cite[Theorem~2.4.7, p.~236]{Palmer1}) in a complex (unital or not, commutative or not) Banach algebra.  $J(A)$ may equivalently be defined as the intersection of all {\em non-trivial characters} of ${A}$, that is, (necessarily continuous) non-zero algebra maps from ${A}$ to $\mathbb{C}$, if in addition ${A}$ is commutative~\cite[Theorem~2.1.8, p.~49]{Kaniuth} because in this case there is a one-one correspondence between maximal modular ideals and non-trivial characters. Note also in passing that the set $\mathsf{char}(A)$ of all non-trivial characters of a commutative Banach algebra $A$ when equipped with the weakest topology with respect to which all maps $\mathsf{char}(A)\xrightarrow{\hat{x}}\mathbb{C}$, $\hat{x}(\ell):=\ell(x)$, $x\in A$, are continuous, is locally compact and even compact when $A$ furthermore is unital (see~\cite[Theorem~2.2.3, p.~52]{Kaniuth}). $\mathsf{char}(A)$ topologized as above is the {\em character space} of $A$.

The Jacobson radical $J({A})$ of a commutative Banach algebra ${A}$ and that $J({A}^1)$ of its {\em unitarization} ${A}^1$ (see e.g.,~\cite[p.~6]{Kaniuth}) are related as follows. 

\begin{lemma}\label{lem:jacofunitarization}(See also~\cite[Theorem~4.3.2.(a), p.~474]{Palmer1}\footnote{Observe however that in~\cite[Definition~1.1.11, p.~19]{Palmer1} $A^1$ is a {\em conditional unitization}, that is, $A^1$ is either the usual unitarization of $A$ when $A$ has no unit, or is $A$ itself when $A$  already has a unit.}.)
Let ${A}$ be a complex commutative Banach algebra. Then, $J({A})=J({A}^1)$. 
\end{lemma}
%
%
%

Call {\em semisimple} an algebra with a trivial radical, and {\em radical} when it coincides with its own radical. (This choice is  consistent with the common terminology from the theory of Banach algebras. In classical algebra the term {\em Jacobson semisimple} or {\em semiprimitive} corresponds to what we call ``semisimple'' while {\em semisimple} means ``Jacobson semisimple and Artinian'', see e.g.,~\cite{Farb}.) The only algebra which is both semisimple and radical is the zero algebra $0$ (which is unital!).

\subsubsection{Group-like elements}\label{sect:grouplike}
The Jacobson radical $J(H,\mu)$ of a Hilbertian algebra $(H,\mu)$ is defined as the Jacobson radical $J(H,\mu_{\mathsf{bil}})$ of its underlying Banach algebra $((H,\|-\|'),\mu_{\mathsf{bil}})$. $(H,\mu)$ is {\em semisimple} (resp. {\em radical}) when so is the Banach algebra $((H,\|-\|'),\mu_{\mathsf{bil}})$. 

$J(H,\mu)$ also has an intrinsic description: Let $G(H,\mu):=\{\, x\in H\setminus\{\, 0\,\}\colon \mu(x\otimes x)=x\,\}$ be the set of all {\em group-like elements} of $(H,\mu)$. Then, $J(H,\mu)=G(H,\mu)^{\perp}$ and $\overline{\langle G(H,\mu)\rangle}=J(H,\mu)^{\perp}$, where here and elsewhere $\langle X\rangle$ denotes the linear span,  and $\overline{X}$ the closure, of a subset $X$ of $H$. 

$(H,\mu)$ is semisimple when $G(H,\mu)^{\perp}=(0)$ and $(H,\mu)$  is radical when $G(H,\mu)=\emptyset$.  

As a consequence of the Riesz representation theorem, the map $G(H,\mu)\xrightarrow{R} \mathsf{char}(H,\mu_{\mathsf{bil}})$, $x\mapsto \langle \cdot,x\rangle$ is bijective (\cite[Lemma~19, p.~17]{Poinsot2019}). $G(H,\mu)$ becomes a locally compact space under the weakest topology associated to the family of maps $(G(H,\mu)\xrightarrow{R}\mathsf{char}(H,\mu_{\mathsf{bil}})\xrightarrow{\hat{x}}\mathbb{C})_{x\in H}$ (\cite[p.~19]{Poinsot2019}) so that under $R$, $G(H,\mu)$ and $\mathsf{char}(H,\mu_{\mathsf{bil}})$ are homeomorphic.

\begin{lemma}\label{lem:subalg_and_subcoalg}
Let $(H,\mu)$ be a Hilbertian algebra and let $V$ be a closed subspace of $H$ which is both a subalgebra and a subcoalgebra. Then, $(\mu_{|_V})^{\dagger}=(\mu^\dagger)_{|_V}$ and $G(V,\mu_{|_V})=G(H,\mu)\cap V$. 
\end{lemma}
\begin{proof}
 By definition of $\mu_{|_V}$, $i_{V}\circ \mu_{|_{V}}=\mu\circ (i_{V}\hat{\otimes}_2 i_{V})$ so that $\mu_{|_{V}}=\pi_{V}\circ \mu\circ  (i_{V}\hat{\otimes}_2 i_{V})$ and thus $(\mu_{|_{V}})^{\dagger}=(\pi_{V}\hat{\otimes}_2 \pi_{V})\circ \mu^{\dagger}\circ i_{V}=(\mu^{\dagger})_{|_{V}}$, the last equality holds because by definition of $(\mu^{\dagger})_{|_V}$, $(i_{V}\hat{\otimes}_2 i_{V})\circ (\mu^{\dagger})_{|_{V}}=\mu^{\dagger}\circ i_{V}$. It is then clear that $G(V,\mu_{|_V})= G(H,\mu)\cap V$. 
\end{proof}

\begin{lemma}\label{lem:cosemigroupmorphisms}
Let $(H,\mu),(K,\gamma)$ be Hilbertian algebras. Let $f\colon H\to K$ be a bounded linear map. If $f\colon (H,\mu^{\dagger})\to (K,\gamma^{\dagger})$ is a coalgebra map, then $f(G(H,\mu))\subseteq G(K,\gamma)\cup\{\, 0\,\}$. When $(H,\mu)$ is semisimple, the converse also holds.
\end{lemma}
\begin{proof}
The first statement is~\cite[Lemma~25, p.~19]{Poinsot2019}. Let us assume that $(H,\mu)$ is semisimple. Then, by assumption and linearity, $(f\hat{\otimes}_2 f)\circ \mu^{\dagger}=\gamma^{\dagger}\circ f$ on $\langle G(H,\mu)\rangle$. By continuity the maps are equal on $\overline{\langle G(H,\mu)\rangle}=J(H,\mu)^{\perp}=H$ by semisimplicity.
\end{proof}

When $(H,\mu)$ is a special Hilbertian algebra, more can be said about $G(H,\mu)$. First it is an orthonormal family~\cite[Lemma~27, p.~21]{Poinsot2019}, and even an orthonormal basis of $J(H,\mu)^{\perp}$,  and secondly it is a discrete space (\cite[p.~20]{Poinsot2019}). As a direct consequence of the above $G(H,\mu)$ is an orthonormal basis  of $H$ when $(H,\mu)$ is a semisimple special Hilbertian algebra $(H,\mu)$ and for $u,v\in H$, $\mu(u\otimes v)=\sum_{x\in G(H,\mu)}\langle u,x\rangle\langle v,x\rangle x$. Note here and now that a semisimple special Hilbertial algebra thus is a (commutative) $\dagger$-special Frobenius semigroup (see below). 

A notation: if $\mathbf{C}$ is a subcategory of ${}_c\mathbf{Sem}(\mathbb{Hilb})$ or of ${}_c\mathbf{Sem}(\mathbb{FdHilb})$, then ${}_{\mathsf{semisimple}}\mathbf{C}$ stands for the full subcategory of $\mathbf{C}$ spanned by the Hilbertian algebras in $\mathbf{C}$ which are semisimple.

\subsection{Hilbertian Frobenius algebras}\label{sec:hilbertian_frob_alg}

Let $H$ be a Hilbert space and let $H\hat{\otimes}_2 H\xrightarrow{\mu}H$ be a bounded linear map. Let us consider the following diagram where $\alpha_{H,H,H}$ is the component at $(H,H,H)$ of the coherence constraint of associativity of $\mathbb{Hilb}$. 
\begin{equation}\label{diag:Frob}
\xymatrix{
\ar[rd]_{\mu}\ar[d]_{\mu^{\dagger}\hat{\otimes}_2 id}H\hat{\otimes}_2 H\ar[r]^{id\hat{\otimes}_2\mu^{\dagger}} &H\hat{\otimes}_2(H\hat{\otimes}_2 H)\ar[dr]^{\alpha^{-1}_{H,H,H}}& \\
(H\hat{\otimes}_2 H)\hat{\otimes}_2H\ar[rd]_{\alpha_{H,H,H}}&H\ar[rd]_{\mu^{\dagger}}&(H\hat{\otimes}_2 H)\hat{\otimes}_2H\ar[d]^{\mu\hat{\otimes}_2 id}\\
&H\hat{\otimes}_2(H\hat{\otimes}_2 H)\ar[r]_{id\hat{\otimes}_2 \mu}&H\hat{\otimes}_2 H\\
&&
}
\end{equation}
One says that $(H,\mu)$ satisfies the  {\em $\dagger$-Frobenius condition} (or simply the {\em Frobenius condition}) -- or that $(H,\mu)$ is {\em Frobenius} -- when the top and the bottom cells of Diag.(\ref{diag:Frob}) commute. In this case, the surrounding diagram commutes too.

By a {\em  Hilbertian Frobenius semigroup} is meant a Hilbertian algebra  which satisfies the $\dagger$-Frobenius condition. (Such objects are referred to as {\em commutative $\dagger$-Frobenius semigroups} in~\cite{Coecke}.) A  Hilbertian Frobenius semigroup $(H,\mu)$ is said to be {\em special} when furthermore $\mu\circ\mu^{\dagger}=id$. 

\begin{example}
Any Hilbert space with the zero multiplication is a Hilbertian Frobenius semigroup.
\end{example}

\begin{remark}\label{rem:left_or_right_Frob_no_matter}
Since $(H,\mu)$ is assumed  commutative, it is not difficult to check that the Frobenius condition actually reduces to the commutativity of only one of the two cells of Diag.~(\ref{diag:Frob}). One thus recovers the definition of a Frobenius algebra in $\mathbb{Hilb}$ provided in~\cite{Heunen}.
\end{remark}

Let ${}_c\mathbf{FrobSem}(\mathbb{Hilb})$ and ${}_{c}^{\dagger}\mathbf{FrobSem}(\mathbb{Hilb})$  be respectively the full subcategories of ${}_c\mathbf{Sem}(\mathbb{Hilb})$ spanned by the  Hilbertian Frobenius semigroups and by the  special  Hilbertian Frobenius semigroups. 

One obtains corresponding categories after the replacement of $\mathbb{Hilb}$ by $\mathbb{FdHilb}$. Still in the finite-dimensional situation one  may  as well consider   {\em  $\dagger$-Frobenius monoids}, that is, commutative monoid objects $(H,\mu,\eta)$ in $\mathbb{FdHilb}$ such that $(H,\mu)$ is a finite-dimensional $\dagger$-Frobenius semigroup. Let ${}_c\mathbf{FrobMon}(\mathbb{FdHilb})$ be the full subcategory of the category ${}_c\mathbf{Mon}(\mathbb{FdHilb})$ of monoid objects of $\mathbb{FdHilb}$, they generate. 

Finally let us call a (cocommutative) comonoid $(H,\delta,\epsilon)$ in $\mathbb{FdHilb}$ a {\em finite-dimensional  Hilbertian Frobenius comonoid} when $(H,\delta^{\dagger},\epsilon^{\dagger})$ is a Hilbertian Frobenius monoid.  The full subcategory ${}_{coc}\mathbf{FrobComon}(\mathbb{FdHilb})$ of the category of comonoid objects ${}_{coc}\mathbf{Comon}(\mathbb{FdHilb})$  in $\mathbb{FdHilb}$, they generate is of course isomorphic to ${}_c\mathbf{FrobMon}(\mathbb{FdHilb})^{\mathsf{op}}$ under the dagger functor. 

\begin{example}\label{ex:premierexemple}
Let $X$ be a set. Let $x\in X$ and let $\delta_x\colon X\to\mathbb{C}$  be the map which takes the value zero at each member of $X$ except at $x$ where it is equal to $1$, that is, $\delta_x(x')$ is the usual Kronecker delta symbol $\delta_{x,x'}$ ($x'\in X$). Let $\mu_X\colon \ell^2(X)\hat{\otimes}_2 \ell^2(X)\to \ell^2(X)$ be given by $\mu_{X}(\delta_x\otimes \delta_{x'})=\delta_{x,x'}\delta_x$, that is, $\ell^2(X)\times\ell^2(X)\xrightarrow{(\mu_{X})_{\mathsf{bil}}}\ell^2(X)$ is the point-wise multiplication of maps. Then, $(\ell^2(X),\mu_X)$ is a special Hilbertian Frobenius semigroup. 

If $X$ is finite, then $e_X\colon \mathbb{C}\to \ell^2(X)$ given by $e_X(1):=\sum_{x\in X}\delta_x$, is a unit for $(\ell^2(X),\mu_X)$ so that $(\ell^2(X),\mu_X,e_X)$ thus is a finite-dimensional special Hilbertian Frobenius monoid. 
\end{example}

\section{Weighted Hilbert spaces}\label{sec:weighted_hilb}

Let $X$ be a non empty set, and let $w\colon X\to [C,+\infty[$ be a map, where $C>0$. Let $\ell^2_{w}(X):=\{\, f\in \mathbb{C}^{X}\colon \sum_{x\in X}w(x)|f(x)|^2<+\infty\,\}$. With inner product $\langle f,g\rangle_w:=\sum_{x\in X}w(x)f(x)\overline{g(x)}$ this provides a Hilbert space. The corresponding norm is denoted $\|-\|_w$. For $x\in X$, let us identify $\delta_x\colon X\to\mathbb{C}$ with $x$ itself. Under this identification, $\{\, \frac{x}{w(x)^{\frac{1}{2}}}\colon x\in X\,\}$ forms a Hilbertian basis of $\ell^2_w(X)$.  Note that $\langle f,\frac{x}{w(x)^{\frac{1}{2}}}\rangle_w=w(x)^{\frac{1}{2}}f(x)$, $x\in X$.

The next result is clear.
\begin{lemma}\label{lem:inclusion_ell_2}
$\ell^2_w(X)\subseteq \ell_2(X)$, the inclusion is bounded and $\overline{\ell^2_w(X)}=\ell^2(X)$. Furthermore if $w$ is also bounded above, then $\ell^2_w(X)=\ell_2(X)$ as vector spaces and the norms $\|-\|_w$ and $\|-\|$ are equivalent. 
\end{lemma}

Let $m_X\colon \ell^2_w(X)\times\ell^2_w(X)\to \ell^2_w(X)$ be given by $m_X(f,g):=fg$ where by juxtaposition is denoted the pointwise product of maps. $m_X$ is a weak Hilbert-Schmidt mapping (\cite{Kadison}) as $\sum_{x,y\in X}|\langle m_X(\frac{x}{w(x)^{\frac{1}{2}}},\frac{x}{w(x)^{\frac{1}{2}}}),f\rangle_w|^2\leq \frac{1}{C}\|f\|_w^{2}$ for each $f\in \ell^2_w(X)$. Let $\mu_X\colon \ell^2_w(X)\hat{\otimes}_2\ell^2_w(X)\to\ell^2_w(X)$ be the corresponding bounded linear map, that is, $\mu_X(f\otimes g)=fg$ (see~\cite[Theorem~2.6.4, p.~132]{Kadison}). In details, $\mu_X(f\otimes g)=\sum_{x\in X}f(x)g(x)x$. 

It is clear by its very definition that $\mu_X$ is commutative and associative making $(\ell^2_w(X),\mu_X)$ a Hilbertian semigroup. Moreover $\|\mu_X\|_{\mathsf{op}}\leq \frac{1}{C^{\frac{1}{2}}}$. Such a semigroup $(\ell^2_w(X),\mu_X)$ was already considered in~\cite[pp.~11-12]{Poinsot2019} in the situation where $C=1$. 

By direct computations one obtains
\begin{proposition}\label{prop:grplike_of_ell_2_alpha}
$G(\ell^2_w(X),\mu_X)=\{\, \frac{x}{w(x)}\colon x\in X\,\}$ and $(\ell^2_w(X),\mu_X)$ is a  semisimple Hilbertian Frobenius semigroup.  
\end{proposition}

Letting $w\equiv 1$ one observes that $(\ell^2(X),\mu_X)$ is also a (special) Hilbertian Frobenius semigroup, with $G(\ell^2(X),\mu_X)=X$  (Example~\ref{ex:premierexemple}).
\begin{lemma}
$(\ell^2_w(X),m_X)$ is a not necessarily closed ideal of $(\ell^2(X),m_X)$. In fact $(\ell^2_w(X),m_X)=(\ell^2(X),m_X)$ if, and only if, $\ell^2_w(X)$ is closed in $\ell^2(X)$ if, and only if, $w$ is bounded above.
\end{lemma}
\begin{proof}
Let $f\in \ell^2_w(X)$ and let $g\in \ell^2(X)$. Then for each finite subset $F\subseteq X$, $\sum_{x\in F}|f(x)|^2|g(x)|^2\leq \frac{1}{C}\sum_{x\in F}w(x)|f(x)|^2|g(x)|^2\leq \frac{1}{C}\|f\|_w^2\|g\|^2$, that is, $fg\in \ell^2(X)$. 

The first equivalence in the second statement is clear by Lemma~\ref{lem:inclusion_ell_2}. Let us prove the second equivalence. 
If $w$ is bounded above, then by Lemma~\ref{lem:inclusion_ell_2}, $\ell^2_w(X)=\ell^2(X)$. So let us assume that $w$ is not bounded above.   Then, for each $k\in \mathbb{N}\setminus\{\, 0\,\}$, there exists $x_k\in X$ with $w(x_k)\geq k^2$. Let $X_k:=\{\, x\in X\colon k^2\leq w(x)\,\}$. So $X_k\not=\emptyset$ for each $k\geq 1$.  It is clear that $(X_k)_{k\geq 1}$ is a decreasing sequence of non-void sets. Moreover this sequence cannot stabilize, i.e., for no $i\geq 1$, $X_i=X_{i+n}$, for all $n\geq 0$, for if it would mean that if $w(x)\geq i^2$, then $w(x)\geq k^2$ for all $k\geq i$. So for each $i\geq 1$, there exists $n\geq 1$ such that $X_{i+n}\subsetneq X_i$. It follows that for each $i\geq 1$, $\{\, i\leq k\colon X_i=X_k\,\}$ is finite and non-void (it contains $i$). For each $i\geq 1$, let $M_i:=\max \{\, i\leq k\colon X_i=X_k\,\}+1$. Whence $M_i\geq i+1$ and $X_{M_i}\subsetneq X_i$. Choose $x_1\in X_1$ and for each $n\geq 1$,  choose $x_{n+1}\in X_{M_n}$. It follows that $x_i\not=x_j$, $i\not=j$, and $w(x_i)\geq i^2$. Define $f(x_n):=\frac{1}{n}$ and $f(x)=0$ when $x\not=x_n$, $n\geq 1$. It follows that $f\in \ell^2(X)$. But $w^{\frac{1}{2}}f\not\in \ell^2(X)$ since for each $n\geq 1$, $w(x_n)|f(x_n)|^2\geq \frac{n^2}{n^2}=1$. 
\end{proof}

Let us also provide a description of $(\ell^2_w(X),\mu_X)$ under another disguise. 
Under the unitary transformation $\Phi\colon f\mapsto \hat{f}$ from $\ell^2_{w}(X)$ to $\ell^2(X)$, with $\hat{f}(x):=\langle f,\frac{x}{w(x)^{\frac{1}{2}}}\rangle_w=w(x)^{\frac{1}{2}}f(x)$, $x\in X$, one may transport the multiplication $\mu_X$ on $\ell^2(X)$. In detail, the inverse $\Phi^{\dagger}$ of $\Phi$ is given by $\Phi^{\dagger}(f):=w^{-\frac{1}{2}}f$, that is, $(\Phi^{\dagger}(f))(x)=\frac{1}{w(x)^{\frac{1}{2}}}f(x)$, $x\in X$, and then one may define $\mu_{w,X}\colon \ell^2(X)\hat{\otimes}_2\ell^2(X)\to \ell^2(X)$ by $\Phi\circ \mu_X\circ (\Phi^{\dagger}\hat{\otimes}_2\Phi^{\dagger})$, so for each $f,g\in \ell^2(X)$, $\mu_{w,X}(f\otimes g)=w^{-\frac{1}{2}}fg$, that is, for each $x\in X$,   $(\mu_{w,X}(f\otimes g))(x)=w(x)^{-\frac{1}{2}}f(x)g(x)$. (Note that as $C\leq w(x)$, $\frac{1}{w(x)^{\frac{1}{2}}}\leq \frac{1}{C^{\frac{1}{2}}}$ for each $x\in X$, and thus pointwise multiplication of functions by $w^{-\frac{1}{2}}$ is an operator on $\ell^2(X)$, and as $\ell^2(X)$ is closed under pointwise product, given $f,g\in \ell^2(X)$, $w^{-\frac{1}{2}}fg\in \ell^2(X)$.) It is clear that $(\ell^2(X),\mu_{w,X})$ is a semisimple Frobenius semigroup, unitarily isomorphic to $(\ell^2_w(X),\mu_X)$. Note that $G(\ell^2(X),\mu_{w,X})=\{\, \frac{x}{w(x)^{\frac{1}{2}}}\colon x\in X\,\}$. 

\begin{remark}
Everything becomes trivial if $X=\emptyset$ (and thus $w$ is the empty map), that is, $\ell_w(X)$ thus is the zero algebra, which is Frobenius, semisimple and radical. 
\end{remark}

\section{Structure theorem for special Hilbertian algebras}\label{sec:structuretheoremgeneral}


This section serves as a sequel to~\cite{Poinsot2019} by providing new results about special Hilbertian algebras.

\subsection{The general situation}\label{subsec:generalsituation}

Given a non necessarily unital algebra $A$, $E(A)$ denotes the set of all its idempotent elements, that is, the members $e$ of $A$ such that 
$e^2=e$.  It is well-known that when $A$ is a commutative Banach algebra, then $E(A)\cap J(A)=(0)$ (see for instance~\cite[Proposition~4.3.12.(a), p.~479]{Palmer1}. 

%

%


Let us provide a corollary of~\cite[Corollary~28, p.~22]{Poinsot2019}. 
\begin{corollary}\label{cor:more_precise}
Let $({H},\mu)$ be a special Hilbertian algebra. Then, the orthogonal complement $J^{\perp}$ of the Jacobson radical of $({H},\mu)$ is a (closed) subalgebra of $({H},\mu)$.  More precisely for $x,y\in G({H},\mu)$, $xy=\delta_{x,y}x$. As a consequence $J$ is a {\em coideal}, that is, $\mu^{\dagger}(J)\subseteq (J^{\perp}\hat{\otimes}_2 J^{\perp})^{\perp}=(J\hat{\otimes}_2 J^{\perp})\oplus_2 (J\hat{\otimes}_2 J)\oplus_2 (J^{\perp}\hat{\otimes}_2 J)$. 
\end{corollary}
\begin{proof}
One knows from~\cite[Corollary~28, p.~22]{Poinsot2019} that $xy\in J$ for each $x,y\in G({H},\mu)$ with $x\not=y$. But $xy$ is  idempotent by commutativity of $\mu$, as both $x$ and $y$ are idempotents ($x=\mu(\mu^{\dagger}(x))=\mu(x\otimes x)=x^2$). Whence by the above, $xy=0$.  That $J^{\perp}$ is a subalgebra now follows from the fact that it is the closure of the linear span of the group-like elements ${G}({H},\mu)$ (\cite[Theorem~22, p.~17]{Poinsot2019}). Since $\mu(J^{\perp}\hat{\otimes}_2 J^{\perp})\subseteq J^{\perp}$, it follows by Lemma~\ref{lem:restriction} that $\mu^{\dagger}(J)\subseteq (J^{\perp}\hat{\otimes}_2 J^{\perp})^{\perp}=(J^{\perp}\hat{\otimes}_2 J^{\perp})^{\perp}=(J\hat{\otimes}_2 J^{\perp})\oplus_2 (J\hat{\otimes}_2 J)\oplus_2 (J^{\perp}\hat{\otimes}_2 J)$.
\end{proof}

\begin{corollary}\label{cor1:more_precise}
Under the same assumptions as Corollary~\ref{cor:more_precise}, $J^{\perp}$ is a semisimple special Hilbertian algebra under the restrictions of $\mu$ and of $\mu^{\dagger}$. 
\end{corollary}
\begin{proof}
$J^{\perp}$ is a special Hilbertian algebra by~\cite[Lemma~13, p.~13]{Poinsot2019} since $J^{\perp}$ is both a (closed) subcoalgebra and a subalgebra of $({H},\mu)$. By Lemma~\ref{lem:subalg_and_subcoalg}, $G(J^{\perp},\mu_{|_{J^{\perp}}})=G(H,\mu)$ which shows that $(J^{\perp},\mu_{|_{J^{\perp}}})$ is semisimple.  
%
\end{proof}

The above may be summarized into the following structure theorem which is a specialization of~\cite[Theorem~22, p.~17]{Poinsot2019} for special Hilbertian algebras. 
\begin{theorem}\label{thm:structthm4specialHilbAlg} (Structure theorem for special Hilbertian algebras.)
Let $({H},\mu)$ be a special Hilbertian algebra.  Then, ${H}=J\hat{\oplus}_2 J^{\perp}$ (orthogonal direct sum of Hilbert spaces) and $J^{\perp}$ is a closed subalgebra and subcoalgebra of $({H},\mu)$.
\end{theorem}

\subsection{The unital case}\label{sec:unitalcase}

Let $(H,\mu)$ be a  special Hilbertian algebra which is assumed to have a unit $\mathbb{C}\xrightarrow{\eta}H$ identified with the member $e:=\eta(1)$ of $H$.  Thus $\mathsf{char}(H,\mu_{\mathsf{bil}})$, and so  $G(H,\mu)$ too, is compact as is the character space of any commutative and unital Banach algebra~\cite[Theorem~2.2.3, p.~52]{Kaniuth}. But $G(H,\mu)$ is discrete (see Section~\ref{sect:grouplike}). Therefore it is finite. As it is an orthonormal basis of $J(H,\mu)^{\perp}$, the latter is finite-dimensional with $\dim_\mathbb{C}J^{\perp}=|G(H,\mu)|$. One may even explicitly describes the unit of $(H,\mu)$.
\begin{lemma}\label{lem:idempotent_are_in_ortho_of_rad}
Let $(H,\mu)$  be a  special Hilbertian algebra with a unit $e$. Then, $E(H,\mu_{\mathsf{bil}})\subseteq J^{\perp}$ and in particular $e\in J^{\perp}$. More precisely, $e=\sum_{x\in G(H,\mu)}x$.
\end{lemma}
\begin{proof}
The first assertion follows from~\cite[Lemma~3.3, p.~112]{BadeDales} in view of Corollary~\ref{cor:more_precise}. The second thus is immediate. Since $(H,\mu)$ is unital, one already knows that $|G(H,\mu)|=\dim_{\mathbb{C}}J^{\perp}<+\infty$. As $1=\langle x,x\rangle=\langle ex,x\rangle=\langle e\otimes x,x\otimes x\rangle=\langle e,x\rangle\langle x,x\rangle$, whence $\langle e,x\rangle=1$ for each $x\in G(H,\mu)$. So $e=\sum_{x\in G(H,\mu)}\langle e,x\rangle x=\sum_{x\in G(H,\mu)}x$.
\end{proof}


If $(H,\mu)$ is a finite-dimensional special Hilbertian algebra, then it is unital~\cite[Corollary~3.3, p.~47]{Dixon} as $H^2:=\mu(H\otimes_{\mathbb{C}} H)=H$ (because $\mu\circ\mu^{\dagger}=id$). One so obtains
\begin{proposition}\label{prop:unital}
Let $(H,\mu)$ be a special Hilbertian algebra. 
\begin{enumerate}
\item If $(H,\mu)$ has a unit, then $J(H,\mu)^{\perp}$ is finite-dimensional.
\item If $(H,\mu)$ is finite-dimensional, then it has a unit.
\item Let us assume $(H,\mu)$ semisimple. $(H,\mu)$ has a unit if, and only if, it is finite-dimensional.
\end{enumerate}
\end{proposition}
Incidentally Proposition~\ref{prop:unital} shows that a special Hilbertian algebra of finite dimension fails to be radical if non-zero (since the unit is contained in no maximal ideals.) 

\section{Semisimplicity of commutative Hilbertian Frobenius semigroups}\label{sec:comHilbFromSemigroups}

The main results of this section are Theorem~\ref{thm:stHFs} and Theorem~\ref{thm:semisimplicityforFrob}. The former states that both the Jacobson radical  of a Hilbertian Frobenius semigroup, and its orthogonal complement are subalgebras and subcoalgebras, and that the Jacobson radical actually coincides with the annihilator of the semigroup. The latter provides the explicit conditions on the multiplication (or comultiplication) of a Hilbertian Frobenius semigroup for it to be semisimple. 

Theorem~\ref{thm:stHFs} is obtained in two steps: at first it is proved that the orthogonal complement of the Jacobson radical is a subalgebra, by analyzing precisely how the product of two group-like elements behaves, and secondly that the operators of multiplication by an element in a Frobenius algebra are forced to be normal operators, which in turn forces the Jacobson radical to be equal to the annihilator. 

\subsection{Commutative Hilbertian Frobenius algebras}

In this current section, $(H,\mu)$ stands for  a commutative Hilbertian Frobenius algebra and one denotes  $J:=J(H,\mu)$. 

Let us recall the following result from~\cite{Coecke}. Recall that in our terminology an orthogonal family  does not contain $0$. 
\begin{lemma}\label{lem:orthogonal}
$G(H,\mu)$ is an orthogonal family, that is, for each $x,y\in G(H,\mu)$, $x\not=y\Rightarrow\langle x,y\rangle=0$. In particular $G(H,\mu)$ is an orthogonal basis of $J^{\perp}$. 
\end{lemma}

Lemma~\ref{lem:orthogonal} has the important consequences listed below.
\begin{corollary}\label{cor:orthogonal2}
\begin{enumerate}
\item Let $x,y\in G(H,\mu)$ such that $x\not=y$. Then, 
$xy\in J$. 
\item Let $x\in G(H,\mu)$. Then, $p_{J^{\perp}}(x^2)=\|x\|^2 x$. 
\end{enumerate}
\end{corollary}
\begin{proof}
\begin{enumerate}
\item Let $x,y\in G(H,\mu)$ such that $x\not=y$.  Let $z\in G(H,\mu)$. Then, $\langle xy,z\rangle=\langle\mu(x\otimes y),z\rangle=\langle x\otimes y,\mu^{\dagger}(z)\rangle=
\langle x\otimes y,z\otimes z\rangle=\langle x,z\rangle\langle y,z\rangle=\|x\|^2\delta_{x,z}\|y\|^2\delta_{y,z}$ (by Lemma~\ref{lem:orthogonal}) $=0$ as by assumption $x\not=y$. Whence $p_{J^{\perp}}(xy)=0$ and thus $xy\in J$. 
\item Let $x\in G(H,\mu)$. Let $z\in G(H,\mu)$. Then, $\langle x^2,z\rangle=\langle \mu(x\otimes x),z\rangle=\langle x\otimes x,\mu^{\dagger}(z)\rangle=\langle x\otimes x,z\otimes z\rangle=\langle x,z\rangle^2=\|x\|^4\delta_{x,z}$ according to Lemma~\ref{lem:orthogonal}. As a result, $p_{J^{\perp}}(x^2)=\langle x^2,\frac{x}{\|x\|}\rangle \frac{x}{\|x\|}=\|x\|^2x$.
\end{enumerate}
\end{proof}

\begin{remark}\label{rem:description_of_semisimple_HilbFrob}
According to Corollary~\ref{cor:orthogonal2} a semisimple Hilbertian Frobenius semigroup $(H,\mu)$ is easily described: as $G(H,\mu)$ is an orthogonal basis of $H=J^{\perp}$, one has for each $u\in H$ a unique expansion as a summable family $\sum_{x\in X}\langle u,\frac{x}{\|x\|}\rangle \frac{x}{\|x\|}$ and given $u,v\in H$, $uv=\sum_{x\in G(H,\mu)}\frac{1}{\|x\|}\langle u,x\rangle\langle v,x\rangle \frac{x}{\|x\|}$ as $\langle uv,x\rangle=\langle u\otimes v,x\otimes x\rangle=\langle u,x\rangle\langle v,x\rangle$, $x\in G(H,\mu)$. 
\end{remark}

\begin{lemma}\label{lem:preidem}
Let $x\in G(H,\mu)$. $x^2\in J^{\perp}$ if, and only if, $x^2=\|x\|^2 x$. In this case, $\frac{1}{\|x\|^2}x$ is an idempotent element which belongs to $J^{\perp}$ and $\|x\|\leq \|\mu\|_{\mathsf{op}}$.
\end{lemma}
\begin{proof}
The first equivalence is due to Corollary~\ref{cor:orthogonal2}. The second statement is immediate. 
%
Let $e$ be a non-zero idempotent element of $(H,\mu)$. Then, $\|e\|=\|e^2\|=\|\mu(e\otimes e)\|\leq \|\mu\|_{\mathsf{op}}\|e\|^2$. If $e\not=0$, then $\frac{1}{\|\mu\|_{\mathsf{op}}}\leq \|e\|$. The last statement thus is obtained by taking $e=\frac{x}{\|x\|^2}$. 
\end{proof}

Let $G'(H,\mu):=\{\, x\in G(H,\mu)\colon x^2\in J^{\perp}\,\}=\{\, x\in G(H,\mu)\colon x^2=\|x\|^2 x\,\}$ (by Lemma~\ref{lem:preidem}). 
\begin{lemma}\label{lem:preidem_2}
Let $x,y\in G'(H,\mu)$ such that $x\not=y$. Then, $xy=0$.
\end{lemma}
\begin{proof}
According to Corollary~\ref{cor:orthogonal2} $xy\in J$. As $x,y\in G'(H,\mu)$, $\frac{x}{\|x\|^2}$ and $\frac{y}{\|y\|^2}$ both are idempotent elements which belong to $J^{\perp}$ by Lemma~\ref{lem:preidem}. By commutativity of $\mu$, $\frac{xy}{\|x\|^2\|y\|^2}$ is also an idempotent element and it is a member of $J$ by the above. Therefore it reduces to zero (see the beginning of Section~\ref{subsec:generalsituation}), and thus $xy=0$ too. 
\end{proof}

\begin{lemma}\label{lem:Jperp_is_subcoalg}
$G'(H,\mu)=G(H,\mu)$. Consequently, $J^{\perp}$ is a closed subalgebra of $(H,\mu)$ and the map $G(H,\mu)\to ]0,+\infty[$, $x\mapsto \|x\|$ is bounded above by $\|\mu\|_{\mathsf{op}}$.
\end{lemma}
\begin{proof}
Let $x\in G(H,\mu)$. Using one of the Frobenius conditions, one obtains $\mu^{\dagger}(x^2)=x^2\otimes x$ and with the other, $\mu^{\dagger}(x^2)=x\otimes x^2$. 
Let $u,v\in H$. Then, $\langle \mu^{\dagger}(x^2),u\otimes v\rangle=\langle x\otimes x^2,u\otimes v\rangle=\langle x,u\rangle\langle x^2,v\rangle$ and also $\langle \mu^{\dagger}(x^2),u\otimes v\rangle=\langle x^2\otimes x,u\otimes v\rangle=
\langle x^2,u\rangle\langle x,v\rangle$. In particular given $u\in J$, one has $\langle \mu^{\dagger}(x^2),u\otimes x\rangle=
\langle x^2,x\rangle \langle x,u\rangle=0$ and $\langle \mu^{\dagger}(x^2),u\otimes x\rangle=\langle x,x\rangle\langle x^2,u\rangle=\|x\|^2\langle x^2,u\rangle$. Therefore $x^2\in J^{\perp}$, that is, $x\in G'(H,\mu)$.

According to Lemma~\ref{lem:preidem} and Lemma~\ref{lem:preidem_2}, the linear span of $G'(H,\mu)=G(H,\mu)$ is closed under $\mu$. So is its closure, by continuity of $\mu$, which is nothing but $J^{\perp}$. The last assertion is a consequence of the last assertion of Lemma~\ref{lem:preidem}.
\end{proof}

\begin{remark}
It is clear from the definition of a group-like element, that $\mu=0\Rightarrow G(H,\mu)=\emptyset \Rightarrow (H,\mu)$ is radical. 

Actually we will prove below that the converse implications are also true (see Corollary~\ref{cor:radicalFrob}). 
\end{remark}

\begin{remark}\label{rem:description_of_Jperp_of_HilbFrob2}
According to Lemma~\ref{lem:Jperp_is_subcoalg} the multiplication of $J^{\perp}$ arising from  the restriction of $\mu$ is as easily described  as in Remark~\ref{rem:description_of_semisimple_HilbFrob}.  $G(H,\mu)$ is an orthogonal basis of $J^{\perp}$, and given $u,v\in H$, $uv=\sum_{x\in G(H,\mu)}\frac{1}{\|x\|}\langle u,x\rangle\langle v,x\rangle \frac{x}{\|x\|}+(p_{J^{\perp}}(u)p_{J}(v)+p_{J}(u)p_{J^{\perp}}(v)+p_{J}(u)p_{J}(v))$. 
\end{remark}

A more precise Structure Theorem for commutative Hilbertian Frobenius semigroups will be provided hereafter (Theorem~\ref{thm:stHFs}) but it is worth summarizing the above results. 
\begin{proposition}\label{prop:stHFs}
Let $(H,\mu)$ be a Hilbertian Frobenius semigroup. Then, ${H}=J{\oplus}_2 J^{\perp}$ (orthogonal direct sum of Hilbert spaces) and $J^{\perp}$ is both a closed subalgebra and subcoalgebra of $({H},\mu)$.
\end{proposition}

\begin{corollary}\label{cor:stHFs0}
Let $(H,\mu)$ be a Hilbertian Frobenius semigroup. Under the  restriction of $\mu$, $J^{\perp}$ is a semisimple Hilbertian Frobenius semigroup. 
\end{corollary}

\subsection{Multiplication operators}

Let $(E,*)$ be a commutative Banach algebra. The {\em annihilator} $A(E,*)$ of $(E,*)$ is $\{\, u\in E\colon \forall v\in E,\ u*v=0\,\}$. For a Hilbertian algebra $(H,\mu)$ let $A(H,\mu):=A(H,\mu_{\mathsf{bil}})$. The following result is clear.
\begin{lemma}\label{lem:annihilator_included_in_jacobson}
$A(H,\mu)\subseteq J(H,\mu)$.
\end{lemma}
%

Let $(E,*)$ be a complex commutative Banach algebra. 
Let $M\colon E\to\mathcal{B}(E)$ be the regular representation of $(E,*)$, that is, $u\mapsto M_u$, where $M_u(v)=u*v$. One notices that $A(E,*)=\ker M$. 

{Let $(B,*)$ be a not necessarily commutative Banach algebra and let $u\in B$. $u$ is said to be a {\em quasi-nilpotent} element when its spectral radius is equal to zero, that is, when $\|u^n\|^{\frac{1}{n}}\to 0$ (\cite[p.~213]{Palmer1}). In general the Jacobson radical of $(B,*)$ is only contained into the set of all quasi-nilpotent elements but as soon as $(B,*)$ is commutative both sets are equal~\cite[Corollary~2.2.6, p.~55]{Kaniuth}. If $H$ is a Hilbert space a linear map $H\xrightarrow{f}H$ is referred to as a {\em quasi-nilpotent operator} when it is  a quasi-nilpotent element of $\mathcal{B}(H)$. 
\begin{lemma}\label{lem:mult_by_quasinil}
Let $u\in J(H,\mu)$. Then, $M_u$ is a quasi-nilpotent operator on $H$. In other words, $M$ maps the Jacobson radical of $(H,\mu)$ into the set of all quasi-nilpotent operators on $H$.
\end{lemma}
\begin{proof}
Let $v\in H$. Then, $\|M^n_u(v)\|=\|u^n v\|\leq \|\mu_{\mathsf{bil}}\|_{\mathsf{op}}\|u^n\|\|v\|\leq \|u^n\|'\|v\|$ so that $\|M^n_u\|_{\mathsf{op}}\leq \|u^n\|'$ for each $n\in\mathbb{N}\setminus\{\, 0\,\}$. Consequently, $\|M^n_u\|_{\mathsf{op}}^{\frac{1}{n}}\leq (\|u^n\|')^{\frac{1}{n}}\rightarrow 0$ as $u$ is a quasi-nilpotent element of the Banach algebra $((H,\|-\|'),\mu_{\mathsf{bil}})$ (\cite[Corollary~2.2.6, p.~55]{Kaniuth}).
\end{proof}

A bounded linear operator $H\xrightarrow{f}H$ over a Hilbert space $H$ is said to be {\em normal} when $f\circ f^{\dagger}=f^{\dagger}\circ f$. 
\begin{corollary}\label{cor:mult_by_quasinil}
Let $(H,\mu)$ be a Hilbertian algebra. Let $u\in J(H)$. $M_u$ is normal if, and only if, $u\in A(H,\mu)$. So $J({H},\mu)\cap\{\, u\in H\colon M_u\ \mbox{is normal}\,\}=A(H,\mu)$.
\end{corollary}
\begin{proof}
The converse implication is clear since the zero operator is normal. So let us assume that $M_u$ is normal. By Lemma~\ref{lem:mult_by_quasinil}, $M_u$ is quasi-nilpotent, whence its spectral radius is equal to zero (\cite[p.~213]{Palmer1}). But for normal  operators the spectral radius coincides with the operator norm (\cite[II.1.6.3, p.~58]{Blackadar}). Whence $\|M_u\|_{\mathsf{op}}=0$, that is, $M_u=0$. 
The second statement follows from the first one and Lemma~\ref{lem:annihilator_included_in_jacobson}.
\end{proof}

\subsection{Frobenius algebras revisited: Hilbertian modules}\label{sec:frobalgrevisited}


Let $(H,\mu)$ be an object of $\mathbf{Sem}(\mathbb{Hilb})$ and let $K$ be a Hilbert space. Let $g\colon {H}\hat{\otimes}_2 {K}\to {K}$ be a bounded linear map such that the following diagram commutes. (The isomorphism arrow corresponds to the coherence constraint of associativity of $\mathbb{Hilb}$.)
\begin{equation}\label{diag:left_hilb_mod}
\xymatrix{
&{H}\hat{\otimes}_2({H}\hat{\otimes}_2 {K}) \ar[r]^-{id\hat{\otimes}_2 g}&{H}\hat{\otimes}_2{K}\ar[dd]^{g}\\
\eq[ur]({H}\hat{\otimes}_2{H})\hat{\otimes}_2{K} \ar[rd]_{\mu\hat{\otimes}_2 id}&\\
& {H}\hat{\otimes}_2{K} \ar[r]_{g}& {K} 
}
\end{equation}
The pair $({K},g)$ is referred to as a {\em (Hilbertian) left $(H,\mu)$-module} and $g$ is called the {\em left action} of $(H,\mu)$. The notion of a {\em (Hilbertian) right $(H,\mu)$-module} $(K,r)$, with $r\colon K\hat{\otimes}_2 H\to K$, is obtained by symmetry.    $r$ is the {\em right action} of $(H,\mu)$.
\begin{equation}\label{diag:right_hilb_mod}
\xymatrix{
&({K}\hat{\otimes}_2{H})\hat{\otimes}_2{H} \ar[r]^-{r\hat{\otimes}_2 id}&{K}\hat{\otimes}_2{H}\ar[dd]^{r}\\
\eq[ur]{K}\hat{\otimes}_2({H}\hat{\otimes}_2{H}) \ar[rd]_{id\hat{\otimes}_2 \mu}&\\
& {K}\hat{\otimes}_2{H} \ar[r]_{r}& {K} 
}
\end{equation}
Let a Hilbert space $K$ being both a left and a right $(H,\mu)$-module $(K,g)$ and $(K,r)$ on $(H,\mu)$. $(H,g,r)$ is said to be a {\em Hilbertian $(H,\mu)$-bimodule} when furthermore the following diagram commutes.
\begin{equation}\label{diag:bimod}
\xymatrix{
\eq[dd]_{\alpha_{H,K,H}}(H\hat{\otimes}_2 K)\hat{\otimes}_2 H \ar[r]^-{g\hat{\otimes}_2 id}& K\hat{\otimes}_2 H\ar[d]^{r}\\
&K&\\
H\hat{\otimes}_2(K\hat{\otimes}_2 H) \ar[r]_-{id\hat{\otimes}_2 r}&H\hat{\otimes}_2 K\ar[u]_{g}
}
\end{equation}
\begin{lemma}\label{lem:char_of_left_mod}
Let $(H,\mu)$ be an object of $\mathbf{Sem}(\mathbb{Hilb})$ and let $K$ be a Hilbert space. Let $g\colon H\hat{\otimes}_2 K\to K$ be a bounded linear map. $(K,g)$ is a left $(H,\mu)$-module if, and only if, for each $x,y\in H$, $z\in K$, $g_{\mathsf{bil}}(x,g_{\mathsf{bil}}(y,z))=g_{\mathsf{bil}}(xy,z)$. 
\end{lemma}
%

By symmetry, 
\begin{lemma}\label{lem:char_of_right_mod}
Let $(H,\mu)$ be an object of $\mathbf{Sem}(\mathbb{Hilb})$ and let $K$ be a Hilbert space. Let $r\colon K\hat{\otimes}_2 H\to K$ be a bounded linear map. $(K,r)$ is a right $(H,\mu)$-module over $(H,\mu)$ if, and only if, for each $x\in K$, $y,z\in H$, $r_{\mathsf{bil}}(r_{\mathsf{bil}}(x,y),z)=r_{\mathsf{bil}}(x,yz)$. 
\end{lemma}


Given  left $(H,\mu)$-modules $({K}_i,g_i)$, $i=1,2$, a bounded linear map ${K}_1\xrightarrow{f}{K}_2$ is a {\em left $(H,\mu)$-module map} or is said to be {\em left $(H,\mu)$-linear}, or simply {\em left linear} when the following diagram commutes.
\begin{equation}
\xymatrix{
{H}\hat{\otimes}_2{K}_1\ar[r]^{id\hat{\otimes}_2 f}\ar[d]_{g_1} & {H}\hat{\otimes}_2{K}_2\ar[d]^{g_2}\\
{K}_1 \ar[r]_f& {K}_2
}
\end{equation} 
By symmetry {\em left $(H,\mu)$-module} (or {\em left $(H,\mu)$-linear}) maps are obtained. 

\begin{example}\label{ex:module_structures}
Let $(H,\mu)$ be an object of ${}_c\mathbf{Sem}(\mathbb{Hilb})$. 
\begin{enumerate}
\item\label{item:1:ex:module_structures} ${H}$ is itself a left and right $(H,\mu)$-module under $g=\mu=r$, by associativity of $\mu$. Of course, $(H,\mu,\mu)$ thus is a bimodule over itself by associativity of $\mu$. 
\item\label{item:2:ex:module_structures} $H\hat{\otimes}_2H$ is a left $(H,\mu)$-module   under $H\hat{\otimes}_2({H}\hat{\otimes}_2{H})\xrightarrow{\alpha_{H,H,H}^{-1}}({H}\hat{\otimes}_2{H})\hat{\otimes}_2{H}\xrightarrow{\mu\hat{\otimes}_2 id}{H}\hat{\otimes}_2{H}$. It is also a right $(H,\mu)$-module with right action $({H}\hat{\otimes}_2{H})\hat{\otimes}_2{H}\xrightarrow{\alpha_{H,H,H}} {H}\hat{\otimes}_2({H}\hat{\otimes}_2{H})\xrightarrow{id\hat{\otimes}_2 \mu}{H}\hat{\otimes}_2{H}$, where one recalls that  $\alpha_{H,K,L}\colon (H\hat{\otimes}_2 K)\hat{\otimes}_2 L\simeq H\hat{\otimes}_2(K\hat{\otimes}_2 L)$ is the coherence constraint of associativity. Observe that $H\hat{\otimes}_2 H$ with the left and the right actions of $(H,\mu)$ as above, is a Hilbertian bimodule.
\end{enumerate}
In what follows, one tacitly assumes that ${H}$ and ${H}\hat{\otimes}_2{H}$ both have the above left or right module structures. 
\end{example}

\begin{remark}
It is clear that in view of Remark~\ref{rem:left_or_right_Frob_no_matter}, a  commutative Hilbertian algebra $(H,\mu)$ is Frobenius  if, and only if, $\mu^{\dagger}\colon H\to H\hat{\otimes}_2 H$ is either left or right $(H,\mu)$-linear. 
\end{remark}

Let us introduce the following notations. Let $H,K,L$ be Hilbert spaces. Let $\gamma\colon H\times K\to L$ be a bounded bililnear map.  One may define $H\xrightarrow{{\gamma_{\mathsf{left}}}}\mathcal{B}(K,L)$ and $K\xrightarrow{\gamma_{\mathsf{right}}}\mathcal{B}(H,L)$ by setting $({\gamma_{\mathsf{left}}}(x))(y):=\gamma(x,y)=:(\gamma_{\mathsf{right}}(y))(x)$, $x\in H$, $y\in K$. When $\gamma\colon H\hat{\otimes}_2 K\to L$ is a bounded linear map, or equivalently when $\gamma_{\mathsf{bil}}\colon H\times K\to L$ is a weak Hilbert-Schmidt map, then one also defines ${\gamma_{\mathsf{left}}}:=(\gamma_{\mathsf{bil}})_{\mathsf{left}}$ and $\gamma_{\mathsf{right}}:=(\gamma_{\mathsf{bil}})_{\mathsf{right}}$. 
\begin{lemma}
 Let $H,K,L$ be  Let $H,K,L$ be Hilbert spaces. Let $\gamma\colon H\times K\to L$ be a bounded bililnear map.  Then, ${\gamma_{\mathsf{left}}}$ and $\gamma_{\mathsf{right}}$ are bounded linear maps. 
\end{lemma} 

\begin{example}\label{ex:module_structures2} 
Let $({H},\mu)$ be an object of ${}_c\mathbf{Sem}(\mathbb{Hilb})$. 
\begin{enumerate}
\item\label{item:1:ex:module_structures2}  For the structure of left or right module over $({H},\mu)$, under $\lambda=\mu=\rho$, one has ${\mu}_{\mathsf{left}}=M=\mu_{\mathsf{right}}$ (by commutativity). 

\item\label{item:2:ex:module_structures2} For the structure $g\colon H\hat{\otimes}_2 (H\hat{\otimes}_2 H)\to H\hat{\otimes}_2 H$ of left  $({H},\mu)$-module on ${H}\hat{\otimes}_2{H}$ from Example~\ref{ex:module_structures}.\ref{item:1:ex:module_structures}, one has ${g}_{\mathsf{left}}(u)(v\otimes w)=(\mu\hat{\otimes}_2 id)((u\otimes v)\otimes w)=\mu(u\otimes v)\otimes w=M_u(v)\otimes w$, $u,v,w\in H$, so that ${g}_{\mathsf{left}}(u)=M_u\hat{\otimes}_2 id$ on $H\otimes_\mathbb{C} H$. As ${g}_{\mathsf{left}}(u)$ and $M_u\hat{\otimes}_2 id$ are both linear and  continuous, and $H\otimes_{\mathbb{C}}H$ is dense in $H\hat{\otimes}_2 H$, these maps are equal on the whole $H\hat{\otimes}_2 H$.

For the structure $r\colon ({H}\hat{\otimes}_2{H})\hat{\otimes}_2{H}\to {H}\hat{\otimes}_2{H}$ of right $({H},\mu)$-module   on ${H}\hat{\otimes}_2{H}$ also  from Example~\ref{ex:module_structures}.\ref{item:1:ex:module_structures},  one has $r_{\mathsf{right}}(u)(v\otimes w)=(id\hat{\otimes}_2 \mu)(v\otimes (w\otimes u))=v\otimes \mu(w\otimes u)=w\otimes \mu(u\otimes w)=w\otimes M_u(v)$ by commutativity of $\mu$. Therefore, $r_{\mathsf{right}}(u)=id\hat{\otimes}_2 M_u$.  
\end{enumerate}
\end{example}

\begin{lemma}\label{lem:char_of_leftlinearmap}
Let $({K},g),(K,g')$ be Hilbertian left $(H,\mu)$-modules, and let $K\xrightarrow{f}K'$ be a bounded linear map. It is left $(H,\mu)$-linear if, and only if, for each $u\in {H}$, 
$g'_{\mathsf{left}}(u)\circ f=f\circ g_{\mathsf{left}}(u)$.
\end{lemma}
\begin{proof}
Let $u\in{H}$ and $v\in {K}$. One has $f(g(u\otimes v))=
f(g_{\mathsf{left}}(u)(v))$ while $g'((id\hat{\otimes}_2 f)(u\otimes v))=g'(u\otimes f(v))={g}'_{\mathsf{left}}(u)(f(v))$. It is thus clear that if $f$ is left linear, then ${g}'_{\mathsf{left}}(u)\circ f=f\circ {g}_{\mathsf{left}}(u)$ for all $u$. Conversely, let us assume that ${g}'_{\mathsf{left}}(u)\circ f=f\circ {g}_{\mathsf{left}}(u)$ for all $u$. By the above, $f\circ g=g'\circ (id\hat{\otimes}_2 f)$ on $H\otimes_{\mathbb{C}} K$. By linearity and continuity, since ${H}\otimes_{\mathbb{C}}{K}$ is dense in ${H}\hat{\otimes}_2{K}$, the equality holds on ${H}\hat{\otimes}_2{K}$, so $f$ is left linear.
\end{proof}

By symmetry, one has
\begin{lemma}\label{lem:char_of_rightlinearmap}
Let $(K,r),(K',r')$ be Hilbertian right $(H,\mu)$-modules, and let ${K}\xrightarrow{f}{K}'$ be a bounded linear map. It is right $(H,\mu)$-linear if, and only if, for each $u\in {H}$, 
$r'_{\mathsf{right}}(u)\circ f=f\circ r_{\mathsf{right}}(u)$.
\end{lemma}

\begin{proposition}\label{prop:normality}
Let $({H},\mu)$ be a Hilbertian Frobenius algebra. Then, for each $u\in {H}$, $M_u$ is normal. In particular, $J(H,\mu)=A(H,\mu)$.
\end{proposition}
\begin{proof}
Let $u,v,w\in \mathcal{H}$. Then, 
\begin{equation}
\begin{array}{lll}
\langle M_u^{\dagger}(M_u(v)),w\rangle&=&\langle M_u(v),M_u(w)\rangle\\
&=&\langle uv,uw\rangle.
\end{array}
\end{equation}
Now let us assume that $\mu^{\dagger}$ is right linear. One  has 
\begin{equation}
\begin{array}{lll}
\langle M_u(M_u^{\dagger}(v)),w\rangle&=&\langle M_u^{\dagger}(v),M_u^{\dagger}(w)\rangle\\
&=&\langle v,M_u(M_u^{\dagger}(w))\rangle\\
&=&\langle v,u M_u^{\dagger}(w)\rangle\\
&=&\langle\mu^{\dagger}(v),u\otimes M_u^{\dagger}(w)\rangle\\
&=&\langle (id\hat{\otimes}_2M_u)(\mu^{\dagger}(v)),u\otimes w\rangle\\
&=&\langle \mu^{\dagger}(uv),u\otimes w\rangle\\
&&\mbox{(according to Lemma~\ref{lem:char_of_leftlinearmap})}\\
&=&\langle uv,uw\rangle.
\end{array}
\end{equation}

The case of left linearity would be  treated similarly  using commutativity of $\mu$. The last statement is a direct consequence of Corollary~\ref{cor:mult_by_quasinil}.  
\end{proof}

We are now in position to state the following structure theorem which completes Proposition~\ref{prop:stHFs} and a corollary that extends Corollary~\ref{cor:stHFs0}.
\begin{theorem}\label{thm:stHFs}(Structure Theorem for Hilbertian Frobenius Semigroups)
Let $(H,\mu)$ be a commutative Hilbertian Frobenius semigroup. Then, ${H}=J{\oplus}_2 J^{\perp}$ (orthogonal direct sum of Hilbert spaces), $J^{\perp}$ is both a closed subalgebra and subcoalgebra of $({H},\mu)$, and $J=A(H,\mu)$ is also both a closed subalgebra and subcoalgebra. In particular, $J$ and $J^{\perp}$ are ideals. 
\end{theorem}
\begin{proof}
By Proposition~\ref{prop:normality}, $A(H,\mu)=J$, and since $J^{\perp}$ is a subalgebra by Proposition~\ref{prop:stHFs}, $J^{\perp}$ is actually an ideal. Whence $J$ is a subcoalgebra  by~\cite[Theorem~18, p.~15]{Poinsot2019}.
\end{proof}

\begin{remark}
As noticed  in the Introduction of~\cite{Poinsot2019}, the proof of~\cite[Proposition~23, p.~16]{Heunen} contains a mistake. However the statement of this proposition is valid since it is nothing but the special case of the above theorem when $\mu$ is counitary. But in fact, in this special case, more may be said (see Theorem~\ref{thm:semisimplicityforFrob} below). 
\end{remark}

\begin{corollary}\label{cor:stHFs}
Let $(H,\mu)$ be a commutative Hilbertian Frobenius semigroup. Under the corresponding restrictions of $\mu$, $J^{\perp}$ is a semisimple Hilbertian Frobenius semigroup and $J$ is a radical Hilbertian Frobenius semigroup. 
\end{corollary}

\begin{remark}\label{rem:description_of_HilbFrob}
Let $(H,\mu)$ be a Hilbertian Frobenius semigroup. By Theorem~\ref{thm:stHFs},  $uv=p_{J^{\perp}}(u)p_{J^{\perp}}(v)$ (since $J=A(H,\mu)$) $=\sum_{x\in G(H,\mu)}\langle u,x\rangle\langle v,x\rangle \frac{x}{\|x\|^2}$ as follows from Remark~\ref{rem:description_of_Jperp_of_HilbFrob2}. In particular for each $x\in G(H,\mu)$ and $u\in H$, $ux=p_{J^{\perp}}(u)x=\langle u,x\rangle x$ and thus $\mathbb{C}x$ is an ideal. (In particular, $xx=\langle x,x\rangle x=\|x\|^2x$ as already known.)
\end{remark}

\begin{corollary}\label{cor:radicalFrob}
Let $(H,\mu)$  be a commutative Hilbertian Frobenius semigroup. $(H,\mu)$ is radical if, and only if, $\mu=0$.
\end{corollary}

\begin{theorem}\label{thm:semisimplicityforFrob}
Let $(H,\mu)$ be a commutative Hilbertian Frobenius semigroup. The following assertions are equivalent.
\begin{enumerate}
\item $(H,\mu)$ is semisimple.
\item $(H,\mu)$ is {\em faithful}, that is, $\ker M=(0)$.
\item $\mu$ has a dense range.
\item $\mu^{\dagger}$ is one-to-one.
\item $\mu\circ \mu^{\dagger}$ is one-to-one.
\end{enumerate}
In particular, any commutative special Hilbertian Frobenius semigroup is semisimple. 
\end{theorem}
\begin{proof}
The last statement is a consequence of the presumed equivalences. That the two first points are equivalent is clear as $J(H,\mu)=A(H,\mu)$ (Proposition~\ref{prop:normality}). That the three other assertions are equivalent is due to the general fact that for a bounded linear map $K\xrightarrow{f}L$ between Hilbert spaces, $\ker f^{\dagger}=\ker (f\circ f^{\dagger})=ran(f)^{\perp}$ (see e.g.,~\cite[Proposition~5.76, p.~390]{Kubrusly2ndEd}). It remains for instance to prove that semisimplicity is equivalent to injectivity of $\mu^{\dagger}$. So let us assume that $\mu^{\dagger}$ is one-to-one. According to Theorem~\ref{thm:stHFs}, $J(H,\mu)=A(H,\mu)$ is a subcoalgebra, that is, $\mu^{\dagger}(J(H,\mu))\subseteq J(H,\mu)\hat{\otimes}_2 J(H,\mu)$. Whence for each $x\in J(H,\mu)$, $\mu(\mu^{\dagger}(x))=0$. But  as $\mu\circ \mu^{\dagger}$ is one-to-one, $x=0$, that is, $J(H,\mu)=(0)$. Finally, let us assume that $(H,\mu)$ is semisimple. By the way $G(H,\mu)$ is an orthogonal basis of $(H,\mu)$ according to Lemma~\ref{lem:orthogonal}. Let $u=\sum_{x\in G(H,\mu)}u_x x$ be an arbitrary element of $H$ with $u_x=\frac{1}{\|x\|^2}\langle u,x\rangle$. Then, $\mu^{\dagger}(u)=\sum_{x\in G(H,\mu)}u_x x\otimes x$, and thus $\mu^{\dagger}(u)=0\Leftrightarrow u=0$, that is, $\mu^{\dagger}$ is one-to-one.
\end{proof}

{\begin{remark}
The last statement of Theorem~\ref{thm:semisimplicityforFrob} answers by the affirmative  the main question of~\cite{Heunen}, that is, are all special Hilbertian Frobenius semigroups semisimple? 
\end{remark}}

In the result below are used the notations from Section~\ref{sec:weighted_hilb}. It shows that the weighted Hilbert spaces are  the only semisimple Hilbertian Frobenius semigroups, up to unitary isomorphisms. 
\begin{proposition}\label{prop:onlyoneexample}
Let $(H,\mu)$ be a Hilbertian Frobenius semigroup. Then, $J(H,\mu)^{\perp}\simeq (\ell^2_{w_{(H,\mu)}}(G(H,\mu)),\mu_{G(H,\mu)})$ (unitarily so), where $w_{(H,\mu)}\colon G(H,\mu)\to [\frac{1}{\|\mu\|_{\mathsf{op}}^2},+\infty[$, $x\mapsto \frac{1}{\|x\|^2}$. (If $G(H,\mu)=\emptyset$, then $w_{(H,\mu)}$ stands for the empty map.)
\end{proposition}
\begin{proof}
It suffices to prove the result for $(H,\mu)$ semisimple. Let us consider the unitary transformation $\Lambda\colon H\to \ell^2_{w_{(H,\mu)}}(G(H,\mu))$  given by $\Lambda(\frac{x}{\|x\|}):=\frac{1}{w_{(H,\mu)}(x)^{\frac{1}{2}}}\delta_x=\|x\|\delta_x$, $x\in G(H,\mu)$. It is a matter of simple verification to prove that $\Lambda$ is an isomorphism of semigroups.\end{proof}

\section{Some direct consequences}\label{sec:directcons}

\subsection{The finite-dimensional case}

The following result explains why every finite-dimensional commutative Hilbertian Frobenius monoid is automatically semisimple (\cite{Heunen,Coecke}). 
\begin{corollary}\label{cor:semisimplicityforFrob1}
Let $(H,\mu)$ be a finite-dimensional commutative Hilbertian Frobenius semigroup. The following assertions are equivalent.
\begin{enumerate}
\item $(H,\mu)$ has a unit.
\item $\mu$ is onto.
\item $\mu^{\dagger}$ is one-to-one.
\item $(H,\mu)$ is semisimple.
\end{enumerate}
\end{corollary}
\begin{proof}
The last three assertions are already given by Theorem~\ref{thm:semisimplicityforFrob} in view of finite dimensionality. If $(H,\mu)$ has a unit, then of course $\mu$ is onto. Conversely, assuming $\mu$ onto, by~\cite[Corollary~3.3, p.~47]{Dixon} $(H,\mu)$ is unital. 
\end{proof}

By Proposition~\ref{prop:unital} and Theorem~\ref{thm:semisimplicityforFrob} one has 
\begin{corollary}\label{cor:semisimplicityforFrob2}(\cite[Lemma~3, p.~6]{Heunen})
Let $(H,\mu)$ be a commutative special Hilbertian Frobenius semigroup. It is unital if, and only if, it is finite-dimensional.
\end{corollary}


\subsection{A dictionary of bases}

Let $H\xrightarrow{f}K$ be a bounded linear map between Hilbert spaces. $f$ is a {\em partial isometry} if $f\circ f^{\dagger}\circ f=f$. Actually $f$ is a partial isometry if, and only if, $f^{\dagger}$ is so (\cite[pp.~401--402]{Kubrusly2ndEd}). 

By a {\em structure} of a Hilbertian semigroup of some specific kind on a given Hilbert space $H$ is meant a bounded linear map $\mu\colon H\hat{\otimes}_2 H\to H$ which makes $(H,\mu)$ a Hilbertian semigroup of the desired kind.  
The following result may be considered as an extension of the summary~\cite[p.~565]{Coecke} to infinite-dimensional spaces.

Call {\em bounded above} (resp. {\em bounded below}) a set  $X$ of a Hilbert space $H$ such that there exists $C > 0$ with $\|x\|\leq C$ (resp. $C\leq \|x\|$) for each $x\in X$.  $C$ is referred to a {\em bound} of $X$. Observe that the empty set is bounded above and below, with any bound $C>0$. 
Recall also that following our terminology (Section~\ref{sec:preliminaries}),  an orthogonal set does not contain $0$.

\begin{theorem}\label{thm:dictionary}
Let $H$ be a Hilbert space. There are  one-one correspondences between 
\begin{enumerate}
\item Non void bounded above   orthogonal sets of $H$ and structures of commutative Hilbertian Frobenius semigroups on $H$ with a non-zero multiplication. 
\item\label{it:2:thm:dictionary} Bounded above orthogonal bases of $H$ and structures of semisimple commutative Hilbertian Frobenius semigroups on $H$. 
\item Non void  orthonormal sets of $H$ and structures of  commutative Hilbertian Frobenius semigroups on $H$, whose comultiplication is a non-zero partial isometry. 
\item Orthonormal bases of $H$ and structures of commutative special Hilbertian Frobenius semigroups. The corresponding semigroups all are  unitarily isomorphic.
\item The empty orthogonal set corresponds to the unique structure of radical Frobenius semigroup on $H$.
\end{enumerate}
\end{theorem}
\begin{proof}
Let $X$ be a non void bounded above orthogonal family of $H$ with bound $C>0$. Let $u,v\in H$. Then, $\sum_{x\in X}\|x\|^2|\langle u,\frac{x}{\|x\|}\rangle|^2\langle v,\frac{x}{\|x\|}\rangle|^2\leq C^2\|u\|^2\|v\|^2$. One thus defines  $m_X\colon H\times H\to H$ by $m_X(u,v):=\sum_{x\in X}\langle u,x\rangle\langle v,x\rangle\frac{x}{\|x\|^2}$. As $m_X(x,x)=\|x\|^2x$, $x\in X$, $m_X$ is non-zero. One notices that $m_X(u,v)=0$ whenever $u\in X^{\perp}$ or $v\in X^{\perp}$. $m_X$ is a weak Hilbert-Schmidt mapping since $\sum_{x,y\in X}|\langle m_X(\frac{x}{\|x\|},\frac{y}{\|y\|}),u\rangle|^2\leq C^2\|u\|^2$, $u\in H$. Let $\mu_X\colon H\hat{\otimes}_2 H\to H$ be the unique linear extension of $m_X$, that is, $\mu_X(u\otimes v)=\sum_{x\in X}\langle u,x\rangle\langle v,x\rangle\frac{x}{\|x\|^2}$. $(H,\mu_X)$ is of course a Hilbertian semigroup, with a non-zero multiplication. As for $u\in H$, $\mu_X(\frac{x}{\|x\|}\otimes \frac{y}{\|y\|})=\delta_{x,y}x$, $x,y\in X$, and $\langle \mu_X^{\dagger}(u),v\otimes w\rangle=\langle u,\mu_X(v\otimes w)\rangle=0$ for $v\in X^{\perp}$ or $w\in X^{\perp}$, it follows that $\mu_X^{\dagger}(u)=\sum_{x\in X}\langle u,x\rangle \frac{x}{\|x\|}\otimes\frac{x}{\|x\|}$ from which  one sees that $G(H,\mu_X)=X$.  
 Moreover $\mu^{\dagger}_X(\mu_X(u\otimes v))=\mu^{\dagger}_X(\sum_{x\in X}\langle u,x\rangle\langle v,x\rangle \frac{x}{\|x\|^2})=\sum_{x\in X}\langle u,x\rangle\langle v,x\rangle \frac{x}{\|x\|}\otimes\frac{x}{\|x\|}$ while $(id\hat{\otimes}_2\mu_X)(\alpha_{H,H,H}(\mu_X^{\dagger}(u)\otimes v))=(id\otimes \mu_X)(\sum_{x\in X}\langle u,x\rangle \frac{x}{\|x\|}\otimes (\frac{x}{\|x\|}\otimes v))=\sum_{x\in X}\langle u,x\rangle\frac{x}{\|x\|}\otimes\mu_X(\frac{x}{\|x\|}\otimes v)=\sum_{x\in X}\langle u,x\rangle\langle v,x\rangle \frac{x}{\|x\|}\otimes \frac{x}{\|x\|}$ as $\mu(\frac{x}{\|x\|}\otimes v)=\langle v,x\rangle \frac{x}{\|x\|}$. This proves that $(H,\mu_X)$ is Frobenius.  By the way, if $X$ is an orthogonal basis, then $(H,\mu_X)$ is a semisimple commutative Hilbertian Frobenius semigroup. 
 
 Conversely, if $(H,\mu)$ is a (semisimple) commutative Hilbertian Frobenius semigroup with a non-zero multiplication, thus $(H,\mu)$ is not radical, then $G(H,\mu)$ is a non-void bounded orthogonal family (basis), with bound $\|\mu\|_{\mathsf{op}}>0$ by Lemmas~\ref{lem:orthogonal} and~\ref{lem:preidem}. It is easily checked that $\mu_{G(H,\mu)}=\mu$.  Therefore the first two statements of the theorem are proved. 

The third and fourth statements are proved similarly by considering orthonormal families (bases) rather than orthogonal families (bases), and by the following discussion. If $X$ is an orthonormal family, then for $x\in X$, $\mu_X^{\dagger}(\mu_X(\mu_X^{\dagger}(x)))=x\otimes x=\mu^{\dagger}_X(x)$ so that $\mu_X^{\dagger}$ is indeed a partial isometry since also for $u\in X^{\perp}$, $\mu^{\dagger}(u)=0=\mu_X^{\dagger}(\mu_X(\mu_X^{\dagger}(u)))$. Of course if $X\not=\emptyset$, then $\mu_X\not=0$ and thus so is $\mu_X^{\dagger}$.

Conversely assuming that $(H,\mu)$ is a commutative Hilbertian Frobenius semigroup with $\mu^{\dagger}$ (or $\mu$) a non-zero partial isometry, then for each $x\in G(H,\mu)$, $\|x\|^2x\otimes x=\mu^{\dagger}(x^2)=\mu^{\dagger}(\mu(\mu^{\dagger}(x)))=\mu^{\dagger}(x)=x\otimes x$, so that $\|x\|=1$, and $G(H,\mu)$ is indeed an orthonormal family. 

Let $X,Y$ be two orthonormal bases of $H$. Let $X\xrightarrow{\pi}Y$ be a bijection. Then it is easily check that $(H,\mu_X)\xrightarrow{\Pi}(H,\mu_Y)$ is a semigroup isomorphism where $\Pi(x):=\pi(x)$, $x\in X$. 

The last statement is obvious.
\end{proof}

Due to the Lemma below,  one may substitute in Theorem~\ref{thm:dictionary}, ``bounded above'' by ``bounded below'' and the resulting statements are still valid.
\begin{lemma}\label{lem:bounded_above_and_bounded_below}
Let $H$ be a Hilbert space. There is a one-one correspondence $\Theta$ between the  bounded above and the bounded below orthogonal sets of $H$, which preserves the cardinality (in particular it  sends $\emptyset$ to itself). More precisely for each  bounded above orthogonal set $X$ of $H$, $X\simeq \Theta(X)$ under $x\mapsto \frac{x}{\|x\|^2}$ and the corresponding orthonormal sets $\{\, \frac{x}{\|x\|}\colon x\in X\,\}$ and $\{\, \frac{x}{\|x\|}\colon x\in \Theta(X)\,\}$  are equal.
\end{lemma}
\begin{proof}
Given an orthogonal basis $X$ of $H$, let $\Theta(X):=\{\, \frac{x}{\|x\|^2}\colon x\in X\,\}$ is  an orthogonal family too. Of course $\Theta(\emptyset)=\emptyset$. Note that $X\simeq \Theta(X)$ under $x\mapsto \frac{x}{\|x\|^2}$. Assume that $X\not=\emptyset$. Since for each $x\in X$, $\|\frac{x}{\|x\|^2}\|=\frac{1}{\|x\|}$, $X$ is bounded above (resp. below) by $C$ if, and only if, $\Theta(X)$ is bounded below (resp. above) by $\frac{1}{C}$. Observe that in any case, $\Theta(\Theta(X))=X$ so the required bijection is obtained. It is clear that for each  orthogonal set $X$ of $H$, $\{\, \frac{x}{\|x\|}\colon x\in X\,\}$ and $\{\, \frac{x}{\|x\|}\colon x\in \Theta(X)\,\}$  are equal. 
\end{proof}

\subsection{$H^*$-algebras}

It is possible to characterize Hilbertian Frobenius semigroups using Ambrose's concept of $H^*$-algebras~\cite{Ambrose} or conversely to characterize commutative $H^*$-algebras in an involution-free way. Let $(H,\mu)$ be a (commutative) Hilbertian algebra and let $u\in H$. By a {\em $H^*$-adjoint} of $u$ is meant a member $v$ of $H$ such that $M_u^{\dagger}=M_v$, that is, for every $w,w'\in H$, $\langle uw,w'\rangle=\langle w,vw'\rangle$. $(H,\mu)$ is a {\em Hilbertian $H^*$-algebra} when every element of $H$ has a $H^*$-adjoint. 
\begin{proposition}\label{prop:hstar}
Let $(H,\mu)$ be a Hilbertian algebra. $(H,\mu)$ is a Hilbertian $H^*$-algebra if, and only if, $(H,\mu)$ a Hilbertian Frobenius semigroup.
\end{proposition}
\begin{proof}
For each $u,w\in H$, $(M_u\hat{\otimes}_2 id)(\mu^{\dagger}(w))=(\mu\hat{\otimes}_2 id)(\alpha^{-1}((id\hat{\otimes}_2 \mu^{\dagger})(u\otimes w))$. Indeed let $X$ be an orthonormal basis of $H$. Then, $\mu^{\dagger}(w)=
\sum_{x,y\in X}\langle \mu^{\dagger}(w),x\otimes y\rangle x\otimes y$. Therefore, $(\mu\hat{\otimes}_2 id)(\alpha^{-1}((id\hat{\otimes}_2 \mu^{\dagger})(u\otimes w))=\sum_{x,y\in X}\langle\mu^{\dagger}(w),x\otimes y\rangle (ux)\otimes y=(M_u\hat{\otimes}_2 id)(\mu^{\dagger}(w))$. 

Let us assume that $(H,\mu)$ is a Hilbertian $H^*$-algebra. Let $u\in H$ and let $v$ be a $H^*$-adjoint of $u$. Let $w,w',w''\in H$. One has 
\begin{equation}
\begin{array}{lll}
\langle \mu^{\dagger}(uw),w'\otimes w''\rangle&=&\langle uw,w'w''\rangle\\
&=&\langle w,vw'w''\rangle\\
&=&\langle w,M_v(w')w''\rangle\\
&=&\langle w,\mu((M_v\hat{\otimes}_2 id)(w'\otimes w''))\rangle\\
&=&\langle\mu^{\dagger}(w),(M_v\hat{\otimes}_2 id)(w'\otimes w'')\rangle\\
&=&\langle (M_v^{\dagger}\hat{\otimes}_2 id)(\mu^{\dagger}(w)),w'\otimes w''\rangle\\
&=&\langle (M_u\hat{\otimes}_2 id)(\mu^{\dagger}(w)),w'\otimes w''\rangle\\
&=&\langle (\mu\hat{\otimes}_2 id)(\alpha^{-1}((id\hat{\otimes}_2 \mu^{\dagger})(u\otimes w)),w'\otimes w''\rangle.\\
&&\mbox{(by the above)}
\end{array}
\end{equation}
Therefore $\mu^{\dagger}(uw)- (\mu\hat{\otimes}_2 id)(\alpha^{-1}((id\hat{\otimes}_2 \mu^{\dagger})(u\otimes w))\in (H\otimes_{\mathbb{C}}H)^{\perp}=0$, and $(H,\mu)$ satisfies the Frobenius condition. 

Let us assume that $(H,\mu)$ is a Hilbertian Frobenius semigroup. Let $u\in H$. Let us define $u^*:=\sum_{x\in G(H,\mu)}\frac{1}{\|x\|^2}\langle x,u\rangle x$. Of course $p_J(u^*)=0$. Now, let $u,v,w\in H$. Then, 
\begin{equation}
\begin{array}{lll}
\langle uv,w\rangle&=&\langle p_{J^{\perp}}(u)p_{J^{\perp}}(v),w\rangle\\
&&\mbox{(according to Remark~\ref{rem:description_of_HilbFrob})}\\
&=&\langle p_{J^{\perp}}(u)p_{J^{\perp}}(v),p_{J^{\perp}}(w)\rangle\\
&&\mbox{(as $J^{\perp}$ is a subalgebra by Theorem~\ref{thm:stHFs})}\\
&=&\sum_{x\in G(H,\mu)}\frac{1}{\|x\|^2}\langle u,x\rangle\langle v,x\rangle\overline{\langle w,x\rangle}\\
&=&\sum_{x\in G(H,\mu)}\frac{1}{\|x\|^2}\langle v,x\rangle\overline{\langle x,u\rangle\langle w,x\rangle}\\
&=&\langle p_{J^{\perp}}(v),u^*p_{J^{\perp}}(w)\rangle\\
&=&\langle p_{J^{\perp}}(v),u^*w\rangle\\
&=&\langle v,u^*w\rangle.\\
&&\mbox{(as $u^*w\in J^{\perp}$)}
\end{array}
\end{equation}
Therefore $(H,\mu)$ is a Hilbertian $H^*$-algebra.
\end{proof}


\begin{remark}
The fact that a Hilbertian $H^*$-algebra is a Hilbertian Frobenius semigroup was already noticed in~\cite[Lemma~6, p.~9]{Heunen}.
\end{remark}

\begin{corollary}\label{cor:bij_ideals_subcoalg}
Let $(H,\mu)$ be a Hilbertian Frobenius semigroup. 
$(-)^{\perp}$ provides an order-reversing involution on the set of closed ideals of $(H,\mu)$. In particular a closed subspace of $H$ is an ideal if, and only if, it is a subcoalgebra. 
\end{corollary}
\begin{proof}
The first statement follows from Proposition~\ref{prop:hstar} and the existence of $H^*$-adjoints: let $I$ be an ideal and $u\in I^{\perp}$. Let $x\in H$  with $H^*$-adjoint $y$, and let $w\in I$. Then, $\langle ux,w\rangle=\langle u,yw\rangle=0$. Whence $I^{\perp}$ is an ideal too. The second statement is due to~\cite[Lemmas~14 and~15, p.~14]{Poinsot2019} which jointly assert  that $I$ is a closed ideal if, and only if, $I^{\perp}$ is a closed subcoalgebra.  
\end{proof}

\subsection{Approximate (co)units}
Any semisimple commutative Hilbertian Frobenius semigroup $(H,\mu)$ has an approximate unit. (This was already noticed in~\cite{Heunen} for $H$ separable and $\mu$ isometric, under the assumption, redundant by Theorem~\ref{thm:semisimplicityforFrob}, of semisimplicity of $(H,\mu)$.)  

To see that, let us first denote by $\mathfrak{P}_{\mathsf{fin}}(X)$ the set of all finite subsets of some set $X$ directed under inclusion. Let $(H,\mu)$ be a commutative Hilbertian Frobenius semigroup.  For $F\in \mathfrak{P}_{\mathsf{fin}}(G(H,\mu))$, let $e_F:=\sum_{x\in F}\frac{x}{\|x\|^2}$. Then, for each $u\in H$, $u$ is the limit of the directed net $(ue_F)_F=(\sum_{x\in F}\langle u,\frac{x}{\|x\|}\rangle \frac{x}{\|x\|})_F$, that is, for each $\epsilon >0$, there exists $F_{\epsilon}\in \mathfrak{P}_{\mathsf{fin}}(G(H,\mu))$ such that for all $F\in\mathfrak{P}_{\mathsf{fin}}(G(H,\mu))$ with $F_{\epsilon}\subseteq F$, $\|p_{J^{\perp}}(u)-ue_F\|<\epsilon$. (This is clear since $p_{J^{\perp}}(u)$ is precisely the sum of the summable family $(\langle u,\frac{x}{\|x\|}\rangle \frac{x}{\|x\|})_{x\in G(H,\mu)}$.) 

More generally let $(H,\mu)$ be a Hilbertian semigroup and let  $(e_{\lambda})_{\lambda\in \Lambda}$ be a directed net on $H$, that is, $\Lambda$ is a directed set and $e_\lambda\in H$, $\lambda\in \Lambda$. $(e_\lambda)_\lambda$ is an {\em approximate unit} when $ue_\lambda\to u$ in the norm topology, for each $u\in H$. By its very definition the existence of an approximate unit forces the multiplication $\mu$ to have a dense range. 

Let $(H,\mu)$ be a Hilbertian semigroup and let $(H\xrightarrow{\epsilon_\lambda}\mathbb{C})_{\lambda\in \Lambda}$ where $\Lambda$ is a directed set. Call $(\epsilon_\lambda)_\lambda$ an {\em approximate counit} when for each $u\in H$, $\|(id\otimes \epsilon_{\lambda})(\mu^{\dagger}(u))-u\otimes 1\|_{H\hat{\otimes}_2 H}\to 0$. The existence of such an approximate counit forces $\mu^{\dagger}$ to be one-to-one. Conversely given a Hilbertian Frobenius semigroup $(H,\mu)$, let $\epsilon_F(u):=\sum_{x\in F}\langle u,\frac{x}{\|x\|}\rangle$ for $F\in\mathfrak{P}_{\mathsf{fin}}(G(H,\mu))$ and $u\in H$. Then, $(id\hat{\otimes}_2\epsilon_F)(\mu^{\dagger}(u))=\sum_{x\in F}\frac{x}{\|x\|}\otimes \langle u,\frac{x}{\|x\|}\rangle\to p_{J^{\perp}}(u)\otimes 1$. Whence if $(H,\mu)$ is semisimple, then $\epsilon_F$ is an approximate counit. 

All of this may be combined as a corollary of Theorem~\ref{thm:semisimplicityforFrob}.
\begin{corollary}\label{cor:semisimplicityforFrob}
A Hilbertian Frobenius semigroup is semisimple if, and only if, it has an approximate unit if, and only if, it has an approximate counit. In particular each special Hilbertian Frobenius semigroup has an approximate unit and an approximate counit. 
\end{corollary}
%
%
%

\section{Reformulation of the Structure Theorem as an equivalences of categories}\label{sec:reformulation}

In this section is showed that the splitting of an Hilbertian Frobenius semigroup into the orthogonal direct sum of a semisimple and a radical Hilbertian Frobenius semigroups may in fact be recasted into an  equivalence between ${}_c\mathbf{FrobSem}(\mathbb{Hilb})$ and the product category ${}_{\mathsf{semisimple},c}\mathbf{FrobSem}(\mathbb{Hilb})\times\mathbf{Hilb}$. 

Let $(H,\mu)$ be a commutative Hilbertian Frobenius algebra. The closure $\overline{ran(\mu)}$ of the range of $\mu$ is also equal to $\overline{H^2}$, with $H^2:=\langle xy\colon x,y\in H\rangle$.  Now $(H^2)^{\perp}=A(H,\mu)$ as it follows easily from the existence of $H^*$-adjoints. Consequently, $\overline{ran(\mu)}=J^{\perp}$ by Proposition~\ref{prop:normality}. One now defines $P(H,\mu):=(J^{\perp},\mu_{|_{J^{\perp}}})$ (see Corollary~\ref{cor:stHFs}).

\begin{proposition}
$(H,\mu)\mapsto P(H,\mu)$ extends to a functor $P$ from ${}_c\mathbf{FrobSem}(\mathbb{Hilb})$ to ${}_{\mathsf{semisimple},c}\mathbf{FrobSem}(\mathbb{Hilb})$ which is a right adjoint left inverse of the full embedding functor $E\colon {}_{\mathsf{semisimple},c}\mathbf{FrobSem}(\mathbb{Hilb})\hookrightarrow {}_c\mathbf{FromSem}(\mathbb{Hilb})$. 
\end{proposition}
\begin{proof}
Let $(H,\mu)\xrightarrow{f}(K,\gamma)$ be a semigroup map between Hilbertian Frobenius semigroups. As $f(H^2)\subseteq K^2$, it follows that $f(\overline{H^2})\subseteq \overline{f(H^2)}\subseteq \overline{K^2}$. Whence $P(H,\mu)\xrightarrow{f}P(K,\gamma)$ is defined as the co-restriction of $f$. In particular, $P(f)=\pi_{P(K,\gamma)}\circ f\circ i_{P(H,\mu)}$ and thus $P(f)$ is bounded. By Corollary~\ref{cor:stHFs}, $P(H,\mu)=J^{\perp}$ is a semisimple Hilbertian Frobenius semigroup, under the co-restriction of $\mu$. $i_{P(H,\mu)}$ is a semigroup morphism since $J(H,\mu)^{\perp}$ is a subalgebra, and $\pi_{P(K,\gamma)}=i_{P(K,\gamma)}^{\dagger}$ is a semigroup morphism too since $J(K,\gamma)^{\perp}$ is a subcoalgebra, so that $P(f)$ is indeed a morphism of semigroups, and thus one obtains the desired functor. $P$ is of course a left inverse of $E$. 

Now let $(H,\mu)$ be a semisimple Hilbertian Frobenius semigroup and let $(K,\gamma)$ be a  Hilbertian Frobenius semigroup. Let $(H,\mu)\xrightarrow{f}P(K,\gamma)$ be a morphism of semigroups. Define $f^{\sharp}:=(H,\mu)\xrightarrow{f}P(K,\gamma)\xrightarrow{i_{P(K,\gamma)}}(K,\gamma)$. Then $P(f^{\sharp})=f$ as $i_{P(K,\gamma)}\circ P(f^{\sharp})=f^{\sharp}\circ i_{P(H,\mu)}=f^{\sharp}=i_{P(K,\gamma)}\circ f$ (as $P(H,\mu)=(H,\mu)$) and  $i_{P(K,\gamma)}$ is a monomorphism. Now let $(H,\mu)=P(H,\mu)\xrightarrow{g}(K,\gamma)$ such that $P(g)=f$. Then  $f^{\sharp}=i_{P(K,\gamma)}\circ f=i_{P(K,\gamma)}\circ P(g)=g\circ i_{P(H,\mu)}=g$.
\end{proof}


One may also define a functor $J\colon {}_c\mathbf{FrobSem}(\mathbb{Hilb})\to\mathbf{Hilb}$ as follows. Let $f\colon (H,\mu)\to (K,\gamma)$ be a semigroup map. Then, $f^{\dagger}\colon (K,\gamma^{\dagger})\to (H,\mu^{\dagger})$ is a cosemigroup map, and thus $f^{\dagger}(G(K,\gamma))\subseteq G(H,\mu)\cup\{\, 0\,\}$. By linearity and continuity, $f^{\dagger}(J(K,\gamma)^{\perp})\subseteq J(H,\mu)^{\perp}$ and thus $f(J(H,\mu))\subseteq J(K,\gamma)$ by Lemma~\ref{lem:restriction}. Then let $J(f)$ be the co-restriction of $f$ thus obtained. Clearly this provides a functor $J\colon {}_c\mathbf{FrobSem}(\mathbb{Hilb})\to\mathbf{Hilb}$. 

In the opposite direction let $T\colon \mathbf{Hilb}\to {}_c\mathbf{FrobSem}(\mathbb{Hilb})$ be the full embedding functor, $T(H\xrightarrow{f}K):=(H,0)\xrightarrow{f}(K,0)$. 
\begin{lemma}\label{lem:trivial_is_equivalence}
The full subcategory ${}_{\mathsf{radical},c}\mathbf{FrobSem}(\mathbb{Hilb})$ of ${}_c\mathbf{FrobSem}(\mathbb{Hilb})$ spanned by the  radical commutative Hilbertian Frobenius algebras, is isomorphic to $\mathbf{Hilb}$. 
\end{lemma}
\begin{proof}
The co-restriction $\mathbf{Hilb}\xrightarrow{T}{}_{\mathsf{radical},c}\mathbf{FrobSem}(\mathbb{Hilb})$ of $T$ is the inverse of the obvious forgetful functor ${}_{\mathsf{radical},c}\mathbf{FrobSem}(\mathbb{Hilb})\xrightarrow{|-|}\mathbf{Hilb}$.
\end{proof}

\begin{remark}
The functor $P$ should not be confused with the functor  $J(-)^{\perp}$ occurring in~\cite[Theorem~24, p.~17]{Poinsot2019}. The functor $J$ as introduced above corresponds to $|-|\circ J$, where $J$ is as in~\cite[Theorem~24, p.~17]{Poinsot2019} and $|-|$ is the forgetful functor from radical Frobenius algebras to Hilbert spaces.  
\end{remark}

In the two results below one identifies external and internal orthogonal direct sums.

Let $(H,\mu)$ and $(K,\gamma)$ be two commutative Hilbertian algebras. By additivity of $\hat{\otimes}_2$, $(H\oplus_2 K)\hat{\otimes}_2 (H\oplus_2 K)$ has the following coproduct presentation.
\begin{equation}
\xymatrix{
H\hat{\otimes}_2 H \ar[rd]^{i_H\hat{\otimes}_2 i_H}&&H\hat{\otimes}_2 K\ar[ld]_{i_H\hat{\otimes}_2 i_K}\\
&(H\oplus_2 K)\hat{\otimes}_2 (H\oplus_2 K)&\\
\ar[ru]_{i_{J}\hat{\otimes}_2 i_H}K\hat{\otimes}_2 H &&K\hat{\otimes}_2 K \ar[lu]^{i_K\hat{\otimes}_2 i_K}
}
\end{equation}
\begin{proposition}\label{prop:construction_of_Frob}
Let $(H,\mu)$ and $(K,\gamma)$ be Hilbertian semigroups. Let $\rho\colon (H\oplus_2 K)\hat{\otimes}_2 (H\oplus_2 K)\to H\oplus_2 K$ be defined by $\rho\circ (i_H\hat{\otimes}_2 i_H):=i_H\circ \mu$, $\rho\circ (i_K\hat{\otimes}_2 i_K)=i_K\circ \gamma$ and $\rho\circ ((i_{H}\hat{\otimes}_2 i_K)\oplus_2(i_{K}\hat{\otimes}_2i_H))=0$. Then, $(H\oplus_2 K,\rho)$ is a Hilbertian algebra, $(H,\mu)\xrightarrow{i_H}(H\oplus_2 K,\rho)\xleftarrow{i_K}(K,\gamma)$ are morphisms of semigroups, and $\rho^{\dagger}\circ i_H=(i_{H}\hat{\otimes}_2 i_H)\circ \mu^{\dagger}$, $\rho^{\dagger}\circ i_K=(i_K\hat{\otimes}_2 i_K)\circ \gamma^{\dagger}$, that is, $(H,\mu^{\dagger})\xrightarrow{i_H}(H\oplus_2 K,\rho^{\dagger})\xleftarrow{i_K}(K,\gamma^{\dagger})$ are morphisms of cosemigroups as well.

Moreover if both $(H,\mu)$ and $(K,\gamma)$ are Frobenius, then so is $(H\oplus_2 K,\rho)$, and in this case, $J(H\oplus_2 K,\rho)=J(H,\mu)\oplus_2 J(K,\gamma)$ and $J(H\oplus_2 H,\rho)^{\perp}=J(H,\mu)^{\perp}\oplus_2 J(K,\gamma)^{\perp}$. In particular, if $(H,\mu)$ is semisimple and $(K,\gamma)$ is radical, then $J(H\oplus_2 K,\rho)=K$ and $J(H\oplus_2 K,\rho)^{\dagger}=H$.
\end{proposition}
\begin{proof}
That $(H\oplus_2 K,\rho)$ is indeed a Hilbertian algebra is easily checked. It is then clear, from the very definition of $\rho$, that  $(H,\mu)\xrightarrow{i_H}(H\oplus_2 K,\rho)\xleftarrow{i_K}(K,\gamma)$ are morphisms of semigroups.

Let $x\in H$, $u,u'\in H$, $v,v'\in K$. Then, 
\begin{equation}
\begin{array}{lll}
\langle \rho^{\dagger}(x),(u+v)\otimes (u'+v')\rangle&=&\langle \rho^{\dagger}(x),u\otimes u'+u\otimes v'+v\otimes u'+v\otimes v'\rangle\\
&=&\langle x,\mu(u\otimes u')+\gamma(v\otimes v')\rangle\\
&=&\langle x,\mu(u\otimes u')\rangle\\
&&\mbox{(since $\gamma(v,v')\in K$)}\\
&=&\langle \mu^{\dagger}(x),u\otimes u'\rangle\\
&=&\langle i_{H\hat{\otimes}_2 H}(\mu^{\dagger}(x)),u\otimes u'+u\otimes v'+v\otimes u'+v\otimes v'\rangle\\
&=&\langle i_{H\hat{\otimes}_2 H}(\mu^{\dagger}(x)),(u+v)\otimes(u'+v')\rangle.
\end{array}
\end{equation}
As a matter of fact, $\rho^{\dagger}(x)\in ((H\hat{\otimes}_2K)\oplus_2 (K\hat{\otimes}_2 H)\oplus_2 (K\hat{\otimes}_2 K))^{\perp}=H\hat{\otimes}_2 H$. Likewise $\rho^{\dagger}(x)\in K\hat{\otimes}_2 K$ for each $x\in K$. Consequently, $\rho^{\dagger}\circ i_H=(i_{H}\hat{\otimes}_2 i_H)\circ \mu^{\dagger}$, $\rho^{\dagger}\circ i_K=(i_K\hat{\otimes}_2 i_K)\circ \gamma^{\dagger}$,  and thus $(H,\mu^{\dagger})\xrightarrow{i_H}(H\oplus_2 K,\rho^{\dagger})\xleftarrow{i_K}(K,\gamma^{\dagger})$ are morphisms of cosemigroups.

Let us now assume that both $(H,\mu)$ and $(K,\gamma)$ are Frobenius. Let $u,u'\in H$, $v,v'\in K$. 
\begin{equation}
\begin{array}{lll}
\rho^{\dagger}(\rho((u+v)\otimes (u'+v')))&=&\rho^{\dagger}(\mu(u\otimes u')+\gamma(v\otimes v'))\\
&=&\rho^{\dagger}(\mu(u\otimes u'))+\rho^{\dagger}(\gamma(v\otimes v'))\\
&=&\mu^{\dagger}(\mu(u\otimes u'))+\gamma^{\dagger}(\gamma(v\otimes v'))\\
&=&(id\hat{\otimes}_2 \mu)(\alpha((\mu^{\dagger}\hat{\otimes}id)(u\otimes u')))\\
&+&(id\hat{\otimes}_2 \gamma)(\alpha((\gamma^{\dagger}\hat{\otimes}id)(u\otimes u')))\\
&=&(id\hat{\otimes}_2\rho)(\alpha((\mu^{\dagger}\hat{\otimes}id)(u\otimes u'))+\alpha((\gamma^{\dagger}\hat{\otimes}id)(u\otimes u')))\\
&=&(id\hat{\otimes}_2\rho)(\alpha(((\mu^{\dagger}+\gamma^{\dagger})\hat{\otimes}_2 id)(u\otimes u'+v\otimes v')))\\
&=&(id\hat{\otimes}_2\rho)(\alpha(((\mu^{\dagger}+\gamma^{\dagger})\hat{\otimes}_2 id)((u+v)\hat{\otimes}_2(u'+v'))))\\
&&\mbox{(as $\gamma((H\hat{\otimes}_2K)\oplus_2 (K\hat{\otimes}_2 H))=0$)}\\
&=&(id\hat{\otimes}_2 \rho)(\alpha((\rho^{\dagger}\hat{\otimes}_2 id)((u+v)\hat{\otimes}_2(u'+v')))).\\
\end{array}
\end{equation}
Consequently, $(H\oplus_2 K,\rho)$ is Frobenius as well. Let $u+v\in A(H\oplus_2 K,\rho)$, $u\in H$, $v\in K$. Then $0=\rho_{\mathsf{bil}}(u+v,u'+v')=\mu_{\mathsf{bil}}(u,u')+\gamma_{\mathsf{bil}}(v,v')$. In particular, $u\in A(H,\mu)=J(H,\mu)$ and $v\in A(K,\gamma)=J(K,\gamma)$, and $u+v\in J(H,\mu)\oplus_2 J(K\gamma)$. Conversely, let $u\in J(H,\mu)$, $v\in J(K,\gamma)$. Then, $u+v\in A(H\oplus_2 K,\rho)$ since $\rho_{\mathsf{bil}}(u+v,u'+v')=\mu_{\mathsf{bil}}(u,u')+\gamma_{\mathsf{bil}}(v,v')=0$ for all $u'\in H$, $v'\in K$. Thus, $J(H\oplus_2 K,\rho)=A(H\oplus_2 K,\rho)=J(H,\mu)\oplus_2 J(K,\gamma)$. As a result, $J(H\oplus_2 K,\rho)^{\perp}=(J(H,\mu)\oplus_2 J(K,\gamma))^{\perp}=J(H,\mu)^{\perp}\hat{\oplus}_2 J(K,\gamma)^{\perp}$. The last assertion is immediate.
\end{proof}

\begin{remark}\label{rem:orthogonal_sum_functor}
The construction $O((H,\mu),(K,\gamma)):=(H\oplus_2 K,\rho)$ extends to a functor $O\colon {}_c\mathbf{Sem}(\mathbb{Hilb})\times{}_{c}\mathbf{Sem}(\mathbb{Hilb})\to {}_c\mathbf{Sem}(\mathbb{Hilb})$ as follows: let $(H,\mu)\xrightarrow{f}(H',\mu')$ and $(K,\gamma)\xrightarrow{g}(K',\gamma')$ be morphisms of semigroups, let $f\oplus_2 g\colon H\oplus_2 K\to H'\oplus K'$, be given by $(f\oplus_2 g)(u+v)=f(u)+g(v)$, $u\in H$, $v\in K$. Then, $O(f,g):=(H\oplus_2 K,\rho)\xrightarrow{f\oplus_2 g}(H'\oplus_2 K',\rho')$ is a morphism of semigroups. 
One also has a functor  $O\colon {}_c\mathbf{FrobSem}(\mathbb{Hilb})\times{}_{c}\mathbf{FrobSem}(\mathbb{Hilb})\to {}_c\mathbf{FrobSem}(\mathbb{Hilb})$ given by co-restriction.
\end{remark}

\begin{lemma}\label{lem:faithful}
Let $f\colon(H,\mu)\to (K,\gamma)$ be a ${}_c\mathbf{FrobSem}(\mathbb{Hilb})$-morphism, then $f=P(f){\oplus}_2 J(f)$. 
\end{lemma}

\begin{proposition}\label{prop:structuretheoremrevisited1}
${}_c\mathbf{FrobSem}(\mathbb{Hilb})$ is equivalent to ${}_{\mathsf{semisimple},c}\mathbf{FrobSem}(\mathbb{Hilb})\times \mathbf{Hilb}$.
\end{proposition}
\begin{proof}
It suffices to prove that the functor $\langle P,J\rangle$ from ${}_c\mathbf{FrobSem}(\mathbb{Hilb})$ to ${}_{\mathsf{semisimple},c}\mathbf{FrobSem}(\mathbb{Hilb})\times \mathbf{Hilb}$, $(H,\mu)\mapsto (P(H,\mu),J(H,\mu))$,  is the required equivalence of categories. Let $((H,\mu),K)$ be an object of ${}_{\mathsf{semisimple},c}\mathbf{FrobSem}(\mathbb{Hilb})\times \mathbf{Hilb}$. Then, $((H,\mu),K)\simeq (P(H\oplus_2 K,\rho),J(H\oplus_2 K,\rho))$ by Proposition~\ref{prop:construction_of_Frob}. By Lemma~\ref{lem:faithful}, $\langle P,V\rangle$ is faithful.  Let $(H,\mu),(K,\gamma)$ be Frobenius semigroups and let $(P(H,\mu)\xrightarrow{f}P(K,\gamma),J(H,\mu)\xrightarrow{g}J(K,\gamma))$ be a pair consisting of a morphism of semigroups and a bounded linear map respectively. Let us consider the bounded linear map $(f\oplus_2 g)\colon H\to K$ given by $(f\oplus_2 g)(u+v)=f(u)+g(v)$, $u\in J(H,\mu)^{\dagger}$, $v\in J(H,\mu)$. Let $u'\in J(H,\mu)^{\dagger}$ and $v'\in J(H,\mu)$. Then, $(f\oplus _2 g)((u+v)(u'+v'))=(f\oplus_2 g)(uu')=f(uu')=f(u)f(u')=(f(u)+g(v))(f(u')+g(v'))=((f\oplus_2 g)(u+v))((f\oplus_2 g)(u'+v'))$. Consequently $(f\oplus_2 g)$ is a semigroup map. Of course, $P(f\oplus_2 g)=f$ and $J(f\oplus_2 g)=g$, so that $\langle P,J\rangle$ is full. 
\end{proof}


\begin{corollary}
The categories ${}_c\mathbf{FrobSem}(\mathbb{FdHilb})$ and ${}_{\mathsf{semisimple},c}\mathbf{FrobSem}(\mathbb{FdHilb})\times\mathbf{FdHilb}$ are equivalent.
\end{corollary}

Let ${}_{\mathsf{partiso},c}\mathbf{FrobSem}(\mathbb{Hilb})$ be  the full subcategory of ${}_c\mathbf{FrobSem}(\mathbb{Hilb})$ spanned by those Hilbertian Frobenius algebras $(H,\mu)$ where $\mu^{\dagger}$ is a partial isometry. Recall also that ${}_{c}^{\dagger}\mathbf{FrobSem}(\mathbb{Hilb})$ is the category of special commutative Hilbertian Frobenius semigroups. Let $(H,\mu)$ be an object of ${}_{\mathsf{partiso},c}\mathbf{FrobSem}(\mathbb{Hilb})$. As $\mu^{\dagger}$ restricts to an isometry from $(\ker \mu^{\dagger})^{\perp}=(ran\ \mu)^{\perp\perp}=J^{\perp}$ to $H\hat{\otimes}_2 H$ (\cite[p.~404]{Kubrusly2ndEd}), it is clear that $P(H,\mu)$ is an object of ${}_c^{\dagger}\mathbf{FrobSem}(\mathbb{Hilb})$. The following result then follows easily.
\begin{corollary}
The equivalence from Proposition~\ref{prop:structuretheoremrevisited1} restricts to an equivalence between 
${}_{\mathsf{partiso},c}\mathbf{FrobSem}(\mathbb{Hilb})$ and ${}_{c}^{\dagger}\mathbf{FrobSem}(\mathbb{Hilb})\times\mathbf{Hilb}$.
\end{corollary}

\section{Equivalences of semisimple Frobenius semigroups and weighted pointed sets}\label{sec:mainequiv}

The one-one correspondence between structures of semisimple Hilbertian Frobenius semigroups on a given Hilbert space and its bounded below (or above) orthogonal bases (Theorem~\ref{thm:dictionary}) may be upgraded to an equivalence between ${}_{\mathsf{semsimple},c}\mathbf{FrobSem}(\mathbb{Hilb})$ and a category of pointed sets.

\subsection{Categories of pointed sets with a weight function}

Let $\mathbf{WSet}_{\bullet}$ be the following category of {\em weighted pointed sets}. Its objects are pointed sets $(X,x_0,\alpha)$ together with a {\em weight function}, i.e., a map $\alpha\colon X\setminus\{\, x_0\,\}\to [C_{\alpha},+\infty[$ for some $C_{\alpha}>0$. A morphism $(X,x_0,\alpha)\xrightarrow{f}(Y,y_0,\beta)$ is a base-point preserving map $(X,x_0)\xrightarrow{f}(Y,y_0)$ such that
\begin{itemize}
\item For each $y\not=y_0$, $|f^{-1}(\{\, y\,\})|<+\infty$,
\item There exists a real number $M_f\geq 0$ such that for all $y\not=y_0$, $\sum_{x\in f^{-1}(\{\, y\,\})}\alpha(x)\leq M_f\beta(y)$.
\end{itemize}
Under the usual composition this indeed forms a category. 

In~\cite{Poinsot2019} is introduced the category $\mathbf{Set}_{\bullet,<+\infty}$ the objects of which are pointed sets $(X,x_0)$ and morphisms $(X,x_0)\xrightarrow{f}(Y,y_0)$ are those base-point preserving maps such that (1) $|f^{-1}(\{\, y\,\})|$ is finite for each $y\not=y_0$ and (2) $\mathsf{B}_f:=\sup_{y\not=y_0}|f^{-1}(\{\, y\,\})|<+\infty$. 
\begin{lemma}\label{lem:fullembedding}
$\mathbf{Set}_{\bullet,<+\infty}$ fully embeds into $\mathbf{WSet}_{\bullet}$.
\end{lemma}
\begin{proof}
The functor $E\colon \mathbf{Set}_{\bullet,<+\infty}\to\mathbf{WSet}_{\bullet}$, $E(X,x_0):=(X,x_0,\mathbf{1})$ with $\mathbf{1}(x)=1$, $x\not=x_0$, $E(f):=f$, is an injective on objects, fully faithful functor. 
\end{proof}

\begin{lemma}\label{lem:precurseur_du_forgetful}
Let $f\in\mathbf{WSet}_{\bullet}((X,x_0,\alpha),(Y,y_0,\beta))$ with $\beta$ bounded above. Then, $f\in \mathbf{Set}_{\bullet,<+\infty}((X,x_0),(Y,y_0))$.
\end{lemma}
\begin{proof}
Let $y\not=y_0$. $|f^{-1}(\{\, y\,\})|\leq \frac{1}{C_\alpha}\sum_{x\in f^{-1}(\{\, y\,\})}\alpha(x) \leq \frac{M_f}{C_{\alpha}}\beta(y)\leq  \frac{M_f}{C_{\alpha}}\sup_{y\not=y_0}\beta(y)$ so that $\sup_{y\not=y_0}|f^{-1}(\{\, y\,\})|\leq  \frac{1}{C_{\alpha}}M_f\sup_{y\not=y_0}\beta(y)<+\infty$.
\end{proof}

Let us define ${}_{\mathsf{bnd}}\mathbf{WSet}_{\bullet}$ (resp. ${}_{\mathsf{unbnd}}\mathbf{WSet}_{\bullet}$) be the full subcategory of $\mathbf{WSet}_{\bullet}$ spanned by those objects $(X,x_0,\alpha)$ with $\alpha$ bounded above (resp. unbounded).

\begin{lemma}\label{lem:equivalence_bounded_normalized}
${}_{\mathsf{bnd}}\mathbf{WSet}_{\bullet}$ is equivalent to $\mathbf{Set}_{\bullet,<+\infty}$. Moreover no object of ${}_{\mathsf{bnd}}\mathbf{WSet}_{\bullet}$ is isomorphic to an object of  ${}_{\mathsf{unbnd}}\mathbf{WSet}_{\bullet}$.
\end{lemma}
\begin{proof}
By Lemma~\ref{lem:precurseur_du_forgetful}, for $(X,x_0,\alpha)\xrightarrow{f}(Y,y_0,\beta)$ with $\alpha,\beta$ bounded above, $(X,x_0)\xrightarrow{f}(Y,y_0)$ is a morphism in $\mathbf{Set}_{\bullet,<+\infty}$. This defines in an obvious way a functor $U\colon {}_{\mathsf{bnd}}\mathbf{WSet}_{\bullet}\to \mathbf{Set}_{\bullet,<+\infty}$. This functor is readily faithful. 

Let $(X,x_0,\alpha),(Y,y_0,\beta)$ with bounded $\alpha,\beta$ and let $(X,x_0)\xrightarrow{f}(Y,y_0)$ in $\mathbf{Set}_{\bullet,<+\infty}$. Then, for $y\not=y_0$, $\sum_{x\in f^{-1}(\{\, y\,\})}\alpha(x)\leq \sup_{x\not=x_0}\alpha(x)\|\mathsf{B}_f\|_{\infty}\leq \sup_{x\not=x_0}\alpha(x)\|\mathsf{B}_f\|_{\infty}\frac{1}{C_\beta}\beta(y)$. Whence $U$ is full. Of course, $U$ is surjective on objects because $U(X,x_0,\mathbf{1})=(X,x_0)$. 

Now, let $(X,x_0,\alpha)$ and $(Y,y_0,\beta)$ with $\alpha$ bounded, and $\beta$ unbounded. Let $(X,x_0,\alpha)\xrightarrow{\phi}(Y,y_0,\beta)$ be an isomorphism in $\mathbf{WSet}_{\bullet}$. In particular, $X\xrightarrow{\phi}Y$ is a bijection with $\phi(x_0)=y_0$. Let $\theta:=\phi^{-1}$ which is a morphism in $\mathbf{WSet}_{\bullet}$ too. Then, for all $x\not=x_0$, $\beta(\phi(x))=\sum_{y\in \theta^{-1}(\{\, x\,\})}\beta(y)\leq M_f\alpha(x)$. In particular, $\sup_{y\not=y_0}\beta(y)=\sup_{x\not=x_0}\beta(\phi(x))\leq M_f\sup_{x\not=x_0}\alpha(x)<+\infty$ which is a contradiction.
\end{proof}


\subsection{The set of minimal ideals functor}\label{sec:minimal_ideal_functor}

Let $(H,\mu)$ be a  Hilbertian Frobenius algebra.

\begin{lemma}\label{lem:support_of_idempotent_is_finite}
$E(H,\mu)\subseteq \langle G(H,\mu)\rangle\subseteq J(H,\mu)^{\dagger}$. 
More precisely, if $e$ is an idempotent element, then the {\em support of $e$} $S_e:=\{\, g\in G(H,\mu)\colon \langle e,g\rangle\not=0\,\}$ is the (unique) finite subset of $G(H,\mu)$  such that $e=\sum_{g\in S_e}\frac{g}{\|g\|^2}$.
\end{lemma}
\begin{proof}
That $E(H,\mu)\subseteq J(H,\mu)^{\dagger}$  follows from~\cite[Lemma~3.3, p.~112]{BadeDales} since $H=J(H,\mu)^{\perp}\oplus_2 J(H,\mu)$ and $J(H,\mu)^{\perp}$ is a subalgebra by the Structure Theorem for Hilbertian Frobenius semigroups (Theorem~\ref{thm:stHFs}). 
 
Let $e=\sum_{g\in G(H,\mu)}\langle e,\frac{g}{\|g\|}\rangle \frac{g}{\|g\|}$ be an idempotent element. $e^2=\sum_{g\in G(H,\mu)}\frac{1}{\|g\|^2}\langle e,g\rangle^2 g$. But $e^2=e$ so that for each $g\in G(H,\mu)$, $\langle e,g\rangle\in \{\, 0,1\,\}$. Let $S_e:=\{\, g\in G(H,\mu)\colon \langle e,g\rangle\not=0\,\}$. Then, $e=\sum_{g\in S_e}\frac{1}{\|g\|}\frac{g}{\|g\|}$ and thus $(\frac{1}{\|g\|^2})_{g\in S_e}$ is summable and $\|e\|^2=\sum_{g\in S_e}\frac{1}{\|g\|^2}$ (since $G(H,\mu)$ is an orthogonal family). But for each group-like element $g$, $\frac{1}{\|\mu\|_{\mathsf{op}}^2}\leq \frac{1}{\|g\|^2}$ (by Lemma~\ref{lem:Jperp_is_subcoalg}) and $\sum_{g\in S_e}\frac{1}{\|\mu\|_{\mathsf{op}}^2}=\frac{|S_e|}{\|\mu\|^2_{\mathsf{op}}}\leq \|e\|^2$, so that $|S_e|<+\infty$. Uniqueness of $S_e$ is clear since $(\frac{g}{\|g\|})_g$ is an orthonormal basis.
\end{proof}

As usually let $e\leq f$ be defined as $ef=e$, for $e,f\in E(H,\mu)$, and call {\em minimal} an idempotent which is minimal in $(E(H,\mu),\leq)$. Let $Min(E(H,\mu))$ be the set of all these minimal idempotent elements. 

As Lemma~\ref{lem:support_of_idempotent_is_finite} actually establishes a one-one correspondence $e\mapsto S_e$, between $E(H,\mu)$ and the set $\mathfrak{P}_{\mathsf{fin}}(G(H,\mu))$ of finite subsets of group-like elements, and as under this bijection the product of idempotents corresponds to the intersection of their supports, $\leq$ corresponds to the usual inclusion of sets. Consequently $Min(E(H,\mu))=\{\, \frac{g}{\|g\|^2}\colon g\in G(H,\mu)\,\}$. 
\begin{remark}\label{rem:someremarksonminidempotents}
Let $(H,\mu)$ be a Hilbertian Frobenius algebra. Lemma~\ref{lem:bounded_above_and_bounded_below} tells us that (1) $G(H,\mu)\simeq Min(E(H,\mu))$ under $g\mapsto \frac{g}{\|g\|^2}$ with inverse $e\mapsto \frac{e}{\|e\|^2}$, (2) $Min(E(H,\mu))$ is a bounded below orthogonal family. More precisely if $G(H,\mu)$ is non void, that is, $\mu\not=0$ (by Corollary~\ref{cor:radicalFrob}), $Min(E(H,\mu))$ is a non void  orthogonal family. (3) $\{\, \frac{e}{\|e\|}\colon e\in Min(E(H,\mu))\,\}=\{\, \frac{g}{\|g\|}\colon g\in G(H,\mu)\,\}$.
%
%
%
%
%
\end{remark}

\begin{lemma}\label{lem:fromIdempotentgenerated_to_Iperp_mod_unit}
Let $I=He$ for some idempotent $e$ of $(H,\mu)$. Then, $I\subseteq J(H,\mu)^{\perp}$ and  $I^{\perp}$ is a modular ideal with modular unit $e$.
\end{lemma}
\begin{proof}
That $I\subseteq J(H,\mu)^{\perp}$ is clear   since $e\in J(H,\mu)^{\perp}$ (Lemma~\ref{lem:support_of_idempotent_is_finite}), $J(H,\mu)^{\perp}$ is a subalgebra (Lemma~\ref{lem:Jperp_is_subcoalg}) and $J(H,\mu)J(H,\mu)^{\perp}=0$ (Proposition~\ref{prop:normality}). $I^{\perp}$ is indeed an ideal (Corollary~\ref{cor:bij_ideals_subcoalg}). To prove that $e$ is a modular unit of $I^{\perp}$ it suffices to check that $\langle u-ue,ve\rangle=0$ for each $u,v\in H$, which is left to the readers. 
%
%
%
%
\end{proof}

Let $I$ be an ideal of $(H,\mu)$ contained in $J(H,\mu)^{\perp}$ and such that $I\not=(0)$. Then, $I^2\not=(0)$. (Indeed if $I^2=(0)$, then for each $x\in I$, $x^2=0$. But then $x\in J(H,\mu)\cap I\subseteq J(H,\mu)\cap J(H,\mu)^{\perp}=(0)$, that is, $x=0$.)

 \begin{lemma}\label{lem:minid_gpgen_idem_gen}
Let $I$ be a minimal ideal. Then, $I=\mathbb{C}g=\mathbb{C}\frac{g}{\|g\|^2}$ for a unique group-like element $g$, in particular $I\subseteq J(H,\mu)^{\perp}$. Equivalently, $I=\mathbb{C}e=\mathbb{C}\frac{e}{\|e\|^2}$ for a unique minimal  idempotent element $e$. $e$ may be characterized as the unit of $I$ qua a semigroup. Finally, $(-)^{\perp}$ establishes a one-one correspondence between the set of all maximal regular ideals of $(H,\mu)$ and the set $Min(H,\mu)$ of all its minimal ideals. 
 \end{lemma}
 \begin{proof}
 By the above observation and by Brauer's lemma (\cite[p.~162]{Lam}), $I=He$ for a non-zero idempotent $e$ of $(H,\mu)$. 
 
 Since $I$ is minimal, $I^{\perp}$ is maximal  (Corollary~\ref{cor:bij_ideals_subcoalg}). It is modular by Lemma~\ref{lem:fromIdempotentgenerated_to_Iperp_mod_unit}. Therefore $I^{\perp}=(\mathbb{C}g)^{\perp}$ for a unique group-like element $g$ of $(H,\mu)$. Since $\mathbb{C}g$ is finite dimensional, it is closed and thus $\overline{I}=I^{\perp\perp}=(\mathbb{C}g)^{\perp\perp}=\mathbb{C}g$. Since $I\subseteq \mathbb{C}g$, $I$ also is finite dimensional so $I=\mathbb{C}g$. (This in particular shows that $I\mapsto I^{\perp}$ is a bijection from the set of all minimal ideals of $(H,\mu)$ to the set of all its maximal regular ideals since $\mathbb{C}g$ certainly is a minimal ideal for each group-like element $g$ by Remark~\ref{rem:description_of_HilbFrob}.) 
 
 Let $h\in G(H,\mu)$ such that $\mathbb{C}\frac{g}{\|g\|^2}=I=\mathbb{C}\frac{h}{\|h\|^2}$, then $h\in \mathbb{C}g$ but this is possible only if $g=h$ since $G(H,\mu)$ is linearly independent (it is an orthogonal family). This proves uniqueness of the group-like generator $g$ of $I$, and thus also of its  idempotent generator $\frac{g}{\|g\|^2}$. That $\frac{g}{\|g\|^2}$ is a unit element for $I$ is also clear.
 %
%
 \end{proof}
 
Let $I$ be a minimal ideal of $(H,\mu)$. Let $\mycirc{g}(I)$ and $\mycirc{e}(I)$ be respectively its group-like and its idempotent generators, provided by Lemma~\ref{lem:minid_gpgen_idem_gen} and which are related by the equalities $\mycirc{e}(I)=\frac{\mycirc{g}(I)}{\|\mycirc{g}(I)\|^2}$  and thus also $\mycirc{g}(I)=\frac{\mycirc{e}(I)}{\|\mycirc{e}(I)\|^2}$.   This defines maps ${Min}(H,\mu)\xrightarrow{\mycirc{g}}G(H,\mu)$ and $Min(H,\mu)\xrightarrow{\mycirc{e}}Min(E(H,\mu))$. 


The proof of the next result is essentially provided by the proof of Lemma~\ref{lem:minid_gpgen_idem_gen}. 
\begin{lemma}\label{lem:bij_minimal_id_and_grplike}
$\mycirc{g}$ is a bijection with inverse $G(H,\mu)\xrightarrow{\mathbb{C}(-)}Min(H,\mu)$, $g\mapsto \mathbb{C}g$, and $\mycirc{e}$ is a bijection from $Min(H,\mu)$ to $MinE(H,\mu)$ with inverse $Min(E(H,\mu))\xrightarrow{\mathbb{C}(-)}Min(H,\mu)$, $e\mapsto \mathbb{C}e$. 
\end{lemma}

\begin{remark}\label{rem:extension_of_notations}
For convenience one still denotes by $\mycirc{g}$ (resp. $\mycirc{e}$) the bijection from $Min(H,\mu)\cup\{\,0\,\}$ (resp. $Min(H,\mu)\cup\{\,0\,\}$) onto $G(H,\mu)\cup\{\, 0\,\}$  (resp. $Min(E(H,\mu))\cup\{\, 0\,\}$) obtained from the original $\mycirc{g}$ (resp. $\mycirc{e}$) by setting $\mycirc{g}(0):=0$ (resp. $\mycirc{e}(0):=0$).
\end{remark}

Let $(H,\mu)$ be a  Hilbertian Frobenius semigroup. One  recalls from Lemma~\ref{lem:preidem}, that for each $g\in G(H,\mu)$, $\|g\|\leq \|\mu\|_{\mathsf{op}}$. (This is true even for $G(H,\mu)=\emptyset$, or equivalently for $\mu=0$.) So in particular for each $e\in Min(E(H,\mu))$, $\frac{1}{\|e\|}=\|\frac{e}{\|e\|^2}\|\leq \|\mu\|_{\mathsf{op}}$. This is equivalent to $\frac{1}{\|\mu\|_{\mathsf{op}}}\leq \|e\|$, $e\in Min(E(H,\mu))$ but only for a non void $Min(E(H,\mu))$ or equivalently for $\mu\not=0$. Nevertheless even when $\mu=0$, $Min(E(H,\mu))$ is bounded below since it is void. In order to avoid statements by cases where $\mu\not=0$ or $\mu=0$, one defines the bound of $(H,\mu)$. 
\begin{definition}
Let $(H,\mu)$ be a  Hilbertian Frobenius semigroup. Define the {\em bound} $b(H,\mu)>0$ of $(H,\mu)$ by $b(H,\mu):=\left\{\begin{array}{ll}
1 & \mbox{if}\ \mu=0\\
\|\mu\|_{\mathsf{op}} & \mbox{if}\ \mu\not=0
\end{array}\right.$. For each $g\in G(H,\mu)$, $\|g\|\leq b(H,\mu)$. So in particular for each $e\in Min(E(H,\mu))$, $\frac{1}{\|e\|}=\|\frac{e}{\|e\|^2}\|\leq b(H,\mu)$. This is equivalent, even for $\mu=0$, to $\frac{1}{b(H,\mu)}\leq \|e\|$, $e\in Min(E(H,\mu))$.  One defines $w_{(H,\mu)}\colon Min(H,\mu)\to [\frac{1}{b(H,\mu)^{2}},+\infty[$  by $w_{(H,\mu)}(I):=\frac{1}{\|\mycirc{g}(I)\|^2}=\|\mycirc{e}(I)\|^2$. Let $Min_{\bullet}(H,\mu):=(Min(H,\mu)\cup\{\, 0\,\},0,w_{(H,\mu)})$. One defines ${}_{\mathsf{bnd},c}\mathbf{FrobSem}(\mathbb{Hilb})$ and ${}_{\mathsf{unbnd},c}\mathbf{FrobSem}(\mathbb{Hilb})$ as the full subcategories of ${}_c\mathbf{FrobSem}(\mathbb{Hilb})$ spanned by the semigroups $(H,\mu)$ with $w_{(H,\mu)}$ bounded above, and respectively unbounded above.   We  also consider  ${}_{\mathsf{semisimple},\mathsf{bnd},c}\mathbf{FrobSem}(\mathbb{Hilb})$ and ${}_{\mathsf{semisimple},\mathsf{unbnd},c}\mathbf{FrobSem}(\mathbb{Hilb})$.
\end{definition}

Let $(H,\mu)\xrightarrow{f}(K,\gamma)$ be a semigroup map between  Hilbertian Frobenius semigroups. Therefore $(K,\gamma^{\dagger})\xrightarrow{f^{\dagger}}(H,\mu^{\dagger})$ is a coalgebra map and thus $f^{\dagger}(G(K,\gamma))\subseteq G(H,\mu)\cup\{\, 0\,\}$. 

Let $\ell\colon Min(K,\gamma)\to Min(H,\mu)\cup\{\, 0\,\}$ be given by $\ell(J)=0$ if, and only if, $f^{\dagger}(\mycirc{g}(J))=0$,  and $\ell(J)=I$ if, and only if, $f^{\dagger}(\mycirc{g}(J))=\mycirc{g}(I)$. Therefore, for each $J\in Min(K,\gamma)\cup\{\, 0\,\}$, $\ell(J)=\mathbb{C}f^{\dagger}(\mycirc{g}(J))$. (Observe that for $J=0$ or for $f^{\dagger}(\mycirc{g}(J))=0$, $\ell(J)=\mathbb{C}0=0$.)

Let $I\in Min(H,\mu)$. Then, 
\begin{equation}
\begin{array}{lll}
\langle f(\mycirc{e}(I)),\frac{\mycirc{e}(J)}{\|\mycirc{e}(J)\|}\rangle&=&\langle \mycirc{e}(I),\|\mycirc{e}(J)\|f^{\dagger}(\frac{\mycirc{e}(J)}{\|\mycirc{e}(J)\|^2})\rangle.
\end{array}
\end{equation}
Consequently, for $J\in Min(K,\gamma)$, $\langle f(\mycirc{e}(I)),\frac{\mycirc{e}(J)}{\|\mycirc{e}(J)\|}\rangle=0$ if, and only if, $f^{\dagger}(g(J))=0$ or $f^{\dagger}(g(J))=g(I')$ with $I'\not=I$, and $\langle f(\mycirc{e}(I)),\frac{\mycirc{e}(J)}{\|\mycirc{e}(J)\|}\rangle=\langle \mycirc{e}(I),\|\mycirc{e}(J)\|g(I)\rangle=\frac{\|\mycirc{e}(J)\|}{\|\mycirc{e}(I)\|^2}\langle \mycirc{e}(I),\mycirc{e}(I)\rangle=\|\mycirc{e}(J)\|$ if, and only if, $f^{\dagger}(\mycirc{g}(J))=\mycirc{g}(I)$ if, and only if, $\ell(J)=I$. 

Since $f$ is a semigroup map, it sends an idempotent element of $(H,\mu)$ to one of $(K,\gamma)$. Whence the image by $f$ of an idempotent belongs to $J(K,\gamma)^{\perp}$. Therefore, for each $I\in Min(H,\mu)$, $f(\mycirc{e}(I))=\sum_{J\in Min(K,\gamma)} \langle f(\mycirc{e}(I)),\frac{\mycirc{e}(J)}{\|\mycirc{e}(J)\|}\rangle \frac{\mycirc{e}(J)}{\|\mycirc{e}(J)\|}=\sum_{J\in \ell^{-1}(\{\, I\,\})}\|\mycirc{e}(J)\|\frac{\mycirc{e}(J)}{\|\mycirc{e}(J)\|}$. 

Now for  $I\in Min(H,\mu)$, $\sum_{J\in \ell^{-1}(\{\, I\,\})}b(K,\gamma)^2\|\mycirc{e}(J)\|^2=b(K,\gamma)^2\|f(\mycirc{e}(I))\|^2\leq b(K,\gamma)^2\|f\|_{\mathsf{op}}^{2}\|\mycirc{e}(I)\|^2<+\infty$. But $1\leq b(K,\gamma)^2\|\mycirc{e}(J)\|^2$ for each $J\in Min(K,\gamma)$, so necessarily $|\ell^{-1}(\{\, I\,\})|$ is finite. 


From the equality $f(\mycirc{e}(I))=\sum_{J\in \ell^{-1}(\{\, I\,\})}\|\mycirc{e}(J)\|\frac{\mycirc{e}(J)}{\|\mycirc{e}(J)\|}=\sum_{J\in \ell^{-1}(\{\, I\,\})}\mycirc{e}(J)$ it follows that $\sum_{J\in \ell^{-1}(\{\, I\,\})}\|\mycirc{e}(J)\|^2=\|f(\mycirc{e}(I))\|^2\leq\|f\|^2\|\mycirc{e}(I)\|^2$ for each $I\in Min(H,\mu)$. 

Consequently, $\ell\in \mathbf{WSet}_{\bullet}(Min_{\bullet}(K,\gamma),Min_{\bullet}(H,\mu))$, where $\ell$ is extended to the whole of $Min(K,\gamma)\cup\{\, 0\,\}$ by setting $\ell(0):=0$.

Contravariance of $f\mapsto \ell$ is clear so that one gets the {\em set of minimal ideals} functor $Min_{\bullet}\colon {}_c\mathbf{FrobSem}(\mathbb{Hilb})\to \mathbf{WSet}_{\bullet}^{\mathsf{op}}$. One also has the following restrictions $Min_{\bullet}\colon {}_{\mathsf{semisimple},c}\mathbf{FrobSem}(\mathbb{Hilb})\to\mathbf{WSet}_{\bullet}^{\mathsf{op}}$, $Min_{\bullet}\colon {}_{?_1,c}\mathbf{FrobSem}(\mathbb{Hilb})\to{}_{?_2}\mathbf{WSet}_{\bullet}^{\mathsf{op}}$ where either $?_1=?_2\in\{\, \mathsf{bnd},\mathsf{unbnd}\,\}$, or $?_1= (\mathsf{semisimple},?_2)$ with $?_2\in\{\, \mathsf{bnd},\mathsf{unbnd}\,\}$.

\subsection{The $\ell^2_{\bullet}$ functor}
Let $(X,x_0,\alpha)$ be an object of $\mathbf{WSet}_{\bullet}$. Define $\alpha^{+}\colon X\to[C_{\alpha},+\infty[$, $0<C_{\alpha}$,  by $\alpha^{+}(x)=\alpha(x)$, $x\not=x_0$, $\alpha^{+}(x_0):=C_\alpha$. One may consider the semisimple  Hilbertian Frobenius algebra $((\ell^2_{\alpha^+}(X),\langle\cdot,\cdot\rangle_{\alpha^{+}}),\mu_X)$ (Proposition~\ref{prop:grplike_of_ell_2_alpha}). 

Let $\ell^2_{\bullet}(X,x_0,\alpha):=\{\, u\in \ell^2_{\alpha^{+}}(X)\colon u(x_0)=0\,\}=\delta_{x_0}^{\perp}$. It is a closed subalgebra, since it is even a closed (maximal) ideal as the kernel of $\langle\cdot ,\delta_{x_0}\rangle$, and $x_0\in G(\ell^2_{\alpha^+}(X),\mu_X)$. 

The Hilbertian algebra $((\ell^2_{\bullet}(X,x_0,\alpha),\langle\cdot,\cdot\rangle_{\alpha^+}),(\mu_{X})_{|_{\delta_{x_0}^{\perp}}})$ is clearly unitarily isomorphic to $((\ell^2_{\alpha}(X\setminus\{\, x_0\,\}),\langle\cdot,\cdot\rangle_{\alpha}),\mu_{X\setminus\{\, x_0\,\}})$. As a matter of fact, $\ell^2_{\bullet}(X,x_0,\alpha)$ is an object of ${}_{\mathsf{semisimple},c}\mathbf{FrobSem}(\mathbb{Hilb})$, with $G(\ell^2_{\bullet}(X,x_0,\alpha))=\{\, \frac{x}{\alpha(x)}\colon x\in X\setminus\{\, x_0\,\}\,\}$ by {Proposition}~\ref{prop:grplike_of_ell_2_alpha}. 

In particular, when $\alpha$ is bounded above (resp. unbounded), then $\ell^2_{\bullet}(X,x_0,\alpha)$ is an object of ${}_{\mathsf{bnd},\mathsf{semisimple},c}\mathbf{FrobSem}(\mathbb{Hilb})$ (resp. ${}_{\mathsf{unbnd},\mathsf{semisimple},c}\mathbf{FrobSem}(\mathbb{Hilb})$).

Let $f\in \mathbf{WSet}_{\bullet}((X,x_0,\alpha),(Y,y_0,\beta))$. Let $u\in \ell^2_{\bullet}(Y,y_0,\beta)$. Then, $u\circ f\in \ell^2_{\bullet}(X,x_0,\alpha)$. Indeed, $u(f(x_0))=u(y_0)=0$. Let $A\subseteq X$ be a finite set. Then, 
\begin{equation}
\begin{array}{lll}\sum_{x\in A\setminus\{\, x_0\,\}}\alpha(x)|u(f(x))|^2&=&
\sum_{y\in f(A)\setminus\{\, y_0\,\}}\left (\sum_{x\in f^{-1}(\{\, y\,\})}\alpha(x)\right)|u(y)|^2\\
&\leq& \sum_{y\not=y_0}M_f\beta(y)|u(y)|^2\\
&=&M_f\|u\|_{\beta}^2.
\end{array}
\end{equation} In particular $\|u\circ f\|_{\alpha}\leq M_f^{\frac{1}{2}}\|u\|_{\beta}$. 

Since $\ell^2_{\bullet}(f)\colon \ell^2_{\bullet}(Y,y_0,\beta)\to \ell^2_{\bullet}(X,x_0,\alpha)$, $u\mapsto u\circ f$, is clearly a semigroup morphism, it follows easily that one has a functor $\ell^2_{\bullet}\colon \mathbf{WSet}_{\bullet}^{\mathsf{op}}\to {}_c\mathbf{FrobSem}(\mathbb{Hilb})$ and thus also the following co-restrictions  $\ell^2_{\bullet}\colon\mathbf{WSet}_{\bullet}^{\mathsf{op}}\to{}_{\mathsf{semisimple},c}\mathbf{FrobSem}(\mathbb{Hilb})$, $\ell^2_{\bullet}\colon {}_{?}\mathbf{WSet}_{\bullet}\to {}_{?,c}\mathbf{FrobSem}(\mathbb{Hilb})$, and  $\ell^2_{\bullet}\colon {}_{?}\mathbf{WSet}_{\bullet}\to {}_{\mathsf{semisimple},?,c}\mathbf{FrobSem}(\mathbb{Hilb})$ where $?$ stands  either for $\mathsf{bnd}$ or $\mathsf{unbnd}$.

\begin{remark}
The $\ell^2_{\bullet}$ functor from~\cite{Poinsot2019} may be regarded as a functor with values in ${}_c^{\dagger}\mathbf{FrobSem}(\mathbb{Hilb})$ or even in ${}_{\mathsf{part iso},c}\mathbf{FrobSem}(\mathbb{Hilb})$, rather than in ${}_c^{\dagger}\mathbf{Sem}(\mathbb{Hilb})$, and as such may be seen as the following co-restriction.
\begin{equation}
\xymatrix{
\mathbf{WSet}_{\bullet} \ar[r]^{\ell^2_{\bullet}}&{}_c\mathbf{FrobSem}(\mathbb{Hilb})\\
&{}_{\mathsf{partiso},c}\mathbf{FrobSem}(\mathbb{Hilb})\ar@{^{(}->}[u]\\
\ar@{^{(}->}[uu]\mathbf{Set}_{\bullet,<+\infty}\ar[r]_{\ell^2_{\bullet}}&{}_{c}^{\dagger}\mathbf{FrobSem}(\mathbb{Hilb})\ar@{^{(}->}[u]
}
\end{equation}
\end{remark}

\subsection{The main equivalences}

\begin{theorem}\label{thm:equivalence_weighted_sets_and_semisimple_Frob}
One has an adjunction $Min_{\bullet}\dashv \ell^2_{\bullet}\colon \mathbf{WSet}_{\bullet}^{\mathsf{op}}\to {}_c\mathbf{FrobSem}(\mathbb{Hilb})$ which restricts to an adjoint equivalence $Min_{\bullet}\dashv \ell^2_{\bullet}\colon \mathbf{WSet}_{\bullet}^{\mathsf{op}}\simeq {}_{\mathsf{semisimple},c}\mathbf{FrobSem}(\mathbb{Hilb})$. In particular, ${}_{\mathsf{semisimple},c}\mathbf{FrobSem}(\mathbb{FdHilb})\simeq\mathbf{FinSet}_{\bullet}^{\mathsf{op}}$, where $\mathbf{FinSet}_{\bullet}$ is the category of finite pointe sets and base-point preserving maps.
\end{theorem}
\begin{proof}
Let $(X,x_0,\alpha)$ be a weighted pointed set, with $X\setminus\{\, x_0\,\}\xrightarrow{\alpha}[C_{\alpha},+\infty[$, where $C_{\alpha}>0$. Then, $Min(\ell^2_{\bullet}(X,x_0,\alpha))\simeq Min(\ell^2_{\alpha}(X\setminus\{\, x_0\,\},\mu_{X\setminus\{\,x_0\,\}}))=\{\, \mathbb{C}\delta_x\colon x\not=x_0\,\}$. 

Moreover $w_{\ell^2_{\bullet}(X,x_0,\alpha)}=w_{(\ell^2_{\alpha}(X\setminus\{\, x_0\,\}),\mu_{X\setminus\{\, x_0\,\}})}$ so $w_{\ell^2_{\bullet}(X,x_0,\alpha)}(\mathbb{C}\delta_x)=\|\mycirc{e}(\mathbb{C}\delta_x)\|^2_{\alpha}=\|\delta_x\|^2_{\alpha}=\alpha(x)$.

 Let $\epsilon_{(X,x_0,\alpha)}\colon (X,x_0)\to (Min(\ell^2_{\bullet}(X,x_0,\alpha))\cup\{\, 0\,\},0)$ be given by $\epsilon_{(X,x_0,\alpha)}(x):=\mathbb{C}\delta_x$, $x\not=x_0$, and $\epsilon_{(X,x_0,\alpha)}(x_0):=0$. $\epsilon_{(X,x_0,\alpha)}$ is clearly a pointed bijection. By the above, $\epsilon_{(X,x_0,\alpha)}$ is clearly a $\mathbf{WSet}_{\bullet}$-isomorphism.  
 
Let $(H,\mu)$ be a Hilbertian Frobenius algebra. An orthonormal basis for the semigroup $\ell^2_{\bullet}(Min_{\bullet}(H,\mu))$ is given by $(\frac{\delta_I}{\sqrt{w_{(H,\mu)}(I)}})_{I\in Min(H,\mu)}=(\frac{\delta_I}{\|\mycirc{e}(I)\|})_{I\in Min(H,\mu)}$. Whence as a Hilbert space, 
 $\ell^2_{\bullet}(Min_{\bullet}(H,\mu))$ is  unitary isomorphic to $J(H,\mu)^{\dagger}$ because an orthonormal basis of the latter is given by $\{\, \frac{e}{\|e\|}\colon e\in Min(E(H,\mu))\,\}$ (Remark~\ref{rem:someremarksonminidempotents}). Let $\Phi_{(H,\mu)}\colon \ell^2_{\bullet}(Min_{\bullet}(H,\mu))\simeq J(H,\mu)^{\perp}$ be the corresponding unitary transformation. For each minimal ideal $I$ of $(H,\mu)$, $\Phi_{(H,\mu)}(\delta_I)=\mycirc{e}(I)$, and  one has $\Phi_{(H,\mu)}(IJ)=\Phi_{(H,\mu)}(\delta_{I,J}I)=\delta_{I,J}\mycirc{e}(I)=\mycirc{e}(I)\mycirc{e}(J)=\Phi_{(H,\mu)}(\delta_I)\Phi_{(H,\mu)}(\delta_J)$ for a minimal ideal $J$, from which it follows that $\Phi_{(H,\mu)}$ is actually an isomorphism of semigroups. 

Now let $(X,x_0,\alpha)\xrightarrow{f}Min_{\bullet}(H,\mu)$ be a $\mathbf{WSet}_{\bullet}$-morphism, where $(H,\mu)$ is an  Hilbertian Frobenius semigroup. Let  $(H,\mu)\xrightarrow{f^{\sharp}}\ell^2_{\bullet}(X,x_0,\alpha):=(H,\mu)\xrightarrow{\pi_{J(H,\mu)^{\perp}}}J(H,\mu)^{\perp}\xrightarrow{\Phi^{-1}_{(H,\mu)}}\ell^2_{\bullet}(Min_{\bullet}(H,\mu))\xrightarrow{\ell^2_{\bullet}(f)}\ell_\bullet^{2}(X,x_0,\alpha)$. By construction $f^{\sharp}$ is a morphism of semigroups. 

For $g\in G(H,\mu)$, $f^{\sharp}(\frac{g}{\|g\|})=\ell^2_{\bullet}(f)(\Phi^{-1}_{(H,\mu)}(\frac{g}{\|g\|}))=\ell^2_{\bullet}(f)(\frac{\delta_{\mathbb{C}g}}{\|e(\mathbb{C}g)\|})=\ell^2_{\bullet}(f)(\|g\|\delta_{\mathbb{C}g})=\|g\|\delta_{\mathbb{C}g}\circ f$. Let $u\in \ell^2_{\bullet}(X,x_0,\alpha)$ and $g\in G(H,\mu)$. Then, $\langle (f^{\sharp})^{\dagger}(u),\frac{g}{\|g\|}\rangle=
\langle u,\|g\|\delta_{\mathbb{C}g}\circ f\rangle_{\alpha}=\sum_{x\not=x_0}\alpha(x)u(x)\|g\|\delta_{\mathbb{C}g}(f(x))$. Consequently, \begin{equation}(f^{\sharp})^{\dagger}(u)=\sum_{g\in G(H,\mu)}\|g\|\left(\sum_{x\in f^{-1}(\{\, \mathbb{C}g\,\})}\alpha(x)u(x)\right)\frac{g}{\|g\|}=\sum_{g\in G(H,\mu)}\left(\sum_{x\in f^{-1}(\{\, \mathbb{C}g\,\})}\alpha(x)u(x)\right){g}.\end{equation} In particular, for $x\in X\setminus\{\, x_0\,\}$, $(f^{\sharp})^{\dagger}(\frac{\delta_x}{\alpha(x)})=\frac{\alpha(x)}{\alpha(x)}\mycirc{g}(f(x))=\mycirc{g}(f(x))$. (Recall that $\mycirc{g}(0)=0$ by Remark~\ref{rem:extension_of_notations}.) This is equivalent to $Min_{\bullet}(f^{\sharp})(\mathbb{C}\delta_x)=f(x)$, $x\in X\setminus\{\, x_0\,\}$. In other words, $Min_{\bullet}(f^{\sharp})\circ \epsilon_{(X,x_0,\alpha)}=f$.

Now let $h\colon (H,\mu)\to \ell^2_{\bullet}(X,x_0,\alpha)$ be a semigroup map such that $Min_{\bullet}(h^{\dagger})\circ \epsilon_{(X,x_0,\alpha)}=f$. Then, for each $x\not=x_0$, $x\in X$, $h^{\dagger}(\frac{\delta_x}{\alpha(x)})=\mycirc{g}(f(x))=(f^{\sharp})^{\dagger}(\frac{\delta_x}{\alpha(x)})$. Since $\{\, \frac{\delta_x}{\alpha(x)}\colon x\in X\setminus\{\, x_0\,\}\,\}$ is dense into $\ell^2_{\bullet}(X,x_0,\alpha_x)$ it follows that $h^{\dagger}=(f^{\sharp})^{\dagger}$, that is, $h=f^{\sharp}$. 

The proof for the adjunction will be concluded when naturality of $(\epsilon_{(X,x_0,\alpha)})_{(X,x_0,\alpha)}$ will be proved. This is equivalent to the requirement that for each $(X,x_0,\alpha)\xrightarrow{f}(Y,y_0,\beta)$, $Min_{\bullet}(\ell^2_{\bullet}(f))(\mathbb{C}\delta_x)=\mathbb{C}\delta_{f(x)}$, $x\in X\setminus\{\, x_0\,\}$ with $f(x)\not=y_0$, and also that $Min_{\bullet}(\ell^2_{\bullet}(f))(\mathbb{C}\delta_x)=0$, $x\in X\setminus\{\,x_0\,\}$ with $f(x)=y_0$. 

But $Min_{\bullet}(\ell^2_{\bullet}(f))(\mathbb{C}\delta_x)=\mathbb{C}((\ell^2_{\bullet}(f))^{\dagger}(\mycirc{g}(\mathbb{C}\delta_x)))$.  Naturality thus is equivalent to 
$(\ell^2_{\bullet}(f))^{\dagger}(\frac{\delta_x}{\alpha(x)})=\frac{\delta_{f(x)}}{\beta(f(x))}$, $x\in X\setminus\{\, x_0\,\}$, $f(x)\not=y_0$, since $\frac{\delta_x}{\alpha(x)}=\mycirc{g}(\mathbb{C}\delta_x)$ and in this case $\mycirc{g}(\mathbb{C}\delta_{f(x)})=\frac{\delta_{f(x)}}{\beta(f(x)}$, and $(\ell^2_{\bullet}(f))^{\dagger}(\frac{\delta_x}{\alpha(x)})=0$, $x\in X\setminus\{\, x_0\,\}$, $f(x)=y_0$.  So one has to compute $(\ell^2_\bullet(f))^{\dagger}(\frac{\delta_x}{\alpha(x)})$.  Let $u\in \ell^2_{\bullet}(X,x_0,\alpha)$ and let $y\in Y\setminus\{\, y_0\,\}$. Then, 
\begin{equation}
\begin{array}{lll}
\langle (\ell^2_\bullet(f))^{\dagger}(u),\frac{\delta_y}{\beta(y)^{\frac{1}{2}}}\rangle_\beta&=&\langle u,\ell^2_{\bullet}(f)(\frac{\delta_y}{\beta(y)^{\frac{1}{2}}})\rangle_\alpha\\
&=&\sum_{x\in X\setminus\{\, x_0\,\}}\alpha(x)u(x)\frac{\delta_y(f(x))}{\beta(y)^{\frac{1}{2}}}.
\end{array}
\end{equation}
Therefore 
\begin{equation}
\begin{array}{lll}
(\ell^2_{\bullet}(f))^{\dagger}(u)&=&\sum_{y\in Y\setminus\{\, y_0\,\}}\frac{1}{\beta(y)^{\frac{1}{2}}}\left (\sum_{x\in f^{-1}(\{\, y\,\})}\alpha(x)u(x)\right)\frac{\delta_y}{\beta(y)^{\frac{1}{2}}}\\
&=&\sum_{y\in Y\setminus\{\, y_0\,\}}\left (\sum_{x\in f^{-1}(\{\, y\,\})}\alpha(x)u(x)\right)\frac{\delta_y}{\beta(y)}.
\end{array}
\end{equation}
In particular, for each $x\in X\setminus\{\, x_0\,\}$ with $f(x)\not=y_0$
\begin{equation}
\begin{array}{lll}
(\ell^2_{\bullet}(f))^{\dagger}(\frac{\delta_x}{\alpha(x)})&=&\sum_{y\not=y_0}\frac{\delta_y}{\beta(y)}\left (\sum_{x'\in f^{-1}(\{\, y\,\})}\alpha(x')\frac{\delta_x(x')}{\alpha(x)}\right)\\
&=&\frac{\delta_{f(x)}}{\beta(f(x))}.
\end{array}
\end{equation}
For each $x\in X\setminus \{\, x_0\,\}$ with $f(x)=y_0$, the same computation as above leads as expected to $(\ell^2_{\bullet}(f))^{\dagger}(\frac{\delta_x}{\alpha(x)})=0$.

It remains to prove the statement about the equivalence of categories. The component at $(H,\mu)$ of the unit of the above adjunction is by definition, $id_{Min_{\bullet}(H,\mu)}^{\sharp}=\Phi^{-1}_{(H,\mu)}\circ \pi_{J(H,\mu)^{\perp}}$. Since the counit $\epsilon$ is an isomorphism, the above equivalence restricts to an equivalence between $\mathbf{WSet}_{\bullet}^{\mathsf{op}}$ and the full subcategory of ${}_c\mathbf{FrobSem}(\mathbb{Hilb})$ spanned by those algebras $(H,\mu)$ such that $\pi_{J(H,\mu)^{\perp}}$ is an isomorphism, that is, the semisimple Hilbertian Frobenius semigroups. 

Concerning the last statement one first notices that the adjunction $min_{\bullet}\dashv \ell^2_\bullet$ co-resricts to the adjunction 
$min_{\bullet}\dashv \ell^2_\bullet\colon {}_c\mathbf{FrobSem}(\mathbb{FdHilb})\to \mathbf{WFinSet}_{\bullet}$, where $\mathbf{WFinSet}_{\bullet}$ stands for the full subcategory of $\mathbf{WSet}_{\bullet}$ spanned by the pointed weighted sets $(X,x_0,\alpha)$ where $X$ is finite. By finiteness the embedding functor $E\colon \mathbf{Set}_{\bullet,<+\infty}\to \mathbf{WSet}_{\bullet}$ from Lemma~\ref{lem:fullembedding}, provides an equivalence  $\mathbf{FinSet}_\bullet\simeq \mathbf{WFinSet}_{\bullet}$. By restriction again one obtains the expected equivalence. 
\end{proof}

\begin{remark}
It is a consequence of Theorem~\ref{thm:equivalence_weighted_sets_and_semisimple_Frob} and of Lemma~\ref{lem:equivalence_bounded_normalized} that not all the semisimple Hilbertian Frobenius semigroups $(H,\mu),(H,\gamma)$ on the same Hilbert space $H$ are isomorphic. (In view of Theorem~\ref{thm:dictionary} it suffices to consider bounded above orthogonal bases of $H$, one also bounded below and the other not.)
\end{remark}

Using some previous results, one obtains the following easily (in particular Proposition~\ref{prop:structuretheoremrevisited1}), by restrictions of the equivalences from Theorem~\ref{thm:equivalence_weighted_sets_and_semisimple_Frob}. (Item~\ref{it:corollary:item2} below is proved as follows: one has an adjoint equivalence $G_{\bullet}\dashv \ell^2_{\bullet}\colon {}_{\mathsf{semisimple},}{{}_{c}^{\dagger}}\mathbf{Sem}(\mathbb{Hilb})\simeq \mathbf{Set}^{\mathsf{op}}_{\bullet,<+\infty}$ by~\cite[Theorem~41, p.~28]{Poinsot2019}. But ${}_{\mathsf{semisimple},}{{}_{c}^{\dagger}}\mathbf{Sem}(\mathbb{Hilb})={}_c^{\dagger}\mathbf{FrobSem}(\mathbb{Hilb})$. So the first equivalence is proved. Lemma~\ref{lem:equivalence_bounded_normalized} provides the second equivalence. For the  last one it suffices to consider the adjunction $Min_{\bullet}\dashv \ell^2_{\bullet}\colon {}_{\mathsf{bnd},c}\mathbf{FrobSem}(\mathbb{Hilb})\to {}_{\mathsf{bnd}}\mathbf{WSet}_{\bullet}^{\mathsf{op}}$ obtained by restriction, which itself restricts to the expected equivalence.)
\begin{corollary}
One has the following equivalences of categories.
\begin{enumerate}
\item ${}_{\mathsf{unbnd}}\mathbf{WSet}^{\mathsf{op}}_{\bullet}\simeq {}_{\mathsf{semisimple},\mathsf{unbnd},c}\mathbf{FrobSem}(\mathbb{Hilb})$.
\item\label{it:corollary:item2} ${}_c^{\dagger}\mathbf{FrobSem}(\mathbb{Hilb})\simeq \mathbf{Set}^{\mathsf{op}}_{\bullet,<+\infty}\simeq {}_{\mathsf{bnd}}\mathbf{WSet}_{\bullet}^{\mathsf{op}}\simeq {}_{\mathsf{semisimple},\mathsf{bnd},c}\mathbf{FrobSem}(\mathbb{Hilb})$.
\item $\mathbf{WSet}_{\bullet}^{\mathsf{op}}\times \mathbf{Hilb}\simeq {}_c\mathbf{FrobSem}(\mathbb{Hilb})$.
\item ${}_{\mathsf{unbnd}}\mathbf{WSet}_{\bullet}^{\mathsf{op}}\times \mathbf{Hilb}\simeq {}_{\mathsf{semisimple},\mathsf{unbnd},c}\mathbf{FrobSem}(\mathbb{Hilb})\times\mathbf{Hilb}\simeq {}_{\mathsf{unbnd},c}\mathbf{FrobSem}(\mathbb{Hilb})$.
\item ${}_{\mathsf{partiso},c}\mathbf{FrobSem}(\mathbb{Hilb})\simeq {}_{c}^{\dagger}\mathbf{FrobSem}(\mathbb{Hilb})\times\mathbf{Hilb}\simeq \mathbf{Set}^{\mathsf{op}}_{\bullet,<+\infty}\times\mathbf{Hilb}\simeq{}_{\mathsf{bnd}}\mathbf{WSet}_{\bullet}^{\mathsf{op}}\times \mathbf{Hilb}\simeq {}_{\mathsf{semisimple},\mathsf{bnd},c}\mathbf{FrobSem}(\mathbb{Hilb})\times\mathbf{Hilb}\simeq {}_{\mathsf{bnd},c}\mathbf{FrobSem}(\mathbb{Hilb})$.
\item $\mathbf{FinSet}_{\bullet}^{\mathsf{op}}\simeq {}_c^{\dagger}\mathbf{FrobSem}(\mathbb{FdHilb})$ and $\mathbf{FinSet}_{\bullet}^{\mathsf{op}}\times \mathbf{FdHilb}\simeq {}_c\mathbf{FrobSem}(\mathbb{FdHilb})\simeq {}_{\mathsf{partiso},c}\mathbf{FrobSem}(\mathbb{FdHilb})$.
\end{enumerate}
\end{corollary}


\section{Some other equivalences}\label{sec:otherequiv}

\subsection{More on Frobenius semigroups with a partial isometric comultiplication}\label{sec:moreonpartiso}

\begin{proposition}\label{prop:bisemigroup}
Let $(H,\mu)$ be a Hilbertian Frobenius semigroup. Let us consider the following (in general non commutative) diagram, where $\sigma_{2,3}\colon (H\hat{\otimes}_2 H)\hat{\otimes}_2 (H\hat{\otimes}_2 H)\to (H\hat{\otimes}_2 H)\hat{\otimes}_2 (H\hat{\otimes}_2 H)$ is the unitary isomorphism given by $(u_1\otimes u_2)\otimes(u_3\otimes u_4)\mapsto (u_1\otimes u_3)\otimes (u_2\otimes u_4)$. 
\begin{equation}\label{diag:bisem}
\xymatrix{
H\ar[r]^{\mu^{\dagger}} & H\hat{\otimes}_2 H\\
&(H\hat{\otimes}_2 H)\hat{\otimes}_2 (H\hat{\otimes}_2 H)\ar[u]_{\mu\hat{\otimes}_2 \mu}\eq[d]^{\sigma_{2,3}}\\
\ar[uu]^{\mu}H\hat{\otimes}_2 H \ar[r]_-{\mu^{\dagger}\hat{\otimes}_2 \mu^{\dagger}}&(H\hat{\otimes}_2 H)\hat{\otimes}_2 (H\hat{\otimes}_2 H)
}
\end{equation}
$((H,\mu),\mu^{\dagger})$  is a bisemigroup in $\mathbb{Hilb}$, that is, Diag.~(\ref{diag:bisem}) above commutes, if, and only if, for each $g\in G(H,\mu)$, $\|g\|=1$ if, and only if, $\mu^{\dagger}$ is a partial isometry.
\end{proposition}
\begin{proof}
As a preambule, since $(H,\mu)$ is Frobenius, one knows that for each $u,v\in H$, $\mu(u\otimes v)=\mu(p_{J(H,\mu)}(u)\otimes p_{J(H,\mu)}(v))=\sum_{g\in G(H,\mu)}\langle u,\frac{g}{\|g\|}\rangle\langle v,\frac{g}{\|g\|}\rangle g$ while $\mu^{\dagger}(u)=\mu^{\dagger}(p_{J(H,\mu)^{\perp}(u)})=\sum_{g\in G(H,\mu)}\langle u,\frac{g}{\|g\|}\rangle\frac{1}{\|g\|}g\otimes g$ (since for each $u\in J(H,\mu)$, $v,w\in H$, $\langle \mu^{\dagger}(u),v\otimes w\rangle=\langle u,p_{J(H,\mu)^{\perp}}(v)p_{J(H,\mu)^{\perp}}(w)\rangle=0$). This being said, one has for each $u,v\in H$, 
\begin{equation}
\begin{array}{lll}
(\mu\hat{\otimes}_2\mu)(\sigma_{2,3}((\mu^{\dagger}\hat{\otimes}_2 \mu^{\dagger})(u\otimes v)))&=&(\mu\hat{\otimes}_2\mu)(\sum_{g,h}\langle u,\frac{g}{\|g\|^2}\rangle\langle v,\frac{h}{\|h\|^2}\rangle(g\otimes h)\otimes (g\otimes h))\\
&=&\sum_{g}\langle u,g\rangle\langle v,g\rangle g\otimes g
\end{array}
\end{equation}
and 
\begin{equation}
\begin{array}{lll}
\mu^{\dagger}(\mu(u\otimes v))&=&\mu^{\dagger}(\sum_{g\in G(H,\mu)}\langle u,\frac{g}{\|g\|}\rangle\langle v,\frac{g}{\|g\|}\rangle g)\\
&=&\sum_{g\in G(H,\mu)}\langle u,\frac{g}{\|g\|}\rangle\langle v,\frac{g}{\|g\|}\rangle g\otimes g.
\end{array}
\end{equation}
As a matter of fact, $((H,\mu),\mu^{\dagger})$ is a bisemigroup in $\mathbb{Hilb}$ if, and only if, for each $g\in G(H,\mu)$, $\|g\|=1$. (The direct implication is clear, while the converse is due to the fact that $\|g\|^2g\otimes g=\mu^{\dagger}(\mu(g\otimes g))=(\mu\hat{\otimes}_2 \mu)(\sigma_{2,3}((\mu^{\dagger}\hat{\otimes}_2\mu^{\dagger})(g\otimes g)))=\|g\|^4g\otimes g$ for each $g\in G(H,\mu)$.)

Now let $u\in H$, then
\begin{equation}
\begin{array}{lll}
\mu^{\dagger}(\mu(\mu^{\dagger}(u)))&=&\mu^{\dagger}(\mu(\sum_{g\in G(H,\mu)}\langle u,\frac{g}{\|g\|}\rangle g\otimes g))\\
&=&\mu^{\dagger}(\sum_{g\in G(H,\mu)}\langle u,\frac{g}{\|g\|}\rangle\|g\|^2g)\\
&=&\sum_{g\in G(H,\mu)}\langle u,g\rangle\|g\|g\otimes g.
\end{array}
\end{equation} 
Therefore $\mu^{\dagger}$ is a partial isometry if, and only if, $\|g\|=1$ for each $g\in G(H,\mu)$. (The converse implication is due to the fact that for each $g\in G(H,\mu)$, $\|g\|^2g\otimes g=\mu^{\dagger}(\mu(\mu^{\dagger}(g)))=\mu^{\dagger}(g)=g\otimes g$.)
\end{proof}

Let ${}_{\mathsf{partiso},c}\mathbf{FrobBisem}(\mathbb{Hilb})$ be the subcategory of ${}_{\mathsf{partiso},c}\mathbf{FrobSem}(\mathbb{Hilb})$ with the same objects but with morphisms preserving both the algebra and the coalgebra structures, that is, with morphism of bisemigroups.  Let ${}_{c}^{\dagger}\mathbf{FrobBisem}(\mathbb{Hilb})$ be its full subcategory spanned by the Frobenius algebras with an isometric comultiplication. 

Being a morphism of bisemigroups is rather restrictive as show the following result and remark below. 
\begin{proposition}Let $f\in {}_{\mathsf{partiso},c}\mathbf{FrobBisem}(\mathbb{Hilb})((H,\mu),(K,\gamma))$ where $\mu^{\dagger}$ is an isometry (that is, $(H,\mu)$ is semisimple). Then, $f$ is a partial isometry.
\end{proposition}
\begin{proof}
Since both $f$ and $f^{\dagger}$ are coalgebra maps, $f(G(H,\mu))\subseteq G(K,\gamma)\cup\{\, 0\,\}$ and $f^{\dagger}(G(K,\gamma))\subseteq G(H,\mu)\cup\{\, 0\,\}$. Therefore for each $g\in G(H,\mu)$ and $h\in G(K,\gamma)$, 
$f(g)=h\Leftrightarrow \langle f(g),h\rangle=1\Leftrightarrow \langle g,f^{\dagger}(h)\rangle=1\Leftrightarrow g=f^{\dagger}(h)$. Consequently, for each $h\in G(K,\gamma)$, $f(f^{\dagger}(h))=h$ when $f^{\dagger}(h)\not=0$ and $f(f^{\dagger}(h))=0$ when $f^{\dagger}(h)=0$ and in any case $f^{\dagger}(f(f^{\dagger}(h)))=f^{\dagger}(h)$. Let $u\in H$. Then, $f^{\dagger}(u)=\sum_{h\in G(K,\gamma)}\langle u,h\rangle f^{\dagger}(h)=\sum_{h\in G(K,\gamma)}\langle u,h\rangle f^{\dagger}(f(f^{\dagger}(h)))=f^{\dagger}(f(f^{\dagger}(u)))$. Then, $f^{\dagger}$ is a partial isometry and so is also $f$.
\end{proof}

\begin{remark}\label{rem:pre_partial_injection}
In general for $f\colon (Y,y_0,\beta)\to (X,x_0,\alpha)$, $\ell^2_{\bullet}(f)$ is not a coalgebra morphism. Indeed if it was the case, then for each $x\in X\setminus\{\, x_0\,\}$, $\frac{\delta_x\circ f}{\alpha(x)}=\ell^2(f)(\frac{\delta_x}{\alpha(x)})=0$ or $\ell^2(f)(\frac{\delta_x}{\alpha(x)})\in G(\ell^2_{\bullet}(Y,y_0,\beta))$, that is, $\frac{\delta_x\circ f}{\alpha(x)}=\frac{\delta_y}{\beta(y)}$ for some $y\in Y\setminus\{\, y_0\,\}$. Equivalently, for each $x\in X\setminus\{\, x_0\,\}$, $f^{-1}(\{\, x\,\})=\emptyset$ or there exists $y\in Y\setminus\{\, y_0\,\}$ such that $f(y)=x$ and $\alpha(f(y))=\beta(y)$, and for each $y'\not=y$, $y'\in Y\setminus\{\, y_0\,\}$, $f(y')\not=x$. In particular, $|f^{-1}(\{\, x\,\})|\leq 1$ for each $x\in X\setminus \{\, x_0\,\}$. 
\end{remark}

Let $\mathbf{PInj}_{\bullet}$ be the category of {\em partial injections}, that is, the objects are pointed sets and morphism $(X,x_0)\xrightarrow{f}(Y,y_0)$ are base-point preserving maps such that for all $y\in Y\setminus\{\, y_0\,\}$, $|f^{-1}(\{\, y\,\})|\leq 1$. 
\begin{remark}
While not being identically presented our category $\mathbf{PInj}_{\bullet}$ is isomorphic to the category $\mathbf{PInj}$  from in~\cite{HeunenL2}.
\end{remark}
$\mathbf{PInj}_{\bullet}$ embeds into $\mathbf{WSet}_{\bullet}$ under $E(X,x_0):=(X,x_0,\mathbf{1})$ (where $\mathbf{1}(x)=1$, $x\not=x_0$) and $E(f):=f$ as for each $y\not=y_0$, (1) $|f^{-1}(\{\, y\,\})|\leq 1$, and (2) $\sum_{x\in f^{-1}(\{\, y\,\})}\mathbf{1}(x)\leq 1=\mathbf{1}(y)$. So one may consider the functor $\mathbf{PInj}_{\bullet}^{\mathsf{op}}\xrightarrow{\ell^2_{\bullet}\circ E}{}_c\mathbf{FrobSem}(\mathbb{Hilb})$. Of course it factors  through ${}_{c}^{\dagger}\mathbf{FrobSem}(\mathbb{Hilb})\hookrightarrow {}_{\mathsf{partiso},c}\mathbf{FrobSem}(\mathbb{Hilb})\hookrightarrow {}_c\mathbf{FrobSem}(\mathbb{Hilb})$. But actually, for each partial injection $(X,x_0)\xrightarrow{f}(Y,y_0)$, $\ell^2_{\bullet}$ is even a coalgebra map as  $\mu_X^{\dagger}(\ell^2_{\bullet}(f)(u))=\sum_{x\not=x_0}u(f(x))\delta_x\otimes \delta_x$ and $(\ell^2(f)\hat{\otimes}_2\ell^2(f))(\delta_Y^{\dagger}(u))=(\ell^2(f)\hat{\otimes}_2\ell^2(f))(\sum_{y\not=y_0}u(y)\delta_y\otimes\delta_y)=
\sum_{y\not=y_0}u(y)(\delta_y\circ f)\otimes (\delta_y\circ f)=\sum_{x\not=x_0}u(f(x))\delta_x\otimes\delta_x$ for each $u\in \ell^2_{\bullet}(Y,y_0,\mathbf{1})$. This provides a functor $\ell^2_{\bullet}\colon \mathbf{PInj}_{\bullet}\to {}_{\mathsf{partiso},c}\mathbf{FrobBisem}(\mathbb{Hilb})$ together with its co-restriction $\ell^2_{\bullet}\colon \mathbf{PInj}_{\bullet}\to {}_{c}^{\dagger}\mathbf{FrobBisem}(\mathbb{Hilb})$.

In the opposite direction one has a functor $G_{\bullet}\colon {}_{\mathsf{partiso},c}\mathbf{FrobBisem}(\mathbb{Hilb})\to \mathbf{PInj}_{\bullet}^{\mathsf{op}}$ given as follows: $G_{\bullet}(H,\mu):=(G(H,\mu)\cup\{\, 0\,\},0)$ and given a morphism of bisemigroups $f\colon ((H,\mu),\mu^{\dagger})\to ((K,\gamma),\gamma^{\dagger})$, $G_{\bullet}(f)\colon G_{\bullet}(K,\gamma)\to G_{\bullet}(H,\mu)$ is the restriction of $f^{\dagger}$. $G_\bullet(f)\colon G_{\bullet}(K,\gamma)\to G_{\bullet}(H,\mu)$ is indeed a partial injection  because for each $g,h\in (f^{\dagger})^{-1}(G(H,\mu))\cap G(K,\gamma)$, $g\not=h$, $0=f^{\dagger}(gh)=f^{\dagger}(g)f^{\dagger}(h)$ (since $f^{\dagger}$ is also a semigroup map), so that $f^{\dagger}(g)\not=f^{\dagger}(h)$ as $f^{\dagger}(g),f^{\dagger}(h)\in G(H,\mu)$.

\begin{proposition}
One has an adjunction $G_\bullet\dashv \ell_\bullet^2\colon  {}_{\mathsf{partiso},c}\mathbf{FrobBisem}(\mathbb{Hilb})\to \mathbf{PInj}_{\bullet}^{\mathsf{op}}$ which restricts to an equivalence of categories ${}_{c}^{\dagger}\mathbf{FrobBisem}(\mathbb{Hilb})\simeq\mathbf{PInj}_{\bullet}^{\mathsf{op}}$.
\end{proposition}
\begin{proof}
For each pointed set $(X,x_0)$, $\epsilon_{(X,x_0)}\colon (X,x_0)\to G_\bullet(\ell^2_{\bullet})=(\{\,\delta_x\colon x\not=x_0\,\}\cup\{\, 0\,\},0)$ given by $\epsilon_{(X,x_0)}(x):=\delta_x$, $\epsilon_{(X,x_0)}(x_0):=0$, is an isomorphism and so in particular a partial injection. 

Let $(H,\mu)$ be a Hilbertian Frobenius algebra with  $\mu^{\dagger}$ a partial isometry. Let $(X,x_0)\xrightarrow{f}G_\bullet(H,\mu)$ be a partial injection. Let us define $(H,\mu)\xrightarrow{f^{\sharp}}\ell^2_{\bullet}(X,x_0)$ by $f^{\sharp}:=\ell_\bullet^{2}(f)\circ \Phi_{(H,\mu)}\circ \pi_{J(H,\mu)^{\perp}}$ where $\Phi_{(H,\mu)}\colon J(H,\mu)^{\perp}\simeq \ell^2_{\bullet}(G_{\bullet}(H,\mu))$ is the unitary isomorphism $\Phi_{(H,\mu)}(g):=\delta_g$, $g\in G(H,\mu)$.  $\Phi_{(H,\mu)}(gh)=\Phi(\delta_{g,h}g)=\delta_{g,h}\delta_g=\Phi(\delta_g)\Phi(\delta_h)$ so $\Phi_{(H,\mu)}$ is a semigroup map and in a same way one sees that $\Phi^{-1}_{(H,\mu)}=\Phi^{\dagger}_{(H,\mu)}$ is also a semigroup map so that $\Phi_{(H,\mu)}$ is a cosemigroup map. By the way $f^{\sharp}$ is a semigroup and a cosemigroup map because so are $\ell^2_{\bullet}(f)$, $\Phi_{(H,\mu)}$ and $\pi_{J(H,\mu)^{\perp}}$. (Indeed, $\pi_{J(H,\mu)^{\perp}}=i_{J(H,\mu)^{\perp}}^{\dagger}$ and $J(H,\mu)^{\perp}$ is both a semigroup and a cosemigroup map.) Whence $f^{\sharp}\in {}_{\mathsf{partiso},c}\mathbf{FrobBisem}(\mathbb{Hilb})((H,\mu),\ell^2_{\bullet}(X,x_0))$. 

Let $x\in X\setminus\{\, x_0\,\}$ and let $g\in G(H,\mu)$. One has $\langle (f^{\sharp})^{\dagger}(\delta_x),g\rangle=\langle \delta_x,f^\sharp(g)\rangle=\langle \delta_x,\delta_g\circ f\rangle$. Whence $(f^{\sharp})^{\dagger}(\delta_x)=g$ if, and only if, $\delta_g\circ f=\delta_x$ if, and only if, $f(x)=g$. In other words $G_\bullet(f^{\dagger})\circ \epsilon_{(X,x_0)}=f$. If there is another algebra/coalgebra map $(H,\mu)\xrightarrow{f'}\ell^2_{\bullet}(X,x_0)$ such that $G_{\bullet}(f')\circ \epsilon_{(X,x_0)}=f=G_{\bullet}(f^{\sharp})\circ \epsilon_{(X,x_0)}$, then $(f')^{\dagger}=(f^{\sharp})^{\dagger}$ on $G(\ell_\bullet^{2}(X,x_0))$ it follows that $(f')^{\dagger}=(f^{\sharp})^{\dagger}$  on $\ell_\bullet^{2}(X,x_0)$ and so $f'=f^{\sharp}$. 

It remains to prove the equivalence of categories. The counit $\epsilon=(\epsilon_{(X,x_0)})_{(X,x_0)}$ is already an isomorphism so it suffices to check the conditions under which the unit $(\eta_{(H,\mu)})_{(H,\mu)}$, $\eta_{(H,\mu)}=\Phi_{(H,\mu)}\circ \pi_{J(H,\mu)^{\perp}}$, is an isomorphism. This is clear that it will be so if, and only if, $(H,\mu)$ is semisimple, that is, if, and only if, $(H,\mu)$ is a ${}_c^{\dagger}\mathbf{FrobBisem}(\mathbb{Hilb})$-object. 
\end{proof}

\subsection{Ambidextrous morphisms: Algebra-and-coalgebra maps}\label{sec:ambi}

Even in the non partial isometric case, that is, even if $((H,\mu),\mu^{\dagger})$ is not a bisemigroup,  it is tempting  to see what happens when morphisms of Frobenius semigroups are chosen as those bounded linear maps which are both semigroup and cosemigroup morphisms. 
Let ${}_c\mathbf{Frob}(\mathbb{Hilb})_{\mathsf{ambi}}$ be the corresponding non full subcategory of ${}_c\mathbf{FrobSem}(\mathbb{Hilb})$. (One drops the suffix ``$\mathbf{Sem}$'' to emphasize the fact that both the semigroup and the cosemigroup structures are of equal importance.) Observe that ${}_{\mathsf{partiso},c}\mathbf{FrobBisem}(\mathbb{Hilb})$ is a full subcategory of ${}_c\mathbf{Frob}(\mathbb{Hilb})_{\mathsf{ambi}}$. In view to Remark~\ref{rem:pre_partial_injection} one introduces the category $\mathbf{Pinj}_{\bullet,\mathsf{w}}$ with
\begin{enumerate}
\item objects the weighted pointed  sets as in $\mathbf{WSet}_{\bullet}$,
\item arrows $(X,x_0,\alpha)\xrightarrow{f}(Y,y_0,\beta)$ the partial injections $(X,x_0)\xrightarrow{f}(Y,y_0)$ such that for each $x\in f^{-1}(Y\setminus\{\, y_0\,\})$, $\alpha(x)=\beta(f(x))$. 
\end{enumerate}
It is  clear that $\mathbf{Pinj}_{\bullet,\mathsf{w}}$ embeds (while not fully) into $\mathbf{WSet}_\bullet$.

\begin{proposition}\label{prop:equiv_partial_injection}
One has an adjunction $Min_{\bullet}\dashv \ell^2_{\bullet}\colon {}_c\mathbf{Frob}(\mathbb{Hilb})_{\mathsf{ambi}}\to \mathbf{Pinj}_{\bullet,\mathsf{w}}^{\mathsf{op}}$ which restricts to an equivalence  ${}_{\mathsf{semisimple},c}\mathbf{Frob}(\mathbb{Hilb})_{\mathsf{ambi}}\simeq \mathbf{Pinj}_{\bullet,\mathsf{w}}^{\mathsf{op}}$.
\end{proposition}
\begin{proof}
The functor $\ell^2_{\bullet}$ occurring in the statement of the proposition is the only one which makes commute the following diagram commutes.
\begin{equation}
\xymatrix{ 
\mathbf{WSet}_{\bullet}^{\mathsf{op}}\ar[r]^{\ell^2_{\bullet}} & {}_c\mathbf{FrobSem}(\mathbb{Hilb})\\
\ar@{^{(}->}[u]\mathbf{Pinj}_{\bullet,\mathsf{w}}^{\mathsf{op}} \ar[r]_{\ell^2_{\bullet}}& {}_c\mathbf{Frob}(\mathbb{Hilb})_{\mathsf{ambi}}\ar@{^{(}->}[u]
}
\end{equation}

Conversely, $Min_{\bullet}$ as defined in the statement is the only functor which makes commute the diagram below.
\begin{equation}
\xymatrix{
{}_c\mathbf{FrobSem}(\mathbb{Hilb})\ar[r]^-{Min_{\bullet}} &\mathbf{WSet}_{\bullet}^{\mathsf{op}} \\
\ar@{^{(}->}[u]{}_c\mathbf{Frob}(\mathbb{Hilb})_{\mathsf{ambi}} \ar[r]_-{Min_{\bullet}}&\mathbf{Pinj}_{\bullet,\mathsf{w}}^{\mathsf{op}} \ar@{^{(}->}[u]
}
\end{equation}

One now uses the notations from the proof of Theorem~\ref{thm:equivalence_weighted_sets_and_semisimple_Frob}. One notes immediately that since $\epsilon_{(X,x_0)}$ is an isomorphism in $\mathbf{WSet}_{\bullet}$, it is also an isomorphism in $\mathbf{Pinj}_{\bullet,\mathsf{w}}$.

The unit of the adjunction from Theorem~\ref{thm:equivalence_weighted_sets_and_semisimple_Frob} is given by $\Phi^{-1}_{(H,\mu)}\circ \pi_{J(H,\mu)^{\perp}}$ which is actually both a semigroup map (already known) and a cosemigroup map. To see this one first observes that $i_{J(H,\mu)^{\perp}}$ is a morphism of semigroups since $J(H,\mu)^{\perp}$ is a subalgebra, and thus $\pi_{J(H,\mu)^{\perp}}=i^{\dagger}_{J(H,\mu)^{\perp}}$ is a cosemigroup map. $\Phi^{-1}_{(H,\mu)}\colon J(H,\mu)^{\perp}\to \ell^2_{\bullet}(Min_{\bullet}(H,\mu))$ is also a cosemigroup map as follows directly from $\Phi^{-1}_{(H,\mu)}(g)=\|g\|\Phi^{-1}_{(H,\mu)}(\frac{g}{\|g\|})=\|g\|\Phi^{-1}(\frac{\mycirc{e}(\mathbb{C}g)}{\|\mycirc{e}(\mathbb{C}g)\|})=\|g\|^2\delta_{\mathbb{C}g}=\frac{\delta_{\mathbb{C}g}}{w_{(H,\mu)}(\mathbb{C}g)}$ for each $g\in G(H,\mu)$ and the fact that $J(H,\mu)^{\perp}$ is semisimple (Lemma~\ref{lem:cosemigroupmorphisms}). Note that $\Phi^{-1}_{(H,\mu)}$ being bijective and both a semigroup and a cosemigroup map, it is an isomorphism of ${}_c\mathbf{Frob}(\mathbb{Hilb})_{\mathsf{ambi}}$ (by Lemma~\ref{lem:semigroupiso}). 

It then follows that the adjunction from Theorem~\ref{thm:equivalence_weighted_sets_and_semisimple_Frob} provides the desired adjunction. It remains to obtain the equivalence of categories by restriction. But the unit of the adjunction $\Phi^{-1}_{(H,\mu)}\circ \pi_{J(H,\mu)^{\perp}}$ is an isomorphism (both of semigroups and cosemigroups) if, and only if, so is $\pi_{J(H,\mu)^{\perp}}$ if, and only if, $\pi_{J(H,\mu)^{\perp}}$ is a bijection (by Lemma~\ref{lem:semigroupiso}) if, and only if, $(H,\mu)$ is semisimple. 
\end{proof}

\subsection{Proper morphisms}\label{sec:proper}

Call a  $\mathbf{WSet}_{\bullet}$-morphism $(X,x_0,\alpha)\xrightarrow{f}(Y,y_0,\beta)$ {\em proper} when $f^{-1}(\{\, y_0\,\})=\{\, x_0\,\}$ or equivalently $f(X\setminus \{\, x_0\,\})\subseteq Y\setminus\{\, y_0\,\}$. It is clear that every $\mathbf{WSet}_{\bullet}$-isomorphism is proper. 

Let $\mathbf{WSet}$ be the category with
\begin{enumerate}
\item objects the pairs $(X,\alpha\colon X\to [C,+\infty[)$, $C>0$. 
\item arrows $(X,\alpha)\xrightarrow{f}(Y,\beta)$ the maps $X\xrightarrow{f}Y$ such that 
\begin{enumerate}
\item $|f^{-1}(\{\,y\,\})|<+\infty$ for each $y\in Y$,
\item there exists  $M_f\geq 0$ such that for each $y\in Y$, $\sum_{x\in f^{-1}(\{\, y\,\})}\alpha(x)\leq M_f\beta(y)$.
\end{enumerate}
\end{enumerate}

For a set $X$, let $X^+:=X+1$, where $1:=\{\, 0\,\}$ and $+$ denotes the disjoint union. Let $X\xrightarrow{f}Y$ be a map. Define $X^+\xrightarrow{f^+}Y^+$ by $f^+:=f+id_1$, that is, roughly speaking, $f^+(x)=f(x)$, $x\in X$, $f^+(0):=0$. This provides a functor $\mathbf{WSet}\xrightarrow{(-)^+}\mathbf{WSet}_{\bullet}$ which acts on objects as $(X,\alpha)^+:=(X^+,0,\alpha)$, and which is injective on objects and faithful. Under this functor $\mathbf{WSet}$ is clearly equivalent to the (non full)  subcategory $\mathbf{WSet}_{\bullet,\mathsf{proper}}$ of $\mathbf{WSet}_{\bullet}$ whose objects are those of $\mathbf{WSet}_{\bullet}$ but with proper morphisms between them.

Let $(H,\mu)\xrightarrow{f}(K,\gamma)$ be a semigroup morphism  between Hilbertian Frobenius semigroups. It is said to be {\em proper} when $ran(f)$ is not included in any maximal modular ideals of $(K,\gamma)$ or alternatively for each $y\in G(K,\gamma)$, there exists $u\in H$ such that $\langle f(u),y\rangle\not=0$.  Properness for $f$ implies that $f^{\dagger}(y)\not=0$ for each $y\in G(K,\gamma)$, and since $f^{\dagger}(G(K,\gamma))\subseteq G(H,\mu)\cup\{\, 0\,\}$, it follows that actually $f^{\dagger}(G(K,\gamma))\subseteq G(H,\mu)$. Conversely if $f^{\dagger}(G(K,\gamma))\subseteq G(H,\mu)$, then for each $y\in G(K,\gamma)$, $\langle f(f^{\dagger}(y)),y\rangle\not=0$ and thus $f$ is proper. One observes that every semigroup isomorphism is proper.

Let ${}_{c}\mathbf{FrobSem}(\mathbb{Hilb})_{\mathsf{proper}}$ be the category whose objects are Hilbertian Frobenius semigroups and morphisms are the proper semigroup morphisms. As usually let ${}_{\mathsf{semisimple},c}\mathbf{FrobSem}(\mathbb{Hilb})_{\mathsf{proper}}$ be its full subcategory spanned by the semisimple objects. 

Let $(X,\alpha)$ be a $\mathbf{WSet}$-object. Define $\ell^2(X,\alpha):=(\ell^2_{\alpha}(X),\mu_X)$ as given in Section~\ref{sec:weighted_hilb}. Let $(X,\alpha)\xrightarrow{f}(Y,\beta)$ be a $\mathbf{WSet}$-morphism. Define $\ell^2(f)\colon (\ell^2_\beta(Y),\mu_Y)\to (\ell^2_\alpha(X),\mu_X)$ by $\ell^2(f)(u):=u\circ f$. 

\begin{lemma}\label{lem:ell_2_f_is_proper}
$\ell^2(f)$ is a proper morphism.
\end{lemma}

One obtains a functor $\ell^2\colon \mathbf{WSet}\to {}_c\mathbf{FrobSem}(\mathbb{Hilb})_{\mathsf{proper}}$ and a co-restriction still denoted $\ell^2$ from $\mathbf{WSet}$ to ${}_{\mathsf{semisimple},c}\mathbf{FrobSem}(\mathbb{Hilb})_{\mathsf{proper}}$.

Now let $f\colon (H,\mu)\to (K,\gamma)$ be a proper morphism between Hilbertian Frobenius semigroups. As $f^{\dagger}(G(K,\gamma))\subseteq G(H,\mu)$, one defines a map $\ell\colon Min(K,\gamma)\to Min(H,\mu)$ by the relation $\ell(J):=\mathbb{C}f^{\dagger}(\mycirc{g}(J))$, $J\in Min(K,\gamma)$ as in Section~\ref{sec:minimal_ideal_functor}. Consequently, $\ell_0\in \mathbf{WSet}_{\bullet}(Min_\bullet(K,\gamma),Min_{\bullet}(H,\mu))$, where $\ell_0$ is the extension of $\ell$ obtained by setting $\ell_0(0):=0$. As $\ell_0(J)=\ell(J)\not=0$, $J\in Min(K,\gamma)$, it follows that actually $\ell\in \mathbf{WSet}((Min(K,\gamma),w_{(K,\gamma)}),(Min(H,\mu),w_{(H,\mu)}))$, and from that one has a functor $Min\colon {}_c\mathbf{FrobSem}(\mathbb{Hilb})_{\mathsf{proper}}\to \mathbf{WSet}$. 


Let $\mathbf{WFinSet}$ be the full subcategory of $\mathbf{WSet}$ spanned by the {\em finite} weighted sets. 
\begin{proposition}
One has an adjunction $Min\dashv \ell^2\colon {}_c\mathbf{FrobSem}(\mathbb{Hilb})_{\mathsf{proper}}\to \mathbf{WSet}^{\mathsf{op}}$ that restricts to an equivalence ${}_{\mathsf{semisimple},c}\mathbf{FrobSem}(\mathbb{Hilb})_{\mathsf{proper}}\simeq\mathbf{WSet}^{\mathsf{op}}$. In particular, ${}_{\mathsf{semisimple},c}\mathbf{FrobSem}(\mathbb{FdHilb})_{\mathsf{proper}}\simeq\mathbf{WFinSet}^{\mathsf{op}}\simeq \mathbf{FinSet}^{\mathsf{op}}$, where $\mathbf{FinSet}$ is the category of finite sets with all maps between them. 
\end{proposition}

\begin{proof}
The adjunction $Min_\bullet\dashv \ell^2_\bullet\colon {}_c\mathbf{FrobSem}(\mathbb{Hilb})\to\mathbf{WSet}_{\bullet}$ from Theorem~\ref{thm:equivalence_weighted_sets_and_semisimple_Frob}, clearly restricts to an adjunction $Min_{\bullet}\dashv\ell^2_\bullet\colon {}_{c}\mathbf{FrobSem}(\mathbb{Hilb})_{\mathsf{proper}}\to\mathbf{WSet}_{\bullet,\mathsf{proper}}$  and thus to the adjunction $Min\dashv \ell^2\colon  {}_c\mathbf{FrobSem}(\mathbb{Hilb})_{\mathsf{proper}}\to \mathbf{WSet}$. The counit of this adjunction is still a natural isomorphism  while the component $\Phi^{-1}_{(H,\mu)}\circ \pi_{J(H,\mu)^{\dagger}}$ at $(H,\mu)$ of unit is an isomorphism if, and only if, $\pi_{J(H,\mu)^{\dagger}}$ (using the notations from the proof of Theorem~\ref{thm:equivalence_weighted_sets_and_semisimple_Frob}) is an isomorphism if, and only if, $(H,\mu)$ is semisimple. So the required equivalence is proved.

The last statement is obvious because $\mathbf{FinSet}$ is clearly equivalent to $\mathbf{WFinSet}$. 
\end{proof}

The second statement of the following corollary corresponds to~\cite[Corollary~7.2, p.~566]{Coecke} as the categories ${}_c\mathbf{FrobComon}(\mathbb{FdHilb})$ and ${}_c\mathbf{FrobMon}(\mathbb{FdHilb})^{\mathsf{op}}$ are isomorphic under the dagger functor (cf.~Section~\ref{sec:hilbertian_frob_alg}). 
\begin{corollary}\label{cor:fdhilb_and_finset}
One has the equivalences of categories ${}_{c}\mathbf{FrobSem}(\mathbb{Hilb})_{\mathsf{proper}}\simeq \mathbf{WSet}^{\mathsf{op}}\times \mathbf{Hilb}$ and ${}_{\mathsf{semisimple},c}\mathbf{FrobSem}(\mathbb{FdHilb})_{\mathsf{proper}}\simeq {}_c\mathbf{FrobMon}(\mathbb{FdHilb})\simeq \mathbf{FinSet}^{\mathsf{op}}$. 
\end{corollary}

\begin{proof}
As the first statement is clear, one only needs to prove the second, and it is clear that one only needs to prove that the categories ${}_c\mathbf{FrobMon}(\mathbb{FdHilb})$ and  ${}_{\mathsf{semisimple},c}\mathbf{FrobSem}(\mathbb{FdHilb})_{\mathsf{proper}}$ are equivalent. According to Corollary~\ref{cor:semisimplicityforFrob1}, the obvious forgetful functor ${}_c\mathbf{FrobMon}(\mathbb{FdHilb})\xrightarrow{|-|}{}_{c}\mathbf{FrobSem}(\mathbb{FdHilb})$ factors through the embedding ${}_{\mathsf{semisimple},c}\mathbf{FrobSem}(\mathbb{FdHilb})\hookrightarrow {}_c\mathbf{FrobSem}(\mathbb{FdHilb})$. One thus only needs to check that $|f|$ is actually proper for each monoid morphism $f$ and that the co-restricted functor ${}_c\mathbf{FrobMon}(\mathbb{FdHilb})\xrightarrow{|-|}{}_{\mathsf{semisimple},c}\mathbf{FrobSem}(\mathbb{FdHilb})_{\mathsf{proper}}$ is full. 

Given a finite-dimensional Hilbertian Frobenius monoid $(H,\mu,\eta)$, by a direct inspection $\eta(1)=\sum_{g\in G(H,\mu)}\frac{g}{\|g\|^2}$. The corresponding counit $\eta^{\dagger}\colon H\to\mathbb{C}$ thus is given by $\eta^{\dagger}(u)=\sum_{g\in G(H,\mu)}\langle u,\frac{g}{\|g\|^2}\rangle$. In particular for each $g\in G(H,\mu)$, $\eta^{\dagger}(g)=1$. 

Let $(H,\mu,\eta)\xrightarrow{f}(H',\mu',\eta')$ be a monoid morphism between finite-dimensional Frobenius monoids. Then, $f^{\dagger}(G(H',\mu'))\subseteq G(H,\mu)\cup\{\, 0\,\}$ and since $f^{\dagger}$ is compatible with the counits, for each $h\in G(H',\mu')$, $\eta^{\dagger}(f^{\dagger}(h))=(\eta')^{\dagger}(h)=1$. As a consequence $f$ is proper. 

Let $(H,\mu,\eta),(H',\mu',\eta')$ be finite-dimensional Frobenius monoids and let $(H,\mu)\xrightarrow{f}(H',\mu')$ be a proper morphism. Then for each $h\in G(H',\mu')$, there is one $g_h\in G(H,\mu)$ such that $f^{\dagger}(h)=g_h$ as $f^{\dagger}(G(H',\mu'))\subseteq G(H,\mu)$. Therefore, for each $h\in G(H',\mu')$, $\eta^{\dagger}(f^{\dagger}(h))=\eta^{\dagger}(g_h)=1=(\eta')^{\dagger}(h)$. Therefore $(\eta')^{\dagger}$ and $\eta^{\dagger}\circ f^{\dagger}$ are equal on $G(H',\mu')$ which spans $H'$, therefore they are equal on the whole $H'$, and as a consequence $f^{\dagger}\colon (H',\mu',\eta')\to (H,\mu,\eta)$ is a comonoid morphism, so that $(H,\mu,\eta)\xrightarrow{f}(H',\mu',\eta')$ is a morphism of monoids. This proves that ${}_c\mathbf{FrobMon}(\mathbb{FdHilb})\xrightarrow{|-|}{}_{\mathsf{semisimple},c}\mathbf{FrobSem}(\mathbb{FdHilb})_{\mathsf{proper}}$ is full. 
\end{proof}

Let ${}_{\mathsf{semisimple},c}\mathbf{Frob}(\mathbb{FdHilb})_{\mathsf{ambi}-\mathsf{proper}}$ be the category of finite-dimensional semisimple Frobenius semigroups with morphisms $f$ such that $(H,\mu)\xrightarrow{f}(K,\gamma)$ is a proper semigroup map and $(K,\gamma)\xrightarrow{f^{\dagger}}(H,\mu)$ is also a proper semigroup map (in particular, $(H,\mu^{\dagger})\xrightarrow{f}(K,\gamma^{\dagger})$ is a cosemigroup map). Let us also define ${}_{1,c}\mathbf{Frob}(\mathbb{FdHilb})_{\mathsf{ambi}}$ to be the category of finite-dimensional Frobenius monoids whose morphisms preserve both the monoid and the comonoid structures (the index ``$1$'' recalls that the structures are both unital and counital).
\begin{remark}
Observe that when $(H,\mu,\eta)$ is a finite-dimensional Hilbertian Frobenius monoid with an isometric comultiplication, then $((H,\mu,\eta),\mu^{\dagger})$ is a cocommutative cosemigroup object in ${}_c\mathbf{Mon}(\mathbb{FdHilb})$ by Proposition~\ref{prop:bisemigroup} and the fact that $\mu^{\dagger}(\eta(1))=\sum_{g\in G(H,\mu)}g\otimes g$. But it is not a bimonoid unless if $\dim_\mathbb{C}H\leq 1$, because $\eta^{\dagger}(\eta(1))=|G(H,\mu)|$.  
\end{remark}
Let $\mathbf{FinSet}_{\mathsf{bij},\mathsf{w}}$ be the category of finite weighted sets and bijections between them preserving the weight functions, that is, $(X,\alpha)\xrightarrow{f}(Y,\beta)$ is given as a bijection $X\xrightarrow{f}Y$ such that $\alpha(x)=\beta(f(x))$ for each $x\in X$. $\mathbf{FinSet}_{\mathsf{bij},\mathsf{w}}$ clearly embeds into $\mathbf{WSet}$. 
\begin{corollary}
${}_{\mathsf{semisimple},c}\mathbf{Frob}(\mathbb{FdHilb})_{\mathsf{ambi-proper}}\simeq {}_{1,c}\mathbf{Frob}(\mathbb{FdHilb})_{\mathsf{ambi}}\simeq \mathbf{FinSet}_{\mathsf{bij},\mathsf{w}}$. Consequently, ${}_{\mathsf{semisimple},c}\mathbf{Frob}(\mathbb{FdHilb})_{\mathsf{ambi-proper}}={}_{\mathsf{semisimple},c}\mathbf{FrobSem}(\mathbb{FdHilb})_{\mathsf{unitary}}$, the category of finite-dimensional semisimple Frobenius semigroups with unitary isomorphisms of semigroups. Moreover every isomorphism of semigroups between finite-dimensional semisimple Frobenius semigroups is a unitary transformation. 
\end{corollary}
\begin{proof}
This is~\cite[Corollary~7.1, p.~365]{Coecke} but it can be deduced from Corollary~\ref{cor:fdhilb_and_finset}. Indeed, one knows that $(H,\mu)\xrightarrow{f}(H',\mu')$ is a proper semigroup map between finite-dimensional semisimple Frobenius semigroups if, and only if, $(H,\mu,\eta)\xrightarrow{f}(H',\mu',\eta')$ is a monoid morphism (using the notations from the proof of Corollary~\ref{cor:fdhilb_and_finset}). So $f\in{}_{\mathsf{semisimple},c}\mathbf{Frob}(\mathbb{FdHilb})_{\mathsf{ambi-proper}}((H,\mu),(H',\mu'))$ if, and only if, $(H,\mu,\eta)\xrightarrow{f}(H',\mu',\eta')$ and $(H',\mu',\eta')\xrightarrow{f^{\dagger}}(H,\mu,\eta)$ are monoid morphisms  if, and only if, $f\in {}_{1,c}\mathbf{Frob}(\mathbb{FdHilb})_{\mathsf{ambi}}((H,\mu,\eta),(H',\mu',\eta'))$. 

Let $f\in {}_{\mathsf{semisimple},c}\mathbf{Frob}(\mathbb{FdHilb})_{\mathsf{ambi-proper}}((H,\mu),(K,\gamma))$. 
Now $f(G(H,\mu))\subseteq G(K,\gamma)$ since $f$ is a cosemigroup map and since $f^{\dagger}$ is proper, and for each $g,h\in G(H,\mu)$, $g\not=h$, $0=f(gh)=f(g)f(h)$ so that $f$ is an injection from $G(H,\mu)$ into $G(K,\gamma)$.  $f^{\dagger}$ is an injection from $G(K,\gamma)$ into $G(H,\mu)$ for similar reasons. From the relation $\langle f(g),h\rangle=\langle g,f^{\dagger}(h)\rangle$, $g\in G(H,\mu)$, $h\in G(K,\gamma)$ it follows that $f\circ f^{\dagger}=id_{G(K,\gamma)}$ and $f^{\dagger}\circ f=id_{G(H,\mu)}$, that is, $G(H,\mu)\xrightarrow{f}G(K,\gamma)$ is a bijection with inverse $f^{\dagger}$.  Now by functoriality $Min(f)\colon Min(K,\gamma)\to Min(H,\mu)$ is a bijection with inverse $Min(f^{\dagger})\colon Min(H,\mu)\to Min(K,\gamma)$. Since ${}_{\mathsf{semisimple},c}\mathbf{Frob}(\mathbb{FdHilb})_{\mathsf{ambi-proper}}\hookrightarrow{}_{\mathsf{semisimple},c}\mathbf{FrobSem}(\mathbb{FdHilb})_{\mathsf{proper}}$ it follows from Proposition~\ref{prop:equiv_partial_injection} that $Min_{\bullet}(f)\in \mathbf{PInj}_{\bullet,\mathsf{w}}(Min_{\bullet}(H,\mu),Min_{\bullet}(K,\gamma))$. This implies that $w_{(H,\mu)}(Min(f)(J))=w_{(K,\gamma)}(J)$ for each $J\in Min(K,\gamma)$, that is, $Min(f)\in \mathbf{FinSet}_{\mathsf{bij},\mathsf{w}}(Min(K,\gamma),Min(H,\mu))$. 

Let $(X,\alpha)\xrightarrow{f}(Y,\beta)$ be a $\mathbf{FinSet}_{\mathsf{bij},\mathsf{w}}$.  Then, of course $\ell^2(f)$ is an isomorphism in ${}_{\mathsf{semisimple},c}\mathbf{FrobSem}(\mathbb{FdHilb})$ from $\ell^2(Y,\beta)$ to $\ell^2(X,\alpha)$, and thus it is a proper morphism. Now for $u\in \ell^2(X,\alpha,\mathbf{1})$ and $v\in \ell^2(Y,\beta,\mathbf{1})$, $\langle \ell^2(f^{-1})(u),v\rangle_\beta=\langle u\circ f^{-1},v\rangle_\beta=\sum_y\beta(y)u(f^{-1}(y))\overline{v(y)}=\sum_x\beta(f(x))u(x)\overline{v(f(x))}=\sum_{x}\alpha(x)u(x)\overline{v(f(x))}=\langle u,v\circ f\rangle_\alpha=\langle u,\ell^2(f)(v)\rangle_\alpha$, that is, 
$\ell^2(f^{-1})=\ell^2(f)^{\dagger}$. Now $\ell^2(f)(\frac{\delta_y}{\beta(y)})=\frac{1}{\beta(y)}\delta_y\circ f=\frac{1}{\beta(y)}\delta_{f^{-1}(y)}=\frac{1}{\alpha(f^{-1}(y))}\delta_{f^{-1}(y)}$. Therefore $\ell^2(f)(G(\ell^2(Y,\beta)))\subseteq G(\ell^2(X,\alpha))$. By Lemma~\ref{lem:cosemigroupmorphisms}, $\ell^2(f)$ is also a morphism of cosemigroups. Thus it is a (unitary) isomorphism of cosemigroups (in particular, it is proper). The required equivalence follows now easily. 

Moreover the above also shows that ${}_{\mathsf{semisimple},c}\mathbf{Frob}(\mathbb{FdHilb})_{\mathsf{ambi-proper}}$ is equal to the groupoid  ${}_{\mathsf{semisimple},c}\mathbf{FrobSem}(\mathbb{FdHilb})_{\mathsf{iso}}$ of finite-dimensional semisimple Frobenius semigroups with isomorphisms of semigroups or  ${}_{\mathsf{semisimple},c}\mathbf{FrobSem}(\mathbb{FdHilb})_{\mathsf{unitary}}$, the groupoid of finite-dimensional semisimple Frobenius semigroups with unitary isomorphisms of semigroups.
\end{proof}

\section{Epilogue: And non-commutativity in all that?}\label{sec:epilogue}

Let $X$ be a non-void set and let $x_0\in X$. Let $m_{x_0}\colon \ell^2(X)\times\ell^2(X)\to\ell^2(X)$ be given by $m_{x_0}(u,v):=v(x_0)u$. It is of course bounded since $\|m_{x_0}(u)\|^2=|v(x_0)|^2\|u\|^2\leq \|v\|^2\|u\|^2$. Then $m_{x_0}$ is a weak Hilbert-Schmidt mapping as $\sum_{x,y}|\langle m_{x_0}(\delta_x,\delta_y),u\rangle|^2=\sum_{x\in X}|\langle \delta_x,u\rangle|^2=\|u\|^2$. Let $\mu_{x_0}\colon \ell^2(X)\hat{\otimes}_2 \ell^2(X)\to\ell^2(X)$ be its unique bounded linear extension. $(\ell^2(X),\mu_{x_0})$ is a Hilbertian semigroup, non-commutative as soon as $X\setminus\{\, x_0\,\}\not=\emptyset$. Since $\langle\mu_{x_0}^{\dagger}(u),\delta_x\otimes \delta_y\rangle=\delta_{y,x_0}u(x)$ for each $x,y\in X$ it follows that $\mu_{x_0}^{\dagger}(u)=u\otimes\delta_{x_0}$. Consequently $\mu_{x_0}(\mu_{x_0}^{\dagger}(u))=u$, $u\in \ell^2(X)$, that is, $\mu_{x_0}^{\dagger}$ is an isometry. 

Moreover for each $u,v\in \ell^2(X)$, $\mu_{x_0}^{\dagger}(\mu_{x_0}(u\otimes v))=v(x_0)\mu_{x_0}^{\dagger}(u)=v(x_0)u\otimes\delta_{x_0}$, 
$(id\hat{\otimes}_2\mu_{x_0})(\alpha((\mu^{\dagger}_{x_0}\hat{\otimes}_2 id)(u\otimes v)))=
(id\hat{\otimes}_2\mu_{x_0})(u\otimes (\delta_{x_0}\otimes v))=v(x_0)u\otimes\delta_{x_0}$ and $(\mu_{x_0}\hat{\otimes}_2 id)(\alpha^{-1}((id\hat{\otimes}_2\mu_{x_0}^{\dagger})(u\otimes v)))=
(\mu_{x_0}\hat{\otimes}_2 id)((u\otimes v)\otimes\delta_{x_0})=v(x_0)u\otimes\delta_{x_0}$. Therefore $(\ell^2(X),\mu_{x_0})$ is Frobenius. To summarize, $(\ell^2(X),\mu_{x_0})$ is a not necessarily commutative special Frobenius Hilbertian semigroup.

It is easily seen that $G(\ell^2(X),\mu_{x_0})=\{\,\delta_{x_0}\,\}$. 

It is also clear that $A(\ell^2(X),\mu_{x_0})=\{\, 0\,\}$ as $u\delta_{x_0}=u$, $u\in \ell^2(X)$ and $\{\, \delta_{x_0}\,\}^{\perp}$ consists entirely of nilpotent elements as if $u(x_0)=0$, then $u^2=0$. Also $E(\ell^2(X),\mu_{x_0})=\{\, u\in \ell^2(X)\colon u(x_0)=1\,\}\cup\{\, 0\,\}$. 

As $\ell^2(X)\to \mathbb{C}$, $u\mapsto u(x_0)$, is a morphism of algebras it follows that its kernel, namely $\{\, \delta_{x_0}\,\}^{\perp}=\{\, u\colon u(x_0)=0\,\}$ is a two-sided maximal modular ideal, with modular unit $\delta_{x_0}$. Let $I$ be a modular right ideal of $(\ell^2(X),m_{x_0})$, that is, $I$ is right ideal with a left-unit $e$, that is, $u-eu\in I$ for each $u\in \ell^2(X)$. As for each $u\in \{\, \delta_{x_0}\,\}^{\perp}$, $u=u-u(x_0)e=u-eu\in I$, it follows that $\{\, \delta_{x_0}\,\}^{\perp}\subseteq I$. By a codimensionality argument it follows that either $I=\{\, \delta_{x_0}\,\}^{\perp}$ or $I=\ell^2(X)$. Consequently $J(\ell^2(X),\mu_{x_0})=\{\,\delta_{x_0}\,\}^{\perp}$ and $(\ell^2(X),\mu_{x_0})$ is not semisimple as soon as $X\setminus\{\,x_0\,\}\not=\emptyset$.  It is not a $H^*$-algebra either for if the annihilator would be equal to the Jacobson radical.

\end{document}